\documentclass[a4paper,11pt]{article}
\usepackage{stmaryrd}
\usepackage{graphicx}
\usepackage{amsfonts}
\usepackage{amsmath,amsthm}
\usepackage{amssymb}
\usepackage{mathrsfs}
\usepackage{indentfirst}
\usepackage{titlesec}
\usepackage{caption2}
\usepackage{color}
\usepackage[normalem]{ulem}
\usepackage{algorithm}
\usepackage{algorithmic}
\usepackage{lineno}
\numberwithin{equation}{section}

\usepackage[english]{babel}
\usepackage{cite}
\usepackage{bbm}

\usepackage[hypertexnames=false]{hyperref}

\providecommand{\bwd}{\leftarrow}
\providecommand{\fwd}{\rightarrow}

\newtheorem{theorem}{Theorem}
\newtheorem*{problem}{Wave Propagation Problem}

\newtheorem{proposition}{Proposition}[section]
\newtheorem{remark}{Remark}
\newtheorem{lemma}{Lemma}[section]
\newtheorem{corollary}{Corollary}[section]

\begin{document}

\title{Wave propagation for 1-dimensional 
reaction-diffusion equations with nonzero random drift}

\author{
Dihang Guan
\thanks{School of Mathematical Sciences, Beijing Normal University, Beijing, 100875, P.~R. China. Email: \texttt{202221130125@mail.bnu.edu.cn}}
\ , 
Hui He
\thanks{School of Mathematical Sciences, Beijing Normal University, Beijing, 100875, P.~R. China. Email: \texttt{hehui@bnu.edu.cn}}
\ ,
Wenqing Hu
\thanks{Department of Mathematics and Statistics, Missouri University of Science and Technology
(formerly University of Missouri, Rolla), Rolla, MO, 65409, USA. Email: \texttt{huwen@mst.edu}}
\ ,
Jiaojiao Yang
\thanks{School of Mathematics and Statistics, Anhui Normal University, Wuhu, 241002, P.~R. China. Email: \texttt{y.jiaojiao1025@ahnu.edu.cn}.}
\ \thanks{Alphabetical Order.}
}

\date{}

\maketitle

\begin{abstract}
We consider the wave propagation for a reaction-diffusion equation on the real line, with a random drift and Fisher-Kolmogorov-Petrovskii-Piscounov (FKPP) type nonlinear reaction. We show that when the average drift is positive, the asymptotic wave fronts propagating to the positive and negative directions are both pushed in the negative direction, leading to the possibility that both wave fronts propagate toward negative infinity. Our proof is based on the Large Deviations Principle for diffusion processes in random environments, as well as an analysis of the Feynman-Kac formula. Such probabilistic arguments also reveal the underlying physical mechanism of the wave fronts formation: the drift acts as an external field that shifts the (quenched) free-energy reference level without altering the intrinsic fluctuation structure of the system.
\end{abstract}



\tableofcontents

\section{Introduction}\label{Sec:Introduction}

We consider the solution $u(t,x)$ of a 1-dimensional reaction-diffusion equation (RDE) with a random drift:

\begin{equation}\label{Eq:RDERandomDrift}
\dfrac{\partial u}{\partial t}=\dfrac{1}{2}\dfrac{\partial^2 u}{\partial x^2}+b(x)\dfrac{\partial u}{\partial x}+ f(u) \ , \ t>0 \ , \ x\in \mathbb{R} \ .
\end{equation}

Here the nonlinearity $f(u)$ is of Fisher-Kolmogorov-Petrovskii-Piscounov (FKPP or KPP in short, see \cite{Fisher}, \cite{KPP}) type: $f\in \mathbf{C}^{(1)}([0,1])$, $f(0)=f(1)$, $0\leq f(u)\leq \beta u$, $\beta = f'(0) >0$ for all $u\in (0,1)$. For example, $f(u)=\beta u(1-u)$, $\beta>0$ satisfies these assumptions, which corresponds to the classical Fisher type nonlinearity, and $\beta>0$ is called the \textit{reaction rate}. The initial data $u(0,x)=u_0(x)$ is non-negative, bounded and compactly supported, with  $\text{supp}u_0\subset (-\delta, \delta)$ for some $\delta>0$. The drift term $b(x)=b(x, \omega): \mathbb{R}\times \Omega \rightarrow \mathbb{R}$ is random, and can be regarded as a random environment. We assume that: 
\begin{itemize}
\item[(1)] $b$ is a stationary random process on $\mathbb{R}$ defined over the probability space $(\Omega, \mathcal{F}, \mathbf{P})$, with $\mathbf{E} b(x)=\mathbf{E} b>0$ \footnote{If $\mathbf{E} b(x)<0$, the results are just following by flipping positive and negative axis.}, where $ \mathbf{E} b(x)=\displaystyle{\int_\Omega b(x;\omega)\mathbf{P}({\rm d}\omega)}$ is the expectation with respect to $\mathbf{P}$;
\item[(2)] $b(\bullet, \omega)$ is almost surely locally Lipschitz continuous and the translation with respect to $x$ generates an ergodic transformation of the space $\Omega$;
\item[(3)] the process $b(x,\omega)$ satisfies
\begin{equation}\label{Eq:AssumptionOnDrift:Boundedness}
{\mathbf P}( |b|\leq B)=1,
\end{equation}
for some deterministic constant $B>0$. 
\end{itemize}

It can be easily shown that equation (\ref{Eq:RDERandomDrift}) admits a unique  solution $u(t,x)$ for all $t>0$ and $x\in \mathbb{R}$. Equation (\ref{Eq:RDERandomDrift}) appears in many application problems such as turbulent combustion (see \cite{RDEAppFlameClavinEtAl}, \cite{RDEAppFlameMajdaEtAl}, \cite{RDEAppFlamePetersBook}, \cite{RDEAppFlameRonney} and references therein), interacting particle systems (see \cite{RDEAppInteractingParticleMullerSowers}, \cite{RDEAppInteractingParticleColonDoering} and references therein) and population biology (see \cite{RDEAppInteractingParticleShenW} and references therein). Moreover, physical advection in reactive or diffusive media is inherently random: turbulence generates erratic local velocities, micro-scale particle motion introduces stochastic biases, and biological movement is affected by environmental variability.
These effects are well described by a stationary ergodic drift field, as used in combustion theory, stochastic particle systems, and population biology.
Thus, the random drift assumption in (\ref{Eq:RDERandomDrift}) captures the physical mechanism in a more realistic way. On the other hand, the stationary ergodic assumption is standard in the analysis of transport in random media and provides the correct mathematical framework to describe typical realizations of heterogeneous environments (see \cite{GartnerFreidlin1979}).

A central concern is to describe the long-time behavior as $t\rightarrow \infty$ of the solution $u(t,x)$, and to characterize the influence of randomness on such long-time behavior. It is easy to see, by maximum principle, that $0\leq u(t,x)\leq 1$ for all $t\geq 0$ and all $x\in \mathbb{R}$. The main focus can be summarized as the following:

\begin{problem}
Is it possible to show that, as time $t\rightarrow \infty$, the solution $u(t,x)$ to \emph{(\ref{Eq:RDERandomDrift})} has profiles of traveling waves? That is, we want to determine domains $R_0, R_1\subset \mathbb{R}$ such that, almost surely with respect to the environment probability $\mathbf{P}$, we have \begin{equation}\label{Problem:RDEWavePropagation:Eq:DomainsOfc}
\lim\limits_{t\rightarrow\infty}\sup\limits_{c\in R_0} u(t,ct) = 0 \ , \ 
\lim\limits_{t\rightarrow\infty}\inf\limits_{c\in R_1} u(t,ct) = 1 \ .
\end{equation}
If so, is it possible to characterize $R_0, R_1$ from the drift term $b(x)$ and the reaction term $f(x)$ in \emph{(\ref{Eq:RDERandomDrift})}? 
\end{problem}

For the deterministic version of (\ref{Eq:RDERandomDrift}) and many of its modified versions (such as general diffusion term/multidimensional/time-dependent coefficients, etc.), the above Wave Propagation Problem has been discussed thoroughly in the literature, and we refer to a recent monograph \cite{Berestycki2015AsymptoticSF} and its references therein. However, the random case has been discussed less, where the Wave Propagation Problem and similar problems of the same type have been considered in \cite[Sections 7.4-7.6, Chapter VII]{FreidlinFunctionalBook}, \cite{NolenXinDCDS}, \cite{Nolen-XinCMP2007}, \cite{FreidlinHu13}, \cite{FanHuTerlovCMP}, among others. These works have been discussing the formation of one unique wavefront of the solution $u(t,x)$ as $t\rightarrow\infty$ respectively in the positive and negative axis, separating the domains of the solution where its value is close to $0$ and is close to $1$, i.e., $R_0=(-\infty, -c^*_1)\cup (c^*_2, +\infty)$ and $R_1=(-c^*_1, c^*_2)$ for some $c^*_1, c^*_2>0$.  Variational formula for the wave speed has also been given in these works, so that the relation between the wave speeds and the equations' coefficients are discussed. However, all these works consider some form of ``zero drift". For example, $b(x)\equiv 0$ in \cite[Sections 7.4-7.6, Chapter VII]{FreidlinFunctionalBook}, and in \cite{NolenXinDCDS} it is assumed that $\mathbf{E} b(x)=0$. 

In this work, we consider the genuinely different regime where the drift $b(x)$ is ``essentially nonzero" in the sense that $\mathbf{E} b(x)>0$ (as in our Assumption (1)), and provide a complete characterization of wave propagation for general stationary ergodic random drift terms $b(x)$. Actually, in his monograph \cite[Section 7.6, Remark 4, pp.524-525]{FreidlinFunctionalBook}, Freidlin made a remark mentioning the case when there is a non-zero drift, without detailed proof. Freidlin's original problem aims to assume that together with the random drift term $b(x)$, the general diffusion term $a(x)$ and reaction term $f(x,u)$ are all random. However, the most relevant assumption would be that to assume the drift $b(x)$ is random and generally non-zero (in expectation).
To the best of our knowledge, this is the first work completely addressing wave propagation of (\ref{Eq:RDERandomDrift}) under the effects of ``essentially nonzero" random drift.

Actually, when the reaction rate $\beta>0$ is large enough, some related results can be derived in a similar way as in \cite{FanHuTerlovCMP} discussing wave propagation for RDEs on infinite random trees. In that work, by projecting the random tree to a one-dimensional axis, the authors can show that the solution to the RDE on the tree is the same as the solution to an RDE on the real-line with additional ``random drift-effects" caused by the branching structure of the random tree (for details see \cite[Section 2.3]{FanHuTerlovCMP}). Such ``random drift-effects" are nonzero, symmetric with respect to the origin $O$, and directing towards $\pm \infty$. Their role played in the dynamics can be regarded as conceptually similar to the random drift term $b(x)$ presented in (\ref{Eq:RDERandomDrift}) here, with $b(-x)\stackrel{d.}{=}-b(x)$. The results of \cite{FanHuTerlovCMP} show that if $f(u)=\beta u(1-u)$ and the reaction rate $\beta=f'(0)>0$ is larger than some $\beta_c>0$, then there will still be a unique wave front formed by the solution to the RDE as $t\rightarrow\infty$. 

However, the results in \cite{FanHuTerlovCMP} cannot answer the question when $\beta$ is smaller than $\beta_c$ or when $b(x)$ is not symmetrically distributed with respect to the origin $O$. This is because we cannot directly apply the Large Deviations Principle for the underlying diffusion process of (\ref{Eq:RDERandomDrift}) with restricted parameter set as in \cite{FanHuTerlovCMP}; and when $b(x)$ is not symmetrically distributed with respect to the origin $O$, we cannot simply flip the wave front from the positive axis to the negative axis, etc.

We overcome the above difficulties in this work. Using probabilistic arguments such as the Large Deviation Principle for random processes in a random environment (see \cite{CometsGantertZeitouni2000}), as well as the functional representation of the solution $u(t,x)$ in terms of the underlying diffusion process (Feynman-Kac formula), we show that, for arbitrarily fixed reaction rate $\beta>0$, depending on the ``strength" of $b(x)$ \footnote{The exact quantitative meaning of the ``strength" of $b(x)$ is characterized by a critical value $\eta_c$ that we will describe in Section \ref{Sec:BasicIdeaMainResults}.}, the wave propagation demonstrates different behavior: if $b(x)$ is  shifting to the right not very strongly, then one sees again two wavefronts with one wave traveling to $+\infty$ at speed $c^*_2$ and the other wave traveling to $-\infty$ at speed $c^*_1$, similar to \cite{Nolen-XinCMP2007}; however, if $b(x)$ is shifting to the right very strongly, then one will see two wavefronts both traveling to $-\infty$, with one wave traveling at the front and another ``negative" wave chasing the front wave and eliminating the effects of the first wave, and in between these two wavefronts there is an expanding region where the limit of $u(t,x)$ as $t\rightarrow \infty$ is $1$, i.e., i.e., $R_0=(-\infty, -c^*_1)\cup (c^*_2, +\infty)$ and $R_1=(-c^*_1, c^*_2)$ for some $c^*_1>0, c^*_2<0$ and $c^*_1+c^*_2>0$. The latter surprising phenomenon can be imagined to come from some strong enough ``push in the negative direction" of the whole wave profile in the first case (i.e. having one positive wavefront and one negative wavefront), which from a PDE perspective is due to the interaction between the drift $b(x)$, the diffusion term and the reaction term. In addition, our rigorous probabilistic arguments also justify the underlying physical mechanism: from Statistical Mechanics perspective, the drift enters as an external field that biases propagation only through a gauge shift of the (quenched) free energy, while leaving the intrinsic fluctuation structure governing the wave formation unchanged.

While preparing this paper, we have noticed a very recent literature  \cite{LIANG2022108568}, which uses purely analytic method involving generalized principle eigenvalues (see \cite{Berestycki2015AsymptoticSF}), and obtains similar predictions on the possibility to have two wavefronts propagating in the same negative direction (assuming the drift on average is pointing in the positive direction). The setting of \cite{LIANG2022108568}, however, concerns (almost) periodic drifts $b(x)$, whereas our results treat
general stationary ergodic random drifts and provide a complete characterization
of wave propagation in this quenched random regime. Since almost periodic functions and ergodic stationary random processes are distinct classes of functions, our work should be considered in parallel with \cite{LIANG2022108568}. Interestingly, we also find that one of the main results of \cite{LIANG2022108568}, namely their Theorem 1.3, which in our case can be stated as $\text{sgn}(c^*_1-c^*_2)=\text{sgn}(\mathbf{E} b(x))$, can also be recovered from our arguments specifically in the case when $\mathbf{E} b(x) \neq 0$ (in our case $\mathbf{E} b(x)>0$, but the $\mathbf{E} b(x)<0$ case follows the same reasoning by just flipping the sign).

The paper is organized as follows: In Section \ref{Sec:BasicIdeaMainResults} we present the basic idea how we solve the Wave Propagation Problem in case when $\mathbf{E} b(x)>0$ is nonzero and the reaction rate $\beta$ has arbitrary magnitude, and we summarize our main results. In Section \ref{Sec:LDP} we present technical results regarding the long-time probabilistic behavior of the diffusion process underlying (\ref{Eq:RDERandomDrift}), especially the needed Large Deviation Principle for the analysis of the wave propagation. In Section \ref{Sec:WavePropagation} we analyze the Wave Propagation Problem and we provide the detailed proof of the wave propagation results announced in Section \ref{Sec:BasicIdeaMainResults}. In Section \ref{Sec:Shape} we describe the exact shape of the wave fronts.

\

\noindent \textbf{Acknowledgment.} We would like to thank Louis Waitong Fan and Greg Terlov for early discussions on this topic. While working on this paper, Dihang Guan and Hui He are supported by the National Key R$\&$D Program of China (No. 2020YFA0712900), NSFC (No. 12271043) of China; Wenqing Hu is supported by the Simons Foundation Collaboration Grant No. 711999; and Jiaojiao Yang is supported by the National Natural Science Foundation of China Grant No.11801199, Natural Science Foundation of Anhui Province (CN) Grant No.1908085QA30. 

\section{Basic Idea of the Proof and Main Results}\label{Sec:BasicIdeaMainResults}
Our approach is probabilistic (see \cite{CometsGantertZeitouni2000}, \cite{FreidlinFunctionalBook}, \cite{NolenXinDCDS} \cite{FanHuTerlovCMP}). We consider the underlying diffusion process given by the following stochastic differential equation written in integral form

\begin{equation}\label{Eq:SDE-RandomDrift}
X^x(t)=x+\int_0^t b(X^x(s)){\rm d}s+W_t \ ,
\end{equation}
where $W_t$ is a standard Brownian motion, with given filtration $(\Omega^W, \mathcal{F}_t^W, P^W)$. 

Notice that we now have two stages of randomness in our set-up: the environmental randomness (quenched probability, see \cite{CometsGantertZeitouni2000}) is characterized by $(\Omega, \mathcal{F}, \mathbf{P})$ and for each fixed environment sample $\omega \in \Omega$, the randomness of the underlying diffusion process (annealed probability \cite{CometsGantertZeitouni2000}) characterized by $(\Omega^W, \mathcal{F}^W, P^W)$.

A probabilistic representation of the solution $u(t,x)$ to \eqref{Eq:RDERandomDrift} is given by the Feynman-Kac formula (see \cite[Section 5.1]{FreidlinFunctionalBook})

\begin{equation}\label{Eq:FeynmannKacRDERandomDrift}
u(t,x)=E^W\left[\underbrace{u_0(X^x(t))}_{\text{cooling}} 
\overbrace{\exp\left(\int_0^t c(u(t-s, X^x(s))){\rm d}s\right)}^{\text{heating}}\right] \ ,
\end{equation}
where $c(u)=\dfrac{f(u)}{u}$ is defined for $u>0$ and $c(0)=\lim\limits_{u\rightarrow 0}c(u)$. By the FKPP property of $f(u)$, we know that $c(u)\geq 0$ for $u\in (0,1]$ and $c(0)=\sup\limits_{u\in (0,1]}c(u)=\beta>0$.

Set $x=ct$, $c>0$\footnote{When $c<0$, the same arguments shown here lead to the formation of wave fronts along the negative direction. This will be considered later in this Section.} From the equation \eqref{Eq:FeynmannKacRDERandomDrift} we see that for those regions of $(t,x)$
that the value of $u(t,x)$ is small (close to $0$), and while $(t-s, X^x(s))$ is still lying in that region, then the contribution of the term  $\displaystyle{\exp\left(\int_0^t c(u(t-s, X^x(s))){\rm d}s\right)}$ in (\ref{Eq:FeynmannKacRDERandomDrift}) is expected be asymptotically $\sim \exp(\beta t)$. This means that the reaction term $f(u)$ will create an exponential
birth of the particles at a rate of $\beta$ in those domains where the solution $u$ is small, i.e., it provides a ``heating" mechanism to $u(t,x)$ (see the overbrace in (\ref{Eq:FeynmannKacRDERandomDrift})).
However, in \eqref{Eq:FeynmannKacRDERandomDrift} this exponential term $\displaystyle{\exp\left(\int_0^t c(u(t-s, X^x(s))){\rm d}s\right)}$ is
multiplied by $u_0(X^x(t))$, which is zero outside its compact support. Taking $E^W$-expectation, this implies that the exponential growth $\exp(\beta t)$ is compensated by $P^W(X^x(t)\in \text{supp} u_0)$, which, as $t\rightarrow\infty$ for $x=ct, c>0$
can be characterized by Large Deviations Principle (LDP) as an exponentially contracting term $\exp\left(-t S(c)\right)$. The function $S(c)$ is called the \textit{rate function} (or \textit{action function}) of the process $X^x(t)$ in LDP theory. Such an exponentially contracting term can be regarded as providing a ``cooling" mechanism to $u(t,x)$ (see the underbrace in (\ref{Eq:FeynmannKacRDERandomDrift})).
Thus when the competition between the ``heating" and ``cooling" effects,
namely, the exponential growth due to reaction and the exponential contraction due to 
the large deviation effect of the underlying diffusion process,
reaches a balance, we will see the propagation of a wavefront separating domains where the solution is close to $0$ and is close to $1$. 

As a rough statement, to compute the wave speed, we consider the LDP rate function $S(c)$ for $X^{ct}(t)$, the underlying diffusion process (\ref{Eq:SDE-RandomDrift}) starting from $x=ct$. The ``heating v.s. cooling" balance as mentioned above can then be formally written as $S(c)=\beta$, which can be shown to have a unique solution $c=c^*>0$ provided that $S(c)$ satisfies some monotonicity conditions. The quantity $c^*>0$ is the desired wave speed. 

To carry out the above program, we need to obtain some LDP for the underlying diffusion process $X^x(t)$. Notice that since the drift $b(x)$ is random, the process $X^x(t)$ is a random process in a random environment. LDP for such processes has been considered in many previous works, such as \cite{CometsGantertZeitouni2000}, \cite{SolomonRWREAnnProb}, \cite{Taleb2001}. Here we use the approach that was first introduced in \cite{CometsGantertZeitouni2000} and was further developed in \cite{NolenXinDCDS}, \cite{FanHuTerlovCMP}. This method introduces the hitting times 
$$T^s_r=\inf\{t>0: X^s(t)\leq r\} \ , \ T^r_s = \inf\{t>0: X^r(t)\geq s\} \ , \ r\leq s \ .$$

Let $c>v>0$. We will first obtain the LDP of $\dfrac{T^{ct}_{vt}}{t}\mathbf{1}_{\{T^{ct}_{vt}<+\infty\}}$ and 
$\dfrac{T^{vt}_{ct}}{t}\mathbf{1}_{\{T^{vt}_{ct}<+\infty\}}$ as $t\rightarrow\infty$,
and then we revert them to the LDP for the process $X^x(t)$. 
To this end we introduce the Lyapunov functions
$$
\mu(\eta)\equiv \mu^\bwd(\eta) \stackrel{\text{def}}{=} \mathbf{E}\left[\ln E^W\left(
e^{\eta T^1_0}\mathbf{1}_{\{T^1_0<+\infty\}}\right)\right] \ , \ 
\mu^\fwd(\eta) \stackrel{\text{def}}{=} \mathbf{E}\left[\ln E^W\left(
e^{\eta T^0_1}\mathbf{1}_{\{T^0_1<+\infty\}}\right)\right] \ .$$

\noindent Let
$$
\eta_c\equiv \eta_c^\bwd =\sup\{\eta\in \mathbb{R}: \mu(\eta)<+\infty\} \ , \ 
\eta_c^\fwd =\sup\{\eta\in \mathbb{R}: \mu^\fwd(\eta)<+\infty\} \ .
$$

\noindent It can be easily seen that in general $\eta_c, \eta_c^\fwd \geq 0$.
Define the Legendre transform of $\mu(\eta), \mu^\fwd(\eta)$ as 

$$
I(a)\equiv I^\bwd(a)=\sup\limits_{\eta\leq \eta_c}(a\eta-\mu(\eta)) \ , \ 
I^\fwd(a)=\sup\limits_{\eta\leq \eta_c^\fwd}(a\eta-\mu^\fwd(\eta)) \ .
$$

In \cite[Theorem 2.3]{NolenXinDCDS}, the function $(v-c)I\left(\dfrac{a}{v-c}\right)$ ($0<c<v$) is shown to be the rate function of the family of random variables $\dfrac{T^{vt}_{ct}}{t}\mathbf{1}_{\{T^{vt}_{ct}<+\infty\}}$ as $t\rightarrow \infty$. In turn, by a duality argument that helps to pass from the LDP of hitting time to the LDP for process $X^x(t)$, the LDP rate function $S(c)$ is given by $S(c)=cI\left(\dfrac{1}{c}\right)$. However, the argument of \cite{NolenXinDCDS} can only conclude that the LDP holds for (open and closed) subsets of $(0, (v-c)\mu'(0))$ instead of $(0, +\infty)$. Under the assumption that the averaged drift is $0$, i.e., $\mathbf{E} b(x)=0$, \cite{NolenXinDCDS} showed that $\eta_c=0$, $\mu'(0)=+\infty$ and thus such LDP domain restrictions will not present in the particular case of zero (mean) drift. Due to this fact, \cite{NolenXinDCDS} was able to obtain the existence of a unique traveling wave front of \eqref{Eq:RDERandomDrift} as $t\rightarrow \infty$.

In general, when the mean drift is non-zero and may vary with respect to the $x$-variable ($\mathbf{E} b(x)\neq 0$), the value $\mu'(0)$ is finite and thus the LDP with domain restrictions as was used in \cite{NolenXinDCDS} will not be able to provide the necessary estimates for the arguments leading to the existence of a unique traveling wave. As was mentioned above, the wave speed is a result of the balance equation $S(c)=\beta$. So if we stick to the method of obtaining LDP restricted only to subsets of $(0, (v-c)\mu'(0))$ as in \cite{NolenXinDCDS}, then we can take a compensation in which we let the quantity $\beta$ to be large. This is because when $\beta$ is larger than some threshold value depending on the drift $b(x)$, we can still show that $S(c)=\beta$ has a unique solution $c^*>0$, and the value $c^*$ just lies in the domain $(0, (v-c)\mu'(0))$ where the LDP works. This argument, done in \cite{FanHuTerlovCMP}, needs a very careful exhaustive investigation of all possible shapes of $I(a)$. Thus under the assumption that $\beta$ is large, a unique wave front is formulated as $t\rightarrow \infty$. 

The above argument used in \cite{NolenXinDCDS}, \cite{FanHuTerlovCMP} becomes invalid in the case when $\beta$ is small and $\mathbf{E} b(x)\neq 0$, since in this case the LDP for the process can only provide useful estimates for (open and closed) subsets of $(0, (v-c)\mu'(0))$ instead of $(0, \infty)$, and such parameter domain restrictions prevent the use of LDP to derive the ``heating v.s. cooling" arguments. Moreover, even if we could apply the LDP, we need to consider two rate functions $S^\bwd(c)$ and $S^\fwd(c)$ that correspond to $I^\bwd(a)$ and $I^\fwd(a)$, respectively. Given a fixed $\beta$, the rate functions in each direction will generate a balance equation, i.e., $S(c)\equiv S^\bwd(c)=\beta$ (for the wave front propagating in the positive direction) and $S^\fwd(c)=\beta$ (for the wave front propagating in the negative direction), whose solutions determine the wave propagation, but the structure of the solutions (i.e. no solution, one solution or two solutions) and how they determine the shape of the wave need to be carefully examined.

We overcome all the above difficulties in this work, and our main results and contributions are summarized as follows:

\begin{itemize}
\item[(1)] We derive full LDP for (open and closed) subsets of $(0, +\infty)$ without any parameter restrictions. To this end, we first obtain the LDP for $\dfrac{T^{vt}_{ct}}{t}\mathbf{1}_{T^{vt}_{ct}<\infty}$ and $\dfrac{T^{ct}_{vt}}{t}\mathbf{1}_{T^{ct}_{vt}<\infty}$ as $t\rightarrow \infty$ (see Section \ref{Sec:LDP:HittingTime}, Theorems \ref{Thm:LDPHittingTime}, \ref{Thm:LDPHittingTime:Forward}). We remove the restriction of this LDP from sets in $(0, (v-c)\mu'(0))$ as in \cite{NolenXinDCDS} to $(0, +\infty)$ by using a truncation argument developed in \cite{CometsGantertZeitouni2000} adapted to our case (Lemma \ref{Lm:PropertiesMuAndTruncatedMu}). Given the full LDP for $\dfrac{T^{vt}_{ct}}{t}\mathbf{1}_{\{T^{vt}_{ct}<+\infty\}}$, $\dfrac{T^{ct}_{vt}}{t}\mathbf{1}_{\{T^{ct}_{vt}<+\infty\}}$, we use a duality argument to turn them into the LDP of the random variables $\dfrac{vt-X^{vt}(\kappa t)}{\kappa t}$ and $\dfrac{-vt-X^{-vt}(\kappa t)}{\kappa t}$ as $t\rightarrow\infty$, where $v>0$ and $\kappa\in (0,1]$ (See Section \ref{Sec:LDP:ProcessX}, Theorem \ref{Thm:LDPProcessX}). Unlike \cite{NolenXinDCDS}, this duality argument needs a much more involved stopping time segmentation technique to overcome the problem of domain restriction (compare with \cite{CometsGantertZeitouni2000} which is for the discrete random walk case). We resolve various new challenges when carrying this argument to the continuous case for our underlying diffusion process.

\item[(2)] We extend the ``heating v.s. cooling" argument based on Feynmann-Kac formula (\ref{Eq:SDE-RandomDrift}) to derive the wave propagation (see Section \ref{Sec:WavePropagation}, Theorem \ref{Thm:WavePropagation}). Using our LDP results in Theorem \ref{Thm:LDPProcessX}, we can derive that $S(c)\equiv S^\bwd(c)=cI\left(\dfrac{1}{c}\right)$ and $S^\fwd(c)=cI^\fwd\left(\dfrac{1}{c}\right)$. In response to our Wave Propagation Problem, we show that $R_0=\{c: S(c)<\beta\}\cup \{-c: S^\fwd(c)<\beta\}$ and $R_1=\{c: S(c)>\beta\}\cup \{-c: S^\fwd(c)>\beta\}$. 

\item[(3)] To determine the exact shapes of $R_0$ and $R_1$, we analyze the relations between $\mu(\eta)$ and $\mu^{\fwd}(\eta)$ (see Section \ref{Sec:LDP:RelationMuAndMufwd}, Lemma \ref{Lm:BwdAndFwdCriticalEtaForMuAreTheSame}, Proposition \ref{Prop:MuBwdMinusMuFwd}, Lemma \ref{Lm:ExpectationLnRhoIsConstant}). Interestingly, we can prove that $\eta_c\equiv \eta_c^\bwd=\eta_c^\fwd$ (Lemma \ref{Lm:BwdAndFwdCriticalEtaForMuAreTheSame}), so that $\mu(\eta)$ and $\mu^\fwd(\eta)$ shares the same domain of finiteness whose interior is $(-\infty, \eta_c)$. Furthermore, in Lemma \ref{Lm:ExpectationLnRhoIsConstant} we demonstrated that within this common domain of finiteness, the difference $\mu(\eta)-\mu^\fwd(\eta)$ is actually a constant. This follows from the fact that the corresponding annealed difference admits a cocycle representation over the spacial shift (see equation (\ref{Lm:ExpectationLnRhoIsConstant:Eq:Cocycle})), whose contribution vanishes while taking the quenched expectation, due to stationarity of the external drift $b(x)$. 

The above cocycle over the annealed difference corresponding to $\mu(\eta)-\mu^\fwd(\eta)$ is analogous to entropy-production functionals in Non-equilibrium Statistical Mechanics \cite{evans1993probability}. Although a nonzero drift $b(x)$ drives the system out of equilibrium, stationarity implies that this cocycle reduces to a boundary (gauge) term, shifting the (quenched) free-energy level without affecting its derivatives (compare with \cite[Chapters II and XIV]{landau1980statistical}). Consequently, the Scaled Cumulant Generating Function (SCGF) for backward and forward propagation (i.e. $\mu(\eta)$ and $\mu^\fwd(\eta)$) differ only by an additive constant and therefore possess identical derivatives ($\mu'(\eta)=(\mu^\fwd)'(\eta)$). In thermodynamic terms, the non-equilibrium driving modifies the (quenched) free energy but leaves the fluctuation symmetry and linear-response structure invariant (compare with the Green-Kubo theory \cite{green1954markoff}, \cite{kubo1957statistical}).

\item[(4)] From our results in (2) and (3) discussed above, in Section \ref{Sec:Shape} we characterize the exact shapes of the wave front (Theorem \ref{Thm:ShapeAsymptoticWave}). First, using our derived properties of $\mu(\eta), \mu^\fwd(\eta)$, we characterize the relative positional configurations of $I(a)$ and $I^\fwd(a)$ in Lemma \ref{Lm:FurtherPropertiesOfMuAndIUsedInWaveFrontShape}. By considering intersections of the line $a\mapsto \beta a$ with the curves $a\mapsto I(a)$ and $a\mapsto I^\fwd(a)$ (see Lemma \ref{Lm:SolutionStructureBalanceEquation}), we conclude the following table regarding the Wave Propagation Problem:

\begin{table}[H]
\centering
\begin{tabular}{|c|c|c|c|}
\hline
Case of $\beta$ & $R_0$ & $R_1$ & Properties of $c_1^*, c_2^*$
\\
\hline 
$\beta\in (0, \eta_c)$ & $(-\infty, -c^*_1)\cup (c^*_2, +\infty)$ & $(-c^*_1, c^*_2)$ & $c^*_1>0, c^*_2<0, c^*_1+c^*_2>0$
\\
\hline 
$\beta = \eta_c$ & $(-\infty, -c_1^*)\cup(0,+\infty)$ & $(-c_1^*, 0)$ & $c_1^*>0$
\\
\hline 
$\beta\in (\eta_c, +\infty)$ & $(-\infty, -c^*_1)\cup (c^*_2, +\infty)$ & $(-c^*_1, c^*_2)$ & $c^*_1>0, c^*_2>0, c^*_1-c^*_2>0$
\\
\hline
\end{tabular}
\caption{Limit behavior of solution $\lim\limits_{t\rightarrow\infty} u(t,ct)$ based on $\beta$.}
\label{Table:BehaviorOfSolution}
\end{table}

We depict a picture describing the possible shapes of our wave front in Figure \ref{Fig:TwoWaves}.

\begin{figure}[H]
\centering
\includegraphics[height=5cm, width=13cm]{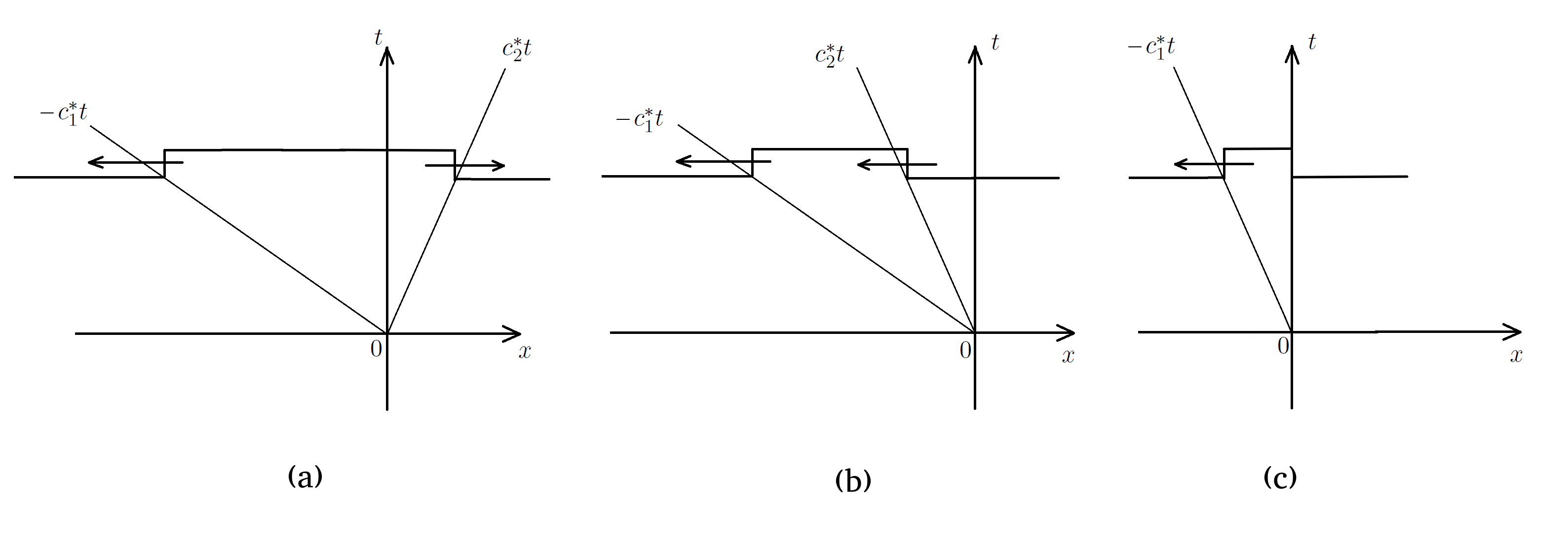}
\caption{The asymptotic wave formation: (a) $\beta\in (0, \eta_c)$; (b) $\beta =\eta_c$; (c) $\beta\in (\eta_c, +\infty)$.}
\label{Fig:TwoWaves}
\end{figure}

To understand the physical intuition of wave phenomenon described above, we recall that the ``heating" mechanism is a result of reaction, so that the term $\exp\left(\displaystyle{\int}_0^t c(u(t-s, X^x(s))){\rm d}s\right)$ pushes $u$ up when $X^x(t)\in \text{supp} u_0$. Given $c\in \mathbb{R}$ and set $x=ct$, then as $t\rightarrow\infty$ the process $X^{ct}(t)$ starts from the initial point $x=ct$ traveling at speed $c$, returns the compact set $\text{supp} u_0$ at an exponentially small probability $\exp(-t S(c))$ (``cooling" mechanism). If $c>0$, $S(c)$ characterizes the exponential probability of coming back around the origin when $X^{ct}(t)$ is at the positive far end; and reversely, if $c<0$, then $S(c)$ is measuring the return back probability from the negative far end. Since our drift term $b(x)$ is stationary and ergodic with positive average, the process $X^x(t)$ travels to the positive direction faster, and thus coming back close to the origin is harder from the positive far end than from the negative far end. In other words, cooling is stronger on the positive far end than the negative far end. Therefore, with the drift $b(x)$, wave propagation tends to be ``slowed down" towards $+\infty$ and ``accelerated" towards $-\infty$, leading to case (a) in Figure \ref{Fig:TwoWaves}. Given fixed drift $b(x)$, if the reaction rate $\beta$ is small (this is equivalent to letting $\beta$ fixed and tune up the drift $b(x)$ strong enough to the positive direction), then it can happen that the whole of positive real axis is eliminated in the wave propagation due to the strong cooling effect in that direction (as compared to the heating). In other words, even when the process $X^{ct}(t)$ starts from $ct$ with $c<0$ but $|c|$ being moderate, the drift $b(x)$ will quickly drag it to the positive far end and the process can hardly come back close to the origin, leading to $c_2^*<0$ as depicted in case (b) of Figure \ref{Fig:TwoWaves}. Of course, between these two phases there is a critical $\beta$ (which we show to be equal to $\eta_c$), at which the positive wave becomes stagnant ($c_2^*=0$), that we show at part (c) of Figure \ref{Fig:TwoWaves}.

\item[(5)] We discuss the relation $\text{sgn}(c^*_1-c^*_2)=\text{sgn}(\mathbf{E} b(x))$ and compare it with Theorem 1.3 of \cite{LIANG2022108568}. It turns out that when $\mathbf{E} b(x)>0$ as in our case, due to the shapes of and relative position between $I(a)$ and $I^\fwd(a)$, it is a direct result that $c_1^*>c_2^*$, or in other words $\text{sgn}(c_1^*-c_2^*)=\text{sgn}(\mathbf{E} b(x))>0$ (see Lemma \ref{Lm:SolutionStructureBalanceEquation}). The case when $\mathbf{E} b(x)<0$ can be obtained similarly by just flipping the sign. 

However, when $\mathbf{E} b(x)=0$, since this does not guarantee $b(x)\stackrel{d.}{=}-b(-x)$, and $\mu(\eta), \mu^\fwd(\eta)$ are  logarithms of the moment generating functions of the hitting times $T^1_0$ and $T^0_1$, it can happen that $c_1^*\neq c_2^*$, i.e., $\text{sgn}(c_1^*-c_2^*)\neq 0$.  And $c_1^*=c_2^*$ may hold if the corresponding scale function integrals behave the same as the case of ${\mathbf E}b(x)>0$ (\eqref{eqn: generalintecond} holds); see Remark \ref{Rem:scalefunction} in  Section \ref{Sec:LDP:RelationMuAndMufwd}.

\end{itemize}

The rest of this article consists of technical arguments addressing the above issues.

\section{Probabilistic Behavior of the Process $X^x(t)$ at large time $t\rightarrow \infty$}\label{Sec:LDP}

\subsection{Large Deviations for the Hitting Time}\label{Sec:LDP:HittingTime}
In this section, we obtain LDP for certain hitting times of the process $X^x(t)$ defined in (\ref{Eq:SDE-RandomDrift}), which will enable us to later turn them into the LDP of the process $X^x(t)$ using a duality argument. 

Assume $0<c<v$. For $s>r$ we define the (backward) hitting time $T^s_r$ as 
\begin{equation}\label{Eq:HittingTimeSDE-RandomDrift:Backward}
T^s_r = \inf\{t>0: X^s(t)\leq r\} \ .
\end{equation}

\noindent We aim to obtain LDP of $\dfrac{T^{vt}_{ct}}{t}\mathbf{1}_{\{T^{vt}_{ct}<\infty\}}$ as $t\rightarrow\infty$. Define the Lyapunov function \begin{equation}\label{Eq:LyapunovFunctionHittingTime:Backward}
\mu^\bwd(\eta)\equiv \mu(\eta)=\mathbf{E}\left[\ln E^W\left(e^{\eta T^1_0}\mathbf{1}_{\{T^1_0<+\infty\}}\right)\right]
\end{equation}
\footnote{We use superscript $\bwd$ here to stand for backward. Unless otherwise specified, the superscript $\bwd$ is usually suppressed for notational convenience.} and the critical value
\begin{equation}\label{Eq:CriticalEta:Backward}
\eta_c^\bwd\equiv \eta_c=\sup\{\eta\in \mathbb{R}: \mu(\eta)<+\infty\} \ .
\end{equation}

\noindent Define the Legendre transform of $\mu(\eta)$ as 
\begin{equation}\label{Eq:EntropyHittingTime:Backward}
I^\bwd(a)\equiv I(a)=\sup\limits_{\eta\leq \eta_c}(a\eta-\mu(\eta)) \ .
\end{equation}

\noindent For any $\eta\in \mathbb{R}$, we define

\begin{equation}\label{Eq:LDPHittingTimeFreeEnergyUnscaled}
H^{t,\bwd}(\eta)\equiv H^t(\eta)=\ln E^W \exp\left(\eta \cdot \dfrac{T^{vt}_{ct}}{t}\mathbf{1}_{\{T^{vt}_{ct}<\infty\}}\right) \ ,
\end{equation}
and 
\begin{equation}\label{Eq:LDPHittingTimeFreeEnergy}
H^\bwd(\eta)\equiv H(\eta)=\lim\limits_{t\rightarrow\infty}\dfrac{1}{t}H^t(t\eta) \ .
\end{equation}

\begin{lemma}[Lyapunov Exponent Identity]\label{Lm:LDPHittingTimeFreeEnergyExist}
Let $\eta\in\mathbb{R}$ be such that 
\begin{equation}\label{Lm:LDPHittingTimeFreeEnergyExist:AbsoluteIntegrabilityCondition}
\mathbf{E}\left[\left|\ln E^{W}\left(e^{\eta T_0^1}\mathbf{1}_{\{T_0^1<\infty\}}\right)\right|\right]<\infty \ .
\end{equation}
Assume $0<c<v$. Then $\mathbf{P}$-almost surely the limit \emph{(\ref{Eq:LDPHittingTimeFreeEnergy})} holds and we have $$H(\eta)=(v-c)\mu(\eta) \ ,$$
where $\mu(\eta)$ is defined in \emph{(\ref{Eq:LyapunovFunctionHittingTime:Backward})}.  
Thus $H(\eta)<\infty$ when $\eta<\eta_c$ and $H(\eta)=\infty$ when $\eta>\eta_c$, where $\eta_c$ is defined in \emph{(\ref{Eq:CriticalEta:Backward})}. 
\end{lemma}

\begin{proof}
This lemma can be proved in a similar way as Theorem 3 in \cite{FanHuTerlovCMP}, using the fact that $X^x(t)$ is a strong Markov process and the ergodicity of $b(x)$ with respect to spacial shifts.
\end{proof}

The Legendre transform of $H(\eta)$ is defined by 

\begin{equation}\label{Eq:LDPHittingTimeEntropy}
H^*(a):=L(a)= \sup\limits_{\eta\in (-\infty, +\infty)}(a\eta-H(\eta)) = \sup\limits_{\eta\leq \eta_c}(a\eta-H(\eta)) \ .
\end{equation}

\noindent It is easy to see that 

$$\begin{array}{ll}
\sup\limits_{\eta\leq \eta_c} (a\eta - H(\eta))
& = \sup\limits_{\eta\leq \eta_c} (a\eta - (v-c)\mu(\eta))
\\
& = (v-c)\sup\limits_{\eta\leq \eta_c}\left(\dfrac{a}{v-c}\eta-\mu(\eta)\right)
\\
& = (v-c)I\left(\dfrac{a}{v-c}\right) \ ,
\end{array}$$
where $I(a)$ is the Legendre transform of $\mu(\eta)$ defined in (\ref{Eq:EntropyHittingTime:Backward}). Thus 
\begin{equation}\label{Eq:LDPHittingTimeEntropyUsingLegendreOfLyapunov} L(a)=(v-c)I\left(\dfrac{a}{v-c}\right) \ .
\end{equation}

\begin{theorem}[LDP for the Hitting Time] \label{Thm:LDPHittingTime}
Let $0<c<v$. Then $\mathbf{P}$-almost surely the following two estimates hold: 

\begin{itemize}
\item[\emph{(a)}] For
any closed set $G\subset (0,+\infty)$ we have
\begin{equation}\label{Thm:LDPHittingTime:Eq:UpperBound}
\limsup\limits_{t\rightarrow\infty}\dfrac{1}{t}\ln P^W
\left(\dfrac{T_{ct}^{vt}}{t}\in G\right)\leq -(v-c)\inf\limits_{a\in G}I\left(\dfrac{a}{v-c}\right) \ ;
\end{equation}
\item[\emph{(b)}] For any open set $F\subset (0,+\infty)$ we have
\begin{equation}\label{Thm:LDPHittingTime:Eq:LowerBound}
\liminf\limits_{t\rightarrow\infty}\dfrac{1}{t}\ln P^W
\left(\dfrac{T_{ct}^{vt}}{t}\in F\right)\geq -(v-c)\inf\limits_{a\in F}I\left(\dfrac{a}{v-c}\right) \ .
\end{equation}
\end{itemize}
\end{theorem}

\begin{remark}\label{Remark:ExplanationOfLDPHittingTime}
\rm 
In \cite[Theorem 2.3]{NolenXinDCDS} and \cite[Theorem 4]{FanHuTerlovCMP}, similar quenched LDP results were proved. We know by part (4) of Lemma 5.1 in \cite{FanHuTerlovCMP}, that $\eta_c\in [0, \infty)$ $\mathbf{P}$-almost surely. The argument used in \cite[Theorem 2.3]{NolenXinDCDS} and \cite[Theorem 4]{FanHuTerlovCMP} requires that $G, F\subset (0, (v-c)\mu'(0))=(0, H'(0))$. So if $H'(0)=+\infty$, then the above Theorem \ref{Thm:LDPHittingTime} can be proved using the same argument in \cite[Theorem 4]{FanHuTerlovCMP}, \cite[Theorem 2.3]{NolenXinDCDS}, and we do not have to repeat the proof. 

Note that in our case, since there is a positive mean of the drift term $b(x)$, it might be common that the case $\eta_c>0$ does happen (see Section \ref{Sec:BasicIdeaMainResults}). When $\eta_c>0$, usually we will have $\mu'(0)<+\infty$ and therefore $H'(0)<+\infty$, since $H'(0)=(v-c)\mu'(0)$. Moreover, even if $\eta_c=0$, it can still happen that $\mu'(0)<+\infty$ and therefore $H'(0)<+\infty$.  In these cases, the previous arguments used in \cite[Theorem 2.3]{NolenXinDCDS} and \cite[Theorem 4]{FanHuTerlovCMP} do not apply, and so some extra efforts have to be made. 

Let us first consider the case when $\eta_c>0$. Our first attempt is to dig into the classical G\"{a}rtner-Ellis lemma for LDP of general probability measures (see \cite[Section 2.3]{DemboZeitouniLDPBook}, \cite[Theorem 3.4]{OlivieryVaresLDPBook}, \cite{GartnerLemma}, \cite{EllisGartnerLemma}). Define
\begin{equation}\label{Thm:LDPHittingTime:Eq:SetOfEtaWithFiniteH}
\mathcal{D}_{H}:=\{\eta: H(\eta)<+\infty\} \ .
\end{equation}
In the G\"{a}rtner-Ellis lemma, it is first required that $0$ is an interior point of $\mathcal{D}_H^\circ$. 
By Lemma \ref{Lm:LDPHittingTimeFreeEnergyExist} we have $\mathcal{D}_{H}=\{\eta: \mu(\eta)<\infty\}=(-\infty, \eta_c)$ or $(-\infty, \eta_c]$, and therefore $\mathcal{D}_H^\circ =(-\infty, \eta_c)$. So if $\eta_c>0$, then $0$ is indeed an interior point of the set $\mathcal{D}_{H}^{\circ}$. We can the apply the upper bound argument of G\"{a}rtner-Ellis to have the LDP. However, the lower bound of G\"{a}rnter-Ellis requires technical conditions such as $H(\eta)$ be an essentially smooth (see \cite[Definition 2.3.5]{DemboZeitouniLDPBook}), lower semi-continuous function. Such a restriction on $H(\eta)$ enables the lower bound to hold with an $\inf$ on $a\in F$ instead of $a\in F\cap \mathcal{F}$, where $F$ is the open set in part (b) of Theorem \ref{Thm:LDPHittingTime} and $\mathcal{F}$ is the set of exposed points of $H(\eta)$ (see \cite[Definition 2.3.3]{DemboZeitouniLDPBook}). In this case $H(\eta)$ is a good rate function for the LDP. The requirement of $H(\eta)$ being essentially smooth is to enable the use of Rockafellar's Lemma (see \cite[Lemma 2.3.12]{DemboZeitouniLDPBook}) to pass from $a\in F\cap \mathcal{F}$ to any $a\in F$ (see part (c) of \cite[Theorem 2.3.6]{DemboZeitouniLDPBook}). 
In our case, we find that due to Lemma \ref{Lm:LDPHittingTimeFreeEnergyExist}, the condition $\mu'(\eta_c)=+\infty$ is equivalent to the fact that $H(\eta)$ is essentially smooth. This means, in order to have the G\"{a}rtner-Ellis argument work for LDP lower bound, we must require that $\mu'(\eta_c-)=+\infty$.

In the case when $\eta_c=0$ but $\mu'(0)<+\infty$, G\"{a}rtner-Ellis also cannot be applied to the upper bound, since in this case $0$ may not be an interior point of $\mathcal{D}_H^\circ$.

To remedy these problems, we make use of the fact that we are working in dimension $1$ and we adapt the symmetry and the truncation arguments used in \cite{CometsGantertZeitouni2000}, but with careful modifications to our case. These arguments finally help us to settle the entire LDP without additional assumptions such as $\mu'(\eta_c-)<+\infty$ and $0$ being in the interior of $\mathcal{D}_H^\circ$.

Note that when $\eta_c>0$, \cite{FanHuTerlovCMP} needs to work with finite $\mu'(0)$ and thus $G, F\subset (0, (v-c)\mu'(0))$. This forces the authors of \cite{FanHuTerlovCMP} to work with large reaction (see \cite[Assumption 4]{FanHuTerlovCMP}) to argue the wave front formation. In our case, we can prove LDP when $G, F\subset (0, +\infty)$ and therefore we will not be restricted to large reaction. 
\end{remark}

\begin{proof}[Proof of Theorem \ref{Thm:LDPHittingTime}]

We separate the proof into two parts: the upper bound and lower bound.

\noindent{\textit{Upper Bound.}} We reduce (\ref{Thm:LDPHittingTime:Eq:UpperBound}) to the following two inequalities: \begin{equation}\label{Thm:LDPHittingTime:Eq:UpperBoundReducedToIntervalPart1} 
\limsup\limits_{t\rightarrow\infty}\dfrac{1}{t}\ln P^W\left(\dfrac{T^{vt}_{ct}}{t}< \alpha\right)\leq -(v-c)I\left(\dfrac{\alpha}{v-c}\right) \text{ when } 0<\alpha\leq (v-c)\mu'(0) \ ,
\end{equation}
and
\begin{equation}\label{Thm:LDPHittingTime:Eq:UpperBoundReducedToIntervalPart2} 
\limsup\limits_{t\rightarrow\infty}\dfrac{1}{t}\ln P^W\left(\dfrac{T^{vt}_{ct}}{t}> \alpha\right)\leq -(v-c)I\left(\dfrac{\alpha}{v-c}\right) \text{ when } \alpha> (v-c)\mu'(0) \ .
\end{equation}
Given (\ref{Thm:LDPHittingTime:Eq:UpperBoundReducedToIntervalPart1}) and (\ref{Thm:LDPHittingTime:Eq:UpperBoundReducedToIntervalPart2}), the large-deviation upper bound (\ref{Thm:LDPHittingTime:Eq:UpperBound}) for arbitrary closed set $G$ follows by partitioning $G$ and approximating it by such intervals.

To prove (\ref{Thm:LDPHittingTime:Eq:UpperBoundReducedToIntervalPart1}) we apply Chebyshev inequality, so that for any $\alpha>0$ and any $\eta\leq 0$:
$$\begin{array}{ll}
\limsup\limits_{t\rightarrow\infty} \dfrac{1}{t}\ln P^W\left(\dfrac{T_{ct}^{vt}}{t}<\alpha\right)
&\leq \limsup\limits_{t\rightarrow\infty}\dfrac{1}{t}\ln P^W\left(e^{\eta T^{vt}_{ct}}>e^{\eta \alpha t}\right)
\\
&\leq \limsup\limits_{t\rightarrow\infty}\dfrac{1}{t}\ln \left\{e^{-\eta \alpha t}\cdot E^W\left[e^{\eta T^{vt}_{ct}}\mathbf{1}_{\{T^{vt}_{ct}<+\infty\}}\right]\right\}
\\
& = -\eta \alpha +\limsup\limits_{t\rightarrow\infty}\dfrac{1}{t}H^t(t\eta)
\\
& = -\eta \alpha +(v-c)\mu(\eta) \ ,
\end{array}$$
where the last equality relies on Lemma \ref{Lm:LDPHittingTimeFreeEnergyExist}. This gives, for any $\eta \leq 0$, the estimate
$$\begin{array}{ll}
\limsup\limits_{t\rightarrow\infty} \dfrac{1}{t}\ln P^W\left(\dfrac{T^{vt}_{ct}}{t}<\alpha\right) &
\leq \inf\limits_{\eta\leq 0}(-\eta \alpha + (v-c)\mu(\eta))
\\
& = -\sup\limits_{\eta \leq 0}(\eta \alpha - (v-c)\mu(\eta))
\\
& = - (v-c)\sup\limits_{\eta\leq 0}\left(\eta\dfrac{\alpha}{v-c}-\mu(\eta)\right) \ .
\end{array}$$
When $\alpha\leq (v-c)\mu'(0)$, the supremum in the above estimate is achieved at some $\eta^*$ such that $\mu'(\eta^*)=\dfrac{\alpha}{v-c}\leq \mu'(0)$, implying $\eta^* \leq 0 \leq \eta_c$, so we have 
$$\limsup\limits_{t\rightarrow\infty}\dfrac{1}{t}\ln P^W\left(\dfrac{T^{vt}_{ct}}{t}<\alpha\right)\leq -(v-c)\sup\limits_{\eta\leq\eta_c}\left(\eta\dfrac{\alpha}{v-c}-\mu(\eta)\right) \ ,$$
which is (\ref{Thm:LDPHittingTime:Eq:UpperBoundReducedToIntervalPart1}) due to (\ref{Eq:EntropyHittingTime:Backward}). 

Next we show (\ref{Thm:LDPHittingTime:Eq:UpperBoundReducedToIntervalPart2}). We apply Chebyshev inequality again, so that for any $\alpha>0$ and any $\eta\geq 0$ we have

$$\begin{array}{ll}
\limsup\limits_{t\rightarrow\infty} \dfrac{1}{t}\ln P^W\left(\dfrac{T^{vt}_{ct}}{t}> \alpha\right) & \leq \limsup\limits_{t\rightarrow\infty}\dfrac{1}{t} \ln P^W\left(e^{\eta T^{vt}_{ct}}> e^{\eta \alpha t}\right)
\\
& \leq \limsup\limits_{t\rightarrow\infty}\dfrac{1}{t}\ln \left\{e^{-\eta \alpha t}\cdot E^W\left[e^{\eta T^{vt}_{ct}}\mathbf{1}_{\{T^{vt}_{ct}<+\infty\}}\right]\right\}
\\
& = -\eta \alpha +\limsup\limits_{t\rightarrow\infty}\dfrac{1}{t}H^t(t\eta)
\\
& = -\eta \alpha +(v-c)\mu(\eta) \ ,
\end{array}$$
which, in a similar fashion as the previous case, will give the estimate

$$\limsup\limits_{t\rightarrow\infty}\dfrac{1}{t}\ln P^W\left(\dfrac{T^{vt}_{ct}}{t}> \alpha\right)\leq -(v-c)\sup\limits_{\eta\geq 0}\left(\eta\dfrac{\alpha}{v-c}-\mu(\eta)\right) \ .$$

When $\alpha>(v-c)\mu'(0)$, the supremum in the above estimate is achieved at some $\eta^*$ such that $\mu'(\eta^*)=\dfrac{\alpha}{v-c}\geq \mu'(0)$, implying $0\leq \eta^*\leq \eta_c$, so we have
$$\limsup\limits_{t\rightarrow\infty}\dfrac{1}{t}\ln P^W\left(\dfrac{T^{vt}_{ct}}{t}>\alpha\right)\leq -(v-c)\sup\limits_{\eta\leq\eta_c}\left(\eta\dfrac{\alpha}{v-c}-\mu(\eta)\right) \ ,$$
which is (\ref{Thm:LDPHittingTime:Eq:UpperBoundReducedToIntervalPart2}) due to (\ref{Eq:EntropyHittingTime:Backward}). 

\noindent{\textit{Lower Bound.}} For the lower bound (\ref{Thm:LDPHittingTime:Eq:LowerBound}) it suffices to show that for any $u\in (0, +\infty)$, 
\begin{equation}\label{Thm:LDPHittingTime:Eq:LowerBoundReduceToLocalBall}
\lim\limits_{\delta\rightarrow 0} \liminf\limits_{t\rightarrow\infty} \dfrac{1}{t} \ln P^W\left(\dfrac{T_{ct}^{vt}}{t}\in B_{u,\delta}\right)\geq -L(u) \ , 
\end{equation}
where $B_{u,\delta}=(u-\delta, u+\delta)$.
Recall that $0<c<v$. Suppose $t$ is large, we find $i\in \mathbb{Z}$ such that 
$$i-1<ct\leq i <i+1<...<i+n-1\leq vt <i+n$$
where $n=n(v,c,t)$ is the number of integers between $ct$ and $vt$. This implies $n-1\leq (v-c)t<n+1$, which is saying 
\begin{equation}\label{Thm:LDPHittingTime:Eq:BoundOnt}
\dfrac{n}{v-c}-\dfrac{1}{v-c}<t<\dfrac{n}{v-c}+\dfrac{1}{v-c} \ .
\end{equation}
We then decompose $T^{vt}_{ct}$ as 
\begin{equation}\label{Thm:LDPHittingTime:Eq:TctvtDecomposeIntoFiniteSum}
T^{vt}_{ct}=T^{vt}_{i+n-1}+T^{i+n-1}_{i+n-2}+...+T^{i+1}_i+T^i_{ct} = T^{vt}_{i+n-1}+\sum\limits_{k=1}^{n-1} T^{i+k}_{i+k-1}+T^i_{ct} \ .
\end{equation}
Note that $0\leq T_{i+n-1}^{vt}+T_{ct}^i\leq T_{i+n-1}^{i+n}+T_{i-1}^i$. So 
\begin{equation}\label{Thm:LDPHittingTime:Eq:BoundOnTctvt}
\sum\limits_{k=1}^{n-1} T_{i+k-1}^{i+k}<T_{ct}^{vt}<\sum\limits_{k=0}^n T_{i+k-1}^{i+k} \ .
\end{equation}
Thus, we have

\begin{equation}\label{Thm:LDPHittingTime:Eq:LowerBoundDecomposeIntoTwoParts}
\begin{array}{ll}
& P^W\left(\dfrac{T^{vt}_{ct}}{t}\in B_{u, \delta}\right)
\\
= & P^W\left(t(u-\delta)<T^{vt}_{ct}<t(u+\delta)\right)
\\
\stackrel{{\rm(a)}}{\geq} & P^W\left(t(u-\delta)<\sum\limits_{k=1}^{n-1} T_{i+k-1}^{i+k}<\sum\limits_{k=0}^n T_{i+k-1}^{i+k}<t(u+\delta)\right)
\\
\stackrel{{\rm(b)}}{\geq} & P^W\left(t(u-\frac{\delta}{2})<\sum\limits_{k=1}^{n-1} T_{i+k-1}^{i+k}<t(u+\frac{\delta}{2})\right)\cdot P^W\left(0<T_{i+n-1}^{i+n}+T^i_{i-1}<t \frac{\delta}{2}\right)
\\
\stackrel{{\rm(c)}}{\geq} & P^W\left((\frac{n}{v-c}+\frac{1}{v-c})(u-\frac{\delta}{2})<\sum\limits_{k=1}^{n-1} T_{i+k-1}^{i+k}<(\frac{n}{v-c}-\frac{1}{v-c})(u+\frac{\delta}{2})\right)
\\
& \qquad \cdot P^W\left(0<T_{i+n-1}^{i+n}+T^i_{i-1}<(\frac{n}{v-c}-\frac{1}{v-c}) \frac{\delta}{2}\right) \ .
\end{array}
\end{equation}
Here (a) is due to (\ref{Thm:LDPHittingTime:Eq:BoundOnTctvt}), (b) is due to (\ref{Thm:LDPHittingTime:Eq:TctvtDecomposeIntoFiniteSum}), (c) is due to (\ref{Thm:LDPHittingTime:Eq:BoundOnt}). 

We will show that \begin{equation}\label{Thm:LDPHittingTime:Eq:LogarithmicTailBound}
\lim\limits_{n\rightarrow\infty}\dfrac{1}{n}\ln P^W\left(0<T_{i+n-1}^{i+n}+T^i_{i-1}<(\frac{n}{v-c}-\frac{1}{v-c}) \frac{\delta}{2}\right)=0 \ .
\end{equation}

\noindent Combining (\ref{Thm:LDPHittingTime:Eq:LowerBoundDecomposeIntoTwoParts}) and (\ref{Thm:LDPHittingTime:Eq:LogarithmicTailBound}) we see that
\begin{equation}\label{Thm:LDPHittingTime:Eq:LowerBoundReducedToOnlyInTermsOfFiniteSum}
\lim\limits_{\delta\rightarrow 0}\liminf\limits_{t\rightarrow\infty}\dfrac{1}{t}\ln P^W\left(\dfrac{T^{vt}_{ct}}{t}\in B_{u,\delta}\right)\geq (v-c)\lim\limits_{\delta\rightarrow 0}\liminf\limits_{n\rightarrow\infty}\dfrac{1}{n}\ln P^W\left(\dfrac{1}{n}\sum\limits_{k=1}^{n-1} T_{i+k-1}^{i+k}\in B_{\frac{u}{v-c}, \delta}\right) \ ,
\end{equation}
so for (\ref{Thm:LDPHittingTime:Eq:LowerBoundReduceToLocalBall}) it suffices to show that 
\begin{equation}\label{Thm:LDPHittingTime:Eq:LowerBoundOnlyInTermsOfFiniteSum}
\lim\limits_{\delta\rightarrow 0}\liminf\limits_{n\rightarrow\infty}\dfrac{1}{n}\ln P^W\left(\dfrac{1}{n}\sum\limits_{k=1}^{n-1} T_{i+k-1}^{i+k}\in B_{u, \delta}\right)\geq -I(u) \ .
\end{equation}

So it remains to show (\ref{Thm:LDPHittingTime:Eq:LogarithmicTailBound}). Let $\eta<0$, and consider splitting the event $$\left\{e^{\eta(T_{i+n-1}^{i+n}+T_{i-1}^i)}\mathbf{1}_{T_{i+n-1}^{i+n}+T_{i-1}^i<+\infty}\right\}$$ into two parts depending on whether or not $T_{i+n-1}^{i+n}+T_{i-1}^i>(\frac{n}{v-c}-\frac{1}{v-c})\frac{\delta}{2}$. This gives
$$\begin{array}{ll}
& E^W\left(e^{\eta(T_{i+n-1}^{i+n}+T_{i-1}^i)}\mathbf{1}_{T_{i+n-1}^{i+n}+T_{i-1}^i<+\infty}\right)
\\
\leq & P^W\left(T_{i+n-1}^{i+n}+T_{i-1}^i\leq (\frac{n}{v-c}-\frac{1}{v-c})\frac{\delta}{2}\right) 
+
e^{\eta(\frac{n}{v-c}-\frac{1}{v-c})\frac{\delta}{2}}\cdot P^W\left(T_{i+n-1}^{i+n}+T_{i-1}^i> (\frac{n}{v-c}-\frac{1}{v-c})\frac{\delta}{2}\right) \ .
\end{array}$$
So we have the estimate
\begin{equation}\label{Thm:LDPHittingTime:Eq:TailHittingTimeProbabilityEstimate1}
\begin{array}{ll}
& P^W\left(0<T_{i+n-1}^{i+n}+T_{i-1}^i< (\frac{n}{v-c}-\frac{1}{v-c})\frac{\delta}{2}\right) 
\\
\geq & 
E^W\left(e^{\eta(T_{i+n-1}^{i+n}+T_{i-1}^i)}\mathbf{1}_{T_{i+n-1}^{i+n}+T_{i-1}^i<+\infty}\right) -
e^{\eta(\frac{n}{v-c}-\frac{1}{v-c})\frac{\delta}{2}} P^W\left(T_{i+n-1}^{i+n}+T_{i-1}^i> (\frac{n}{v-c}-\frac{1}{v-c})\frac{\delta}{2}\right)
\\
\geq & E^W\left(e^{\eta(T_{i+n-1}^{i+n}+T_{i-1}^i)}\mathbf{1}_{T_{i+n-1}^{i+n}+T_{i-1}^i<+\infty}\right) - e^{\eta(\frac{n}{v-c}-\frac{1}{v-c})\frac{\delta}{2}} \ .
\end{array}
\end{equation}

\noindent By the ergodic theorem, $\mathbf{P}$-almost surely, we have
$$\lim\limits_{n\rightarrow\infty}\dfrac{1}{n}\sum\limits_{k=1}^{n-1}\ln E^W \left(e^{\eta T_{i+n-1}^{i+n}}\mathbf{1}_{T_{i+n-1}^{i+n}<+\infty}\right)=\mathbf{E}\left[\ln E^W \left(e^{\eta T_0^1}\mathbf{1}_{T_0^1<+\infty}\right)\right]=\mu(\eta) \ ,$$
which is finite by Lemma \ref{Lm:LDPHittingTimeFreeEnergyExist}. This gives, $\mathbf{P}$-almost surely
$$\lim\limits_{n\rightarrow\infty}\dfrac{1}{n}\ln E^W\left(e^{\eta(T_{i+n-1}^{i+n}+T_{i-1}^i)}\mathbf{1}_{\{T_{i+n-1}^{i+n}+T_{i-1}^i<+\infty\}}\right) = 0 \ .$$ In other words, $\mathbf{P}$-almost surely, for any $\varepsilon>0$ and $n$ large enough we have $$E^W\left(e^{\eta(T_{i+n-1}^{i+n}+T_{i-1}^i)}\mathbf{1}_{\{T_{i+n-1}^{i+n}+T_{i-1}^i<+\infty\}}\right)\geq e^{-\varepsilon n} \ .$$ So if $0<\varepsilon<-\eta (\frac{n}{v-c}-\frac{1}{v-c})\frac{\delta}{2}$ by (\ref{Thm:LDPHittingTime:Eq:TailHittingTimeProbabilityEstimate1}) we will have 
$$\begin{array}{ll}
0 & \geq \liminf\limits_{n\rightarrow\infty} \dfrac{1}{n}\ln P^W\left(0<T_{i+n-1}^{i+n}+T_{i-1}^i< (\frac{n}{v-c}-\frac{1}{v-c})\frac{\delta}{2}\right)
\\
& \geq \liminf\limits_{n\rightarrow\infty}\dfrac{1}{n}\ln(e^{-\varepsilon n}-e^{\eta (\frac{n}{v-c}-\frac{1}{v-c})\frac{\delta}{2}}) \geq -\eta \delta \ ,
\end{array}$$
which implies (\ref{Thm:LDPHittingTime:Eq:LogarithmicTailBound}).

Now we turn to the proof of (\ref{Thm:LDPHittingTime:Eq:LowerBoundOnlyInTermsOfFiniteSum}). Define the event $\mathcal{Y}_{n,\delta}=\left\{\dfrac{1}{n}\sum\limits_{k=1}^{n-1} T_{i+k-1}^{i+k}\in B_{u, \delta}\right\}$. Let 
$$\mathcal{T}_0=0 \ , \ \mathcal{T}_\ell=\sum\limits_{k=1}^{\ell}T^{i+k}_{i+k-1} \ , \ \ell=1,2,...,n-1 \ .$$
For any $M>0$, we define the event
$$A_n^M = \bigcap\limits_{1\leq k\leq n-1}\{T_{i+k-1}^{i+k}\leq M\} \ .$$
Let $P^{W,M}$ stand for the conditional probablity of $P^W$ given $A_n^M$, with the corresponding expectation written as $E^{W,M}$. Then we get
\begin{equation}\label{Thm:LDPHittingTime:Eq:DecomposeNeighborHoodProbIntoConditionalMPlusTail}
\dfrac{1}{n}\ln P^W(\mathcal{Y}_{n,\delta})=\dfrac{1}{n}\ln P^{W,M}(\mathcal{Y}_{n,\delta})+\dfrac{1}{n}\ln P^W(A_n^M) \ , \ \text{ for all } \delta>0 \ .
\end{equation}

\noindent Given $M>0$, set
\begin{equation}\label{Thm:LDPHittingTime:Eq:muM}
\mu^M(\eta)=\mathbf{E}\left[\ln E^W\left(e^{\eta T_0^1}\mathbf{1}_{\{T_0^1<M\}}\right)\right] \ , 
\end{equation}
and 
\begin{equation}\label{Thm:LDPHittingTime:Eq:muCondM}
\mu^{|M}(\eta) = \mathbf{E}\left[\ln E^{W,M}\left(e^{\eta T_0^1}\mathbf{1}_{\{T_0^1<M\}}\right)\right]
\ .
\end{equation}

\noindent We consider the following equation
\begin{equation}\label{Thm:LDPHittingTime:Eq:EquationForetauM}
\mathbf{E}\left[
\dfrac{E^{W,M}(T_0^1e^{\eta_u^M T_0^1}\mathbf{1}_{T_0^1<\infty})}{E^{W,M}(e^{\eta_u^M T_0^1}\mathbf{1}_{T_0^1<\infty})}\right]=u \ .
\end{equation}
By Lemma \ref{Lm:PropertiesMuAndTruncatedMu} part (2) we know that (\ref{Thm:LDPHittingTime:Eq:EquationForetauM}) is just $(\mu^M)'(\eta_u^M)=u$. By part (3) of that Lemma, we know that the solution satisfies $\limsup\limits_{M\rightarrow \infty} \eta_u^M \leq \eta_c$. 

Given the solution $\eta_u^M$, we make the following change of measure

\begin{equation}\label{Thm:LDPHittingTime:Eq:ChangeOfMeasure}
\dfrac{d\widehat{P}^{W,M}}{d P^{W,M}}=\dfrac{e^{\eta_u^M\mathcal{T}_{n-1}}}{E^{W,M}[e^{\eta_u^M \mathcal{T}_{n-1}}]} \ .
\end{equation}

\noindent Then we can estimate 

\begin{equation}\label{Thm:LDPHittingTime:Eq:LowerBoundOfNeighborhoodUsingChangedMeasure:1}
\begin{array}{ll}
& \dfrac{1}{n}\ln P^{W,M}(\mathcal{Y}_{n,\delta})
\\
= & \dfrac{1}{n}\ln \displaystyle{\int_{\mathcal{Y}_{n,\delta}} e^{-\eta_u^M \mathcal{T}_{n-1}}\cdot E^{W,M}[e^{\eta_u^M \mathcal{T}_{n-1}}]{\rm d}\widehat{P}^{W,M}}
\\
= & \dfrac{1}{n}\ln E^{W,M}[e^{\eta_u^M \mathcal{T}_{n-1}}] + \dfrac{1}{n}\ln\displaystyle{\int_{\mathcal{Y}_{n,\delta}}e^{-\eta_u^M\mathcal{T}_{n-1}}{\rm d}\widehat{P}^{W,M}}
\\
\geq & 
\left\{
\begin{array}{ll}
\dfrac{1}{n}\ln E^{W,M}[e^{\eta_u^M \mathcal{T}_{n-1}}]-\eta_u^M(u+\delta)+\dfrac{1}{n}\ln \widehat{P}^{W,M}(\mathcal{Y}_{n,\delta}) \ , & \text{ if } \eta_u^M>0 \ ;
\\
\dfrac{1}{n}\ln E^{W,M}[e^{\eta_u^M \mathcal{T}_{n-1}}]-\eta_u^M(u-\delta)+\dfrac{1}{n}\ln \widehat{P}^{W,M}(\mathcal{Y}_{n,\delta}) \ , & \text{ if } \eta_u^M\leq 0 \ .
\end{array}
\right.
\end{array}
\end{equation}
Without loss of generality below we will consider only the case $\eta_u^M>0$, and the case when $\eta_u^M\leq 0$ can be argued similarly, since we finally will send $\delta\rightarrow 0$ (see (\ref{Thm:LDPHittingTime:Eq:LowerBoundNeighborhoodUsingIM})). 

Since $X^x(t)$ is a strong Markov process, the independence of $\{T_{i+k-1}^{i+k}\}_{k=1}^{n-1}$ yields, $\mathbf{E}$ almost surely,
\begin{equation}\label{Thm:LDPHittingTime:Eq:LowerBoundOfNeighborhoodUsingChangedMeasure:2}
\begin{array}{ll}
& \dfrac{1}{n}\ln E^{W,M}[e^{\eta_u^M \mathcal{T}_{n-1}}]
= \dfrac{1}{n} \ln E^{W,M} [e^{\eta_u^M\sum\limits_{k=1}^{n-1} T^{i+k}_{i+k-1}}]
= \dfrac{1}{n}\sum\limits_{k=1}^{n-1} \ln E^{W,M} [e^{\eta_u^M T_{i+k-1}^{i+k}}]
\\
\stackrel{n\rightarrow \infty}{\longrightarrow} &  \mathbf{E}\left[\left.\ln E^W\left(e^{\eta_n^M T_0^1}\mathbf{1}_{\{T_0^1<M\}}\right)\right|T_0^1<M\right] = \mu^M(\eta_u^M) - \mathbf{E}\left[\ln P^W\left(T_0^1<M\right)\right] \ .
\end{array}
\end{equation}
Similarly we have 
\begin{equation}\label{Thm:LDPHittingTime:Eq:LowerBoundOfNeighborhoodUsingChangedMeasure:3}
\dfrac{1}{n}\ln P^W(A_n^M)
= \dfrac{1}{n} \sum\limits_{k=1}^{n-1} \ln P^W\left(T_{i+k-1}^{i+k}\leq M\right)
\stackrel{n \rightarrow \infty}{\longrightarrow}  \mathbf{E}\left[\ln P^W\left(T_0^1<M\right)\right] \ .
\end{equation}

\noindent Combining (\ref{Thm:LDPHittingTime:Eq:DecomposeNeighborHoodProbIntoConditionalMPlusTail}), (\ref{Thm:LDPHittingTime:Eq:LowerBoundOfNeighborhoodUsingChangedMeasure:1}), (\ref{Thm:LDPHittingTime:Eq:LowerBoundOfNeighborhoodUsingChangedMeasure:2}), (\ref{Thm:LDPHittingTime:Eq:LowerBoundOfNeighborhoodUsingChangedMeasure:3}) we get
\begin{align}\label{Thm:LDPHittingTime:Eq:LowerBoundOfNeighborhoodUsingChangedMeasure:Last}
 \dfrac{1}{n}\ln P^W(\mathcal{Y}_{n,\delta})
&\geq  \mu^M(\eta_u^M)-\mathbf{E}[\ln P^W\left(T_0^1<M\right)]-\eta_u^M(u+\delta)
\cr & \quad+\dfrac{1}{n}\ln \widehat{P}^{W,M}(\mathcal{Y}_{n,\delta})+\mathbf{E}\left[\ln P^W\left(T_0^1<M\right)\right]
\cr
& = \mu^M(\eta_u^M)-\eta_u^M(u+\delta)+\dfrac{1}{n}\ln \widehat{P}^{W,M}(\mathcal{Y}_{n,\delta}) \ .
\end{align}

Next we prove that, $\mathbf{P}$-almost surely, for any $\delta>0$, we have
\begin{equation}\label{Thm:LDPHittingTime:Eq:ChangedMeasureTendsToOne} \lim\limits_{n\rightarrow \infty}\widehat{P}^{W,M}(\mathcal{Y}_{n,\delta})=1 \ .
\end{equation}
First, by the ergodic theorem, we have
$$\dfrac{1}{n}\widehat{E}^{W,M}\left(\sum\limits_{k=1}^{n-1} T_{i+k-1}^{i+k}\right)=
\dfrac{1}{n}\sum\limits_{k=1}^{n-1}\widehat{E}^{W,M}( T_{i+k-1}^{i+k}) \stackrel{n\rightarrow\infty}{\longrightarrow} \mathbf{E}\left(\widehat{E}^{W,M}[T_0^1]\right)=u$$
by the definition of $\eta_u^M$ and the changed measure $\widehat{P}^{W,M}$. Moreover
$$\begin{array}{ll} & \widehat{E}^{W,M}\left[\left(\dfrac{1}{n}\sum\limits_{k=1}^{n-1} T_{i+k-1}^{i+k} - \widehat{E}^{W,M}\left(\dfrac{1}{n}\sum\limits_{k=1}^{n-1} T_{i+k-1}^{i+k}\right)\right)^4\right]
\\
= & \widehat{E}^{W,M}\left[\left(\dfrac{1}{n}\sum\limits_{k=1}^{n-1}\left( T_{i+k-1}^{i+k} - \widehat{E}^{W,M}[T_{i+k-1}^{i+k}]\right)\right)^4\right]
\leq  \dfrac{3M^4}{n^2} \ ,
\end{array}$$
where we have used the independence of $\{T_{i+k-1}^{i+k}\}_{k=1}^{n-1}$ under $\widehat{P}^{W,M}$. So the conclusion comes from Borel-Cantelli Lemma.

In summary of (\ref{Thm:LDPHittingTime:Eq:LowerBoundOfNeighborhoodUsingChangedMeasure:Last}), (\ref{Thm:LDPHittingTime:Eq:ChangedMeasureTendsToOne}), it follows that $\mathbf{P}$-a.s. we have 
\begin{align}\label{Thm:LDPHittingTime:Eq:LowerBoundNeighborhoodUsingIM}
\lim\limits_{\delta\rightarrow 0} \liminf\limits_{n\rightarrow\infty} \dfrac{1}{n}\ln P^W(\mathcal{Y}_{n,\delta}) & \geq -\eta_u^M u + \mu^M(\eta_u^M)
 \cr& \geq -\sup\limits_{\eta\in (-\infty, +\infty)} (\eta u -\mu^M(\eta))\stackrel{\text{def}}{=}-I_M(u) \ .
\end{align}
Let $I^*(u)=\limsup\limits_{M\rightarrow \infty} I_M(u)$. Then by the definition of $I_M(u)$ we know that $$I_M(u)=\sup\limits_{\eta\in (-\infty, +\infty)}(\eta u - \mu^M(\eta))\geq -\mu^M(0)\geq 0 \ .$$
This gives $I^*(u)\geq 0$. Moreover, since the sup in $I_M(u)$ is achieved at $\eta_u^M$, and we have $\limsup\limits_{M \rightarrow \infty}\eta_u^M\leq \eta_c$ by part (3) of Lemma \ref{Lm:PropertiesMuAndTruncatedMu}, this means for large $M>0$ we have $I_M(u)<+\infty$ and so is $I^*(u)$. Hence, the level sets $\{\eta: \eta u-\mu^M(\eta)\geq I^*(u)\}$ are non-empty, compact, nested sets implying that their intersection as $M\rightarrow+\infty$ contains some $\eta^*\in (-\infty, \eta_c]$. Since $\mu^M(\eta)$ is increasing in $M$ for every fixed $\eta\leq \eta_c$, by Lebesgue's monotone convergence this gives
$$-\mu(\eta^*)=-\lim\limits_{M\rightarrow+\infty} \mu^M(\eta^*) \geq -\eta^* u + I^*(u) \ .$$
Thus \begin{equation}\label{Thm:LDPHittingTime:Eq:IStarEstimateI}
-I^*(u)\geq -(\eta^* u - \mu(\eta^*))\geq -\sup\limits_{\eta\leq \eta_c}(\eta u - \mu(\eta))=-I(u) \ .
\end{equation}
The desired (\ref{Thm:LDPHittingTime:Eq:LowerBoundOnlyInTermsOfFiniteSum}) comes from combining (\ref{Thm:LDPHittingTime:Eq:LowerBoundNeighborhoodUsingIM}) and (\ref{Thm:LDPHittingTime:Eq:IStarEstimateI}).
\end{proof}

\begin{lemma}[Properties of $\mu(\eta)$, $\mu^M(\eta)$ and $\mu^{|M}(\eta)$]\label{Lm:PropertiesMuAndTruncatedMu}
Let $\mu(\eta)$ be defined in \emph{(\ref{Eq:LyapunovFunctionHittingTime:Backward})}, $\mu^M(\eta)$ be defined in \emph{(\ref{Thm:LDPHittingTime:Eq:muM})}, and $\mu^{|M}(\eta)$ be defined in \emph{(\ref{Thm:LDPHittingTime:Eq:muCondM})}. Then 
\begin{itemize}
\item[\emph{(1)}] $\mu(\eta)$, $\mu'(\eta)$ exists when $\eta<\eta_c$ and $\mu(\eta)=+\infty$ when $\eta>\eta_c$. When $\eta<\eta_c$ we have
$$\mu'(\eta)=\mathbf{E}\left[
\dfrac{E^{W}(T_0^1e^{\eta T_0^1}\mathbf{1}_{T_0^1<\infty})}{E^{W}(e^{\eta T_0^1}\mathbf{1}_{T_0^1<\infty})}\right].$$ Moreover, $\mu(\eta)$ is convex for $\eta<\eta_c$, and $\mu'(\eta)$ is strictly increasing for $\eta<\eta_c$;
\item[\emph{(2)}] For every fixed $M>0$, $\mu^M(\eta)$, $\mu^{|M}(\eta)$ are well defined for all $\eta \in (-\infty, +\infty)$, with
$$\mu^{|M}(\eta) = \mu^M(\eta)-\mathbf{E}\left[\ln P^W\left(T_0^1<M\right)\right]
= \mathbf{E}\left[\ln E^W\left(\left.e^{\eta T_0^1}\mathbf{1}_{\{T_0^1<M\}}\right| T_0^1<M\right)\right] \ ,$$
and $$(\mu^M)'(\eta)=(\mu^{|M})'(\eta)=\mathbf{E}\left[
\dfrac{E^W(T_0^1e^{\eta T_0^1}\mathbf{1}_{T_0^1<M})}{E^W(e^{\eta T_0^1}\mathbf{1}_{T_0^1<M})}\right]=\mathbf{E}\left[
\dfrac{E^{W,M}(T_0^1e^{\eta T_0^1}\mathbf{1}_{T_0^1<\infty})}{E^{W,M}(e^{\eta T_0^1}\mathbf{1}_{T_0^1<\infty})}\right] \ ,$$
where $E^{W,M}$ is the conditional expectation corresponding to $P^{W,M}$ defined right before \emph{(\ref{Thm:LDPHittingTime:Eq:DecomposeNeighborHoodProbIntoConditionalMPlusTail})}. Moreover, $\mu^M(\eta)$, $\mu^{|M}(\eta)$ are convex functions in $\eta$, and $(\mu^M)'(\eta)=(\mu^{|M})'(\eta)$ are strictly increasing in $\eta\in (-\infty, \infty)$;
\item[\emph{(3)}] For every fixed $\eta<\eta_c$ we have $\lim\limits_{M\rightarrow \infty} \mu^M(\eta)=\mu(\eta)$; For every fixed $\eta>\eta_c$ we have $\lim\limits_{M\rightarrow\infty}\mu^M(\eta)=+\infty$. Moreover, the solution $\eta_u^M$ to the equation $(\mu^M)'(\eta_u^M)=u$
has the limit bound $\limsup\limits_{M\rightarrow\infty}\eta_u^M \leq \eta_c$.
\end{itemize}
\end{lemma}

\begin{proof}
(1) This follows from the same argument in the proof of Lemma 2.2 in \cite{NolenXinDCDS}.

(2) Since $M>0$ is fixed, it is easy to see from (\ref{Eq:LyapunovFunctionHittingTime:Backward}) that $$|\mu^M(\eta)|\leq
\mathbf{E}\left[\ln \left|E^W\left(e^{\eta T_0^1}\mathbf{1}_{\{T_0^1<M\}}\right)\right|\right] \leq |\eta|\cdot M \ ,$$ and so it is well defined for all $\eta\in (-\infty, \infty)$. Then we have 
$$\mu^{|M}(\eta)=
\mathbf{E}\left[\ln E^W\left(\left.e^{\eta T_0^1}\mathbf{1}_{\{T_0^1<+\infty\}}\right| T_0^1<M\right)\right]=
\mathbf{E}\left[\ln \dfrac{E^W\left(e^{\eta T_0^1}\mathbf{1}_{\{T_0^1<M\}}\right)}{P^W\left(T_0^1<M\right)}\right] \ ,$$ 
so that $\mu^{|M}(\eta)=\mu^M(\eta)-\mathbf{E}\left[\ln P^W\left(T_0^1<M\right)\right]$ is also well defined for all $\eta\in (-\infty, \infty)$. The last equality also gives $(\mu^M)'(\eta)=(\mu^{|M})'(\eta)$. Lastly, using the same arguments as in part (vi) of \cite[Lemma 2.2]{NolenXinDCDS}, based on dominated convergence and the chain rule, we can show 

$$(\mu^M)'(\eta)=\mathbf{E}\left[
\dfrac{E^W(T_0^1e^{\eta T_0^1}\mathbf{1}_{T_0^1<M})}{E^W(e^{\eta T_0^1}\mathbf{1}_{T_0^1<M})}\right] \ , \ (\mu^{|M})'(\eta)=\mathbf{E}\left[
\dfrac{E^{W,M}(T_0^1e^{\eta T_0^1}\mathbf{1}_{T_0^1<\infty})}{E^{W,M}(e^{\eta T_0^1}\mathbf{1}_{T_0^1<\infty})}\right] \ ,$$
as well as both $\mu^M(\eta)$, $\mu^{|M}(\eta)$ are convex with strictly increasing derivatives $(\mu^M)'(\eta)$, $(\mu^{|M})'(\eta)$ for all $\eta\in (-\infty, \infty)$.

(3) We can calculate $$\begin{array}{ll}
\mu^M(\eta)-\mu(\eta)& = \mathbf{E}\left[\ln E^W\left(e^{\eta T_0^1}\mathbf{1}_{\{T_0^1<M\}}\right)\right]-\mathbf{E}\left[\ln E^W\left(e^{\eta T_0^1}\mathbf{1}_{\{T_0^1<+\infty\}}\right)\right]
\\
& = \mathbf{E}\left[\ln\dfrac{ E^W\left(e^{\eta T_0^1}\mathbf{1}_{\{T_0^1<M\}}\right)}{E^W\left(e^{\eta T_0^1}\mathbf{1}_{\{T_0^1<+\infty\}}\right)}\right]
\end{array} \ ,$$
and so $\lim\limits_{M\rightarrow\infty}\mu^M(\eta)=\mu(\eta)$ for every fixed $\eta<\eta_c$. When $\eta>\eta_c$, we know by (\ref{Eq:CriticalEta:Backward}) that $\mu(\eta)=+\infty$, which gives $\lim\limits_{M\rightarrow\infty}\mu^M(\eta)=\lim\limits_{M\rightarrow\infty}\mathbf{E}\left[\ln E^W\left(e^{\eta T_0^1}\mathbf{1}_{\{T_0^1<M\}}\right)\right]=\mathbf{E}\left[\ln E^W\left(e^{\eta T_0^1}\mathbf{1}_{\{T_0^1<+\infty\}}\right)\right]=+\infty$. 

Let us now consider the solution $\eta_u^M$ to the equation $(\mu^M)'(\eta_u^M)=u$. Fix some small $\delta>0$. Then for any $U>0$, there exists some $M_0=M_0(\delta, U)$ such that $(\mu^M)'(\eta_c+\delta)>U$ whenever $M\geq M_0$. Since $(\mu^M)'(\eta)$ is strictly increasing in $\eta$, we know that for any $\eta\geq \eta_c+\delta$, the inequality $(\mu^M)'(\eta)>U$ still hold. This means when we pick $U>u$ and when $M>M_0(U, \delta)$, we must have $\eta^M_u\leq \eta_c+\delta$. Since $\delta$ is arbitrary, we have $\limsup\limits_{M\rightarrow\infty} \eta_u^M\leq\eta_c$ as desired.
\end{proof}

\noindent{\textbf{The case of forward hitting time.}}

Recall that we set $0<c<v$. Similarly as (\ref{Eq:HittingTimeSDE-RandomDrift:Backward}), we define, for $s<r$, the forward hitting time
\begin{equation}\label{Eq:HittingTimeSDE-RandomDrift:Forward}
T^s_r=\inf\{t>0: X^s(t)\geq r\} \ .
\end{equation}
Now we aim to obtain the LDP of $\dfrac{T^{ct}_{vt}}{t}\mathbf{1}_{\{T^{ct}_{vt}<\infty\}}$
as $t\rightarrow\infty$. Similarly as (\ref{Eq:LyapunovFunctionHittingTime:Backward}), (\ref{Eq:CriticalEta:Backward}), we define the Lyapunov function

\begin{equation}\label{Eq:LyapunovFunctionHittingTime:Forward}
\mu^{\fwd}(\eta)=\mathbf{E}\left[\ln E^W\left(e^{\eta T^{-1}_0}\mathbf{1}_{\{T^{-1}_0<+\infty\}}\right)\right]
\end{equation}
\footnote{We use the superscript $\fwd$ to stand for forward. Unlike the backward case, the superscript $\fwd$ is in general \textit{not} suppressed.} and the critical value
\begin{equation}\label{Eq:CriticalEta:Forward}
\eta_c^{\fwd}=\sup\{\eta\in \mathbb{R}: \ \mu^{\fwd}(\eta)<+\infty\} \ .
\end{equation}
Define the Legendre transform of $\mu^\fwd(\eta)$ as
\begin{equation}\label{Eq:EntropyHittingTime:Forward}
I^\fwd(a)=\sup\limits_{\eta\leq \eta_c^\fwd}(a\eta-\mu^\fwd(\eta)) \ .
\end{equation}

For any $\eta\in \mathbb{R}$ we define

\begin{equation}\label{Eq:LDPHittingTimeFreeEnergyUnscaled:Forward}
H^{t,\fwd}(\eta)=\ln E^W \exp\left(\eta \cdot \dfrac{T^{ct}_{vt}}{t}\mathbf{1}_{\{T^{ct}_{vt}<\infty\}}\right) \ ,
\end{equation}
and 
\begin{equation}\label{Eq:LDPHittingTimeFreeEnergy:Forward}
H^\fwd(\eta)=\lim\limits_{t\rightarrow\infty}\dfrac{1}{t}H^{t,\fwd}(t\eta) \ .
\end{equation}
Then similarly as Lemma \ref{Lm:LDPHittingTimeFreeEnergyExist} we have the following
\begin{lemma}[Lyapunov Exponent Identity for the Forward Hitting Time]\label{Lm:LDPHittingTimeFreeEnergyExist:Forward}
Let $\eta\in\mathbb{R}$ be such that 
\begin{equation}\label{Lm:LDPHittingTimeFreeEnergyExist:Forward:AbsoluteIntegrabilityCondition}
\mathbf{E}\left[\left|\ln E^{W}\left(e^{\eta T_0^{-1}}\mathbf{1}_{\{T_0^{-1}<\infty\}}\right)\right|\right]<\infty \ .
\end{equation}
Assume $0<c<v$. Then $\mathbf{P}$-almost surely the limit \emph{(\ref{Eq:LDPHittingTimeFreeEnergy:Forward})} holds and we have $$H^\fwd(\eta)=(v-c)\mu^\fwd(\eta) \ ,$$
where $\mu^\fwd(\eta)$ is defined in \emph{(\ref{Eq:LyapunovFunctionHittingTime:Forward})}.  
Thus $H^\fwd(\eta)<\infty$ when $\eta<\eta_c^\fwd$ and $H^\fwd(\eta)=\infty$ when $\eta>\eta_c^\fwd$, where $\eta_c^\fwd$ is defined in \emph{(\ref{Eq:CriticalEta:Forward})}. 
\end{lemma}

Using Corollary \ref{Corrollary:MuBwdMFwdDifferByConstant} that we will prove in the next Section, we can obtain a few properties of $\mu^\fwd$ that are in parallel to $\mu$ as in Lemma \ref{Lm:PropertiesMuAndTruncatedMu}:

\begin{lemma}\label{Lm:PropertiesMuFwd}
The function $\mu^\fwd(\eta)$ and its corresponding truncated versions have the same properties as $\mu(\eta)$, which are demonstrated in \emph{Lemma \ref{Lm:PropertiesMuAndTruncatedMu}}.
\end{lemma}

Therefore using the same argument as we did in the proof of Theorem \ref{Thm:LDPHittingTime}, we can derive, in parallel with Theorem \ref{Thm:LDPHittingTime}, the following

\begin{theorem}[LDP for the Hitting Time in the Forward case] \label{Thm:LDPHittingTime:Forward}
Let $0<c<v$. Then $\mathbf{P}$-almost surely the following two estimates hold: 

\begin{itemize}
\item[\emph{(a)}] For
any closed set $G\subset (0,+\infty)$ we have
\begin{equation}\label{Thm:LDPHittingTime:Forward:Eq:UpperBound}
\limsup\limits_{t\rightarrow\infty}\dfrac{1}{t}\ln P^W
\left(\dfrac{T_{vt}^{ct}}{t}\in G\right)\leq -(v-c)\inf\limits_{a\in G}I^\fwd\left(\dfrac{a}{v-c}\right) \ ;
\end{equation}
\item[\emph{(b)}] For any open set $F\subset (0,+\infty)$ we have
\begin{equation}\label{Thm:LDPHittingTime:Forward:Eq:LowerBound}
\liminf\limits_{t\rightarrow\infty}\dfrac{1}{t}\ln P^W
\left(\dfrac{T_{vt}^{ct}}{t}\in F\right)\geq -(v-c)\inf\limits_{a\in F}I^\fwd\left(\dfrac{a}{v-c}\right) \ ,
\end{equation}
\end{itemize}
\end{theorem}

\begin{proof}
The proof follows the same arguments as in Theorem \ref{Thm:LDPHittingTime} and we omit it.
\end{proof}

\subsection{Relation between $\mu^\fwd(\eta)$ and $\mu^\bwd(\eta)$}\label{Sec:LDP:RelationMuAndMufwd}
This subsection is dedicated to an investigation of the relation between $\mu^\bwd(\eta)$ and $\mu^\fwd(\eta)$. In particular, we will show (see Lemma \ref{Lm:BwdAndFwdCriticalEtaForMuAreTheSame}) that $\eta^\bwd=\eta^\fwd$, as well as the fact that $\mu^\bwd$ and $\mu^\fwd$ only differ by a constant (see Proposition \ref{Prop:MuBwdMinusMuFwd}, Lemma \ref{Lm:ExpectationLnRhoIsConstant} and Corollary \ref{Corrollary:MuBwdMFwdDifferByConstant}).

According to \cite[equation (5.61) in Chapter 5]{KSbook}, one has
\begin{align}
\label{eqn: disT12}
P^W\left(T_2^1<T_0^1\right)=\frac{\displaystyle{\int_0^1 e^{-2\int_0^y b(s){\rm d}s}{\rm d}y}}{\displaystyle{\int_0^2 e^{-2\int_0^y b(s){\rm d} s}{\rm d}y}}\ ,\quad
P^W\left(T_0^1<T_2^1\right)=\frac{\displaystyle{\int_1^2 e^{-2\int_0^y b(s){\rm d}s}{\rm d}y}}{\displaystyle{\int_0^2 e^{-2\int_0^y b(s){\rm d} s}{\rm d}y}}\ .
\end{align}

In general, 
\[
P^{W}(T_0^1<\infty)=1 \text{ is equivalent to }\int^{\infty}_0 e^{-2\int_0^y b(s){\rm d}s}{\rm d}y=\infty\ ,
\]
and 
\[
P^{W}(T^0_1<\infty)=1  \text{ is equivalent to }\int_{-\infty}^0 e^{-2\int_0^y b(s){\rm d} s}{\rm d} y=\infty\ .
\]
Note that since $b$ is stationary and ergodic, then $ {\mathbf P}$-a.s,
\[
\lim_{y\rightarrow\infty}\frac{1}{y}\int_0^y b(s) {\rm d} s={\mathbf E}[b]\ .
\]
Thus when ${\mathbf E}[b]>0$, one always has 
\begin{equation}
    \label{eqn: generalintecond}
{\mathbf P}\left(\int_{-\infty}^0 e^{-2\int_0^y b(s){\rm d}s}{\rm d} y=\infty\right)=1  \text{ and }
{\mathbf P}\left(\int_0^{\infty}  
e^{-2\int_0^y b(s){\rm d}s}{\rm d} y<\infty\right)=1 \ .
\end{equation}
and hence, ${\mathbf P}$-a.s,
\begin{align}\label{eqn:probilityT_01}
 P^W(T_1^0<\infty)=1  \text{ and }P^W(T_0^1<\infty)<1\ . 
\end{align}
See Section 5.5 in \cite{KSbook} for more details.  

We remark here that in the following rest part of this Section, local integrability of $b$ would be enough to guarantee all the results we derive. This means that for the rest of Section \ref{Sec:LDP:RelationMuAndMufwd} in particular, our original Assumption \eqref{Eq:AssumptionOnDrift:Boundedness} can be weakened to local integrability of $b$. However, Assumption \eqref{Eq:AssumptionOnDrift:Boundedness} may still be needed in other Sections, such as Section \ref{Sec:LDP:ProcessX} (see Lemma \ref{Lm:WorstCaseEnvironmentEstimateForaInThmLDPProcessX}).

\begin{lemma}\label{lem: basiccomp}
 Assume  ${\mathbf E}\left[\displaystyle{\int_0^1 |b(s)|{\rm d}s}\right]<\infty.$ Then
\[
0> {\mathbf E}[\ln  P^W\left({T_2^1<T_0^1}\right)]>-\infty\ .
\]
  and
   \begin{align*}
{\mathbf E}[\ln {P^W\left(T_0^1>T_2^1\right)}]-{\mathbf E}[\ln{P^W\left(T_0^1<T_2^1\right)}]=2{\mathbf E}[b]\ .
\end{align*}
\end{lemma}

\begin{proof}
First, note that 
\[
-\int_0^1 |b(s)|{\rm d}s \leq \int_0^y b(s){\rm d} s\leq \int_0^1 |b(s)|{\rm d} s,\quad 0\leq y\leq 1\ ,
\]
and 
 \[
\int_{1}^{2} e^{-2\int_0^y b(s){\rm d}s}{\rm d}y=e^{-2\int_0^1 b(s) {\rm d}s}\int_0^1 e^{-2\int_0^y b(s+1){\rm d}s}{\rm d} y\ .
 \]
We also notice that
\[
 \int_0^{2} e^{-2\int_0^y b(s){\rm d}s}{\rm d}y \leq e^{2\int_0^1 |b(s)| {\rm d}s}  e^{2\int_0^1 |b(s+1)|{\rm d}s}+ e^{2\int_0^1 |b(s)|{\rm d}s}\ .
\]
Using elementary inequality $\ln(x+y)\leq \ln (2x)+ \ln (2y)$ for $x\geq 1, y\geq 1$, one has 
 \begin{align}
   0\leq -{\mathbf E}\ln P^W\left(T_2^1<T_0^1\right)
& \leq {\mathbf E}\ln{\int_0^1 e^{-2\int_0^y b(s){\rm d}s}{\rm d}y}+{\mathbf    E}\ln{\int_0^2 e^{-2\int_0^y b(s){\rm d}s}{\rm d}y} 
\cr
& \leq  2\ln 2+ 8{\mathbf E}\int_0^1 b(s){\rm d}s\ .
 \end{align}
Then
    \begin{align}
\label{eqn: disrho0}
\frac{P^W\left(T_0^1<T_2^1\right)}{P^W\left(T_0^1>T_2^1\right)}
 =\frac{\displaystyle{\int_1^2 e^{-2\int_0^y b(s){\rm d}s}{\rm d}y}}{\displaystyle{\int_0^1 e^{-2\int_0^y b(s){\rm d}s}{\rm d}y}}
=\frac{ \displaystyle e^{-2\int_0^1 b(s){\rm d}s}\int_0^1 e^{-2\int_0^y b(s+1){\rm d}s}{\rm d}y }{\displaystyle{\int_0^1 e^{-2\int_0^y b(s){\rm d}s}{\rm d}y}}\ .
\end{align}
Since $b$ is stationary, then
 \begin{align}
\label{eqn: disrho0a}
{\mathbf E}[\ln{P^W\left(T_0^1<T_2^1\right)}]-{\mathbf E}[\ln {P^W\left(T_0^1>T_2^1\right)}]=-2{\mathbf E}[b]\ .
\end{align}
 The desired result follows readily.

\end{proof}

\noindent Define
\begin{align}\label{eqn:defrho}
    \rho_{0,2}(\eta):=\frac{  E^W\left[ e^{\eta T_0^1}\mathbf{1}_{\left\{T_0^1<T^1_2\right\}} \right] }{E^W\left[ e^{\eta T_2^1}\mathbf{1}_{\left\{T_2^1<T^1_0\right\}} \right] }\ ,
\end{align}
and \[\Lambda^{\bwd}=\{\eta\in {\mathbb R}: \mu^{\bwd}(\eta)<\infty\},\quad \Lambda^{\fwd}=\{\eta\in {\mathbb R}: \mu^{\fwd}(\eta)<\infty\}\ .\]
Recall from (\ref{Eq:CriticalEta:Backward}) and (\ref{Eq:CriticalEta:Forward}) that
\[\eta_c^{\bwd}=\sup\Lambda^{\bwd},\quad \eta_c^{\fwd}=\sup\Lambda^{\fwd}\ .\]
\begin{lemma}\label{Lm:BwdAndFwdCriticalEtaForMuAreTheSame}
    Assume ${\mathbf E}[b]>0$ and ${\mathbf E}\left[\displaystyle{\int_0^1 |b(s)|{\rm d}s}\right]<\infty.$ Then 
    \begin{align}\label{eqnLamfwdbwd}
    \Lambda:= \Lambda^{\fwd}=\Lambda^{\bwd} \ , \end{align}
    and the interior of $\Lambda$ is $(-\infty, \eta_c)$, with $\eta_c:=\eta_c^\bwd=\eta_c^\fwd.$ Furthermore,  for every $\eta\in \Lambda$, $\ln \rho_{0,2}(\eta)$ is integrable  and 
    \begin{align}\label{eqnmufwdbwd} 
   \mu^{\bwd}(\eta)\leq  -\mu^{\fwd}(\eta) \ .
    \end{align}
    
    \noindent In particular,
\begin{align}\label{eqnmuetac} 
\mu^{\bwd}(\eta_c-)\leq -\mu^{\fwd}(\eta_c-)\leq 0 \ .
\end{align}
    If $\eta_c>0$, then  $\mu^{\bwd}(\eta_c-)\leq -\mu^{\fwd}(\eta_c-)< 0 $.
\end{lemma}

    \begin{proof}  Fix $M>2$. Define $T^1_M:=\inf\{t\geq0: X^1(t)\geq M\}.$ Since ${\mathbf E}[b]>0$,  then ${\mathbf P}$-almost surely
$P^W(T_M^1<\infty)=1$ and $P^W(T^0_1<\infty)=1$; see \eqref{eqn:probilityT_01} and Section 5.5 in \cite{KSbook}. In particular, $P^W$-almost surely, $T_M^1\rightarrow\infty$ as $M\rightarrow\infty$.
Recall that by (\ref{Eq:LyapunovFunctionHittingTime:Backward})
\[
\mu^{\fwd}(\eta)={\mathbf E}\left[\ln E^W\left[e^{\eta T_1^0}\mathbf{1}_{\left\{T_1^0<\infty\right\}}\right]\right]={\mathbf E}\left[\ln E^W\left[e^{\eta T_1^0}\right]\right]\ .
\]
  The remaining proof will be divided into three steps.

{\bf Step 1.} Assume $\eta\in \Lambda^{\fwd}$. We shall show that for all $\eta\in \Lambda^{\fwd}$, $\mu^{\fwd}(\eta)>-\infty$ and ${\mathbf P}$-almost surely,    $E^W[e^{\eta T_M^1}]$ is finite.  Since $\eta\in \Lambda^{\fwd}$, then ${\mathbf P}$-almost surely, $E^W[e^{\eta T^0_{1}}]$ is finite. According to the strong Markov property, one has
\[E^W[e^{\eta T_M^1}]=E^W[e^{\eta T^1_2}]E^W[e^{\eta T^2_3}] \cdots E^W[e^{\eta T_M^{M-1}}]\ .
\]
Since $b$ is stationary, then $E^W[e^{\eta T^n_{n+1}}]$ has the same distribution for every $n$. Thus ${\mathbf P}$-almost surely,  $E^W[e^{\eta T^n_{n+1}}]$ is finite  and hence $E^W[e^{\eta T_M^1}]$ is finite. On the other hand, since 
$
{\mathbf E}\left[\ln P^W\left[T_1^0\leq n\right]\right]$ converges to $0$ increasingly, as $n\rightarrow\infty$, then for any $\eta\in \Lambda^{\fwd}$,
\[
\mu^{\fwd}(\eta)\geq (\eta\wedge0) n+ {\mathbf E}\left[\ln P^W\left[T_1^0\leq n\right]\right]>-\infty\ . 
\]
Hence for all $\eta\in {\mathbb R}$, $\mu^{\fwd}(\eta)>-\infty\ .$

{\bf Step 2. } We shall prove $ \Lambda^{\fwd}\subseteq \Lambda^{\bwd}$. 
Assume $\eta\in \Lambda^{\fwd}$. On
$\{T_0^1< T^1_M\}$, we can decompose $T^1_M$ as
\[
T^1_M=T_0^1+\hat{T}_1^0+\hat{T}^1_M\ ,
\]
where $T_0^1+\hat{T}_1^0:=\inf\{t\geq T_0^1: X^1(t) \geq1\}$ is the first hitting time of $1$ after $T_0^1$, and $\hat{T}^1_M$ is independent of $T_0^1$ and $\hat{T}_1^0$, with the same distribution of $T^1_M$. By strong Markov property, conditioning on $T_0^1$, $\hat{T}_1^0$ has the same distribution  as $T_1^0$.  Denote by ${\mathcal F}^W_{t}$ the $\sigma$-field generated by the events before time $t$. Notice that $\{T_0^1< T^1_M\}\in {\mathcal F}^W_{T_0^1}$. Therefore, on $\{T_0^1< T^1_M\} $,
\begin{align}
  E^W\left[e^{\eta T_M^1}\mathbf{1}_{\left\{T_0^1<T^1_M\right\}}\big{|} {\mathcal F}^W_{T_0^1}\right ]
&=\mathbf{1}_{\left\{T_0^1<T^1_M\right\}}E^W\left[e^{\eta  T_0^1+\hat{T}_1^0+\hat{T}^1_M}\big{|} {\mathcal F}^W_{T_0^1}\right ]
\cr &= \mathbf{1}_{\left\{T_0^1<T^1_M\right\}} e^{\eta T_0^1}E^W\left[e^{\eta T^0_1}\right]E^W\left[e^{\eta T_M^1}\right]\ .  
\end{align}
Thus
\begin{align}\label{eqn:decomposition}
  E^W\left[e^{\eta T_M^1}\mathbf{1}_{\left\{T_0^1<T^1_M\right\}}  \right]
=E^W\left[ e^{\eta T_0^1}\mathbf{1}_{\left\{T_0^1<T^1_M\right\}} \right]E^W\left[e^{\eta T^0_1}\right]E^W\left[e^{\eta T_M^1}\right]\ .  
\end{align}

\noindent  Since ${\mathbf P}$-almost surely,  $E^W[e^{\eta T_M^1}]$ is finite, then \eqref{eqn:decomposition} implies that for every $\eta\in \Lambda^{\fwd}$, ${\mathbf P}$-almost surely,
\begin{align}
 1\geq \frac{E^W\left[e^{\eta T_M^1}\mathbf{1}_{\left\{T_0^1<T^1_M\right\}}  \right]}{E^W\left[e^{\eta T_M^1}\right]} 
=E^W\left[ e^{\eta T_0^1}\mathbf{1}_{\left\{T_0^1<T^1_M\right\}} \right]E^W\left[e^{\eta T^0_1} \right]\ .
\end{align}
Hence 
\[
-\ln E^W\left[e^{\eta T^0_1}\right]\geq \ln E^W\left[ e^{\eta T_0^1}\mathbf{1}_{\left\{T_0^1<T^1_M\right\}} \right]\ .
\]
Letting ${M\rightarrow\infty}$, one has 
\begin{align}
\label{leq:Tfb}
-\ln E^W\left[e^{\eta T^0_1}\right]\geq  \ln E^W\left[ e^{\eta T_0^1}\mathbf{1}_{\left\{T_0^1<\infty\right\}}\right],\quad \eta \in \Lambda^{\fwd}\ .
\end{align}

\noindent Moreover, according to Lemma \ref{lem: basiccomp}, one has
\[
 P^W\left({T^1_0<T_M^1}\right)
 \geq P^W\left({T^1_0<T_2^1}\right)>-\infty\ .
\]
By monotone convergence theorem,  there exists $n_0$ such that 
\[
{\mathbf E}[\ln  P^W\left({T^1_0<T_2^1}\right)]\geq {\mathbf E}[\ln  P^W\left({T^1_0<T_2^1\wedge n_0}\right)]>-\infty\ .
\]
Then for any $\eta\in \Lambda^{\fwd}$ (in fact for all $\eta\in {\mathbb R}$),
\[{\mathbf E}\left[\ln E^W\left[ e^{\eta T_0^1}\mathbf{1}_{\left\{T_0^1<T^1_M\right\}} \right]\right]
\geq (\eta\wedge 0) n_0+ {\mathbf E}\left[\ln P^W\left(T_0^1<T^1_2\wedge n_0\right)\right]>-\infty\ .
\]
Therefore, \eqref{leq:Tfb} yields that for  $\eta\in \Lambda^{\fwd}$, 
\begin{align}\label{eqn:mucomparsion}
-\mu^{\fwd}(\eta)\geq \mu^{\bwd}(\eta) \geq {\mathbf E}\left[\ln E^W\left[ e^{\eta T_0^1}\mathbf{1}_{\left\{T_0^1<T^1_M\right\}} \right]\right]>-\infty\ . \end{align}

 \noindent One can further conclude that ${\mathbf E}[\ln \rho_{0,2}(\eta)]$ is also finite. Therefore
\begin{align}
\label{fwdsubbwd}
\Lambda^{\fwd}\subseteq \Lambda^{\bwd}\ .
\end{align}

{\bf Step 3.} We shall show $\Lambda^{\bwd}\subseteq \Lambda^{\fwd}$. Assume $\eta\in \Lambda^{\bwd}$. The proof is similar to {\bf Step 2.} We only give an outline. Using sample path decomposition and strong Markov property, one has $E^W\left[e^{\eta T_{-M}^0}\mathbf{1}_{\left\{T^0_{-M}<\infty\right\}}\right]$ is finite and
\begin{align}\label{eqn:decompositionb}
&  E^W\left[e^{\eta T_{-M}^0}\mathbf{1}_{\left\{T^0_1<T^0_{-M}<\infty\right\}}  \right]\cr
&=E^W\left[ e^{\eta T^0_1}\mathbf{1}_{\left\{T^0_1<T^0_{-M}\right\}} \right]E^W\left[e^{\eta T^1_0} \mathbf{1}_{\left\{T^1_0<\infty\right\}}\right]E^W\left[e^{\eta T_{-M}^0}\mathbf{1}_{\left\{T^0_{-M}<\infty\right\}}\right]\ .  
\end{align}
Thus ${\mathbf P}$-a.s,
\[
1\geq \frac{ E^W\left[e^{\eta T_{-M}^0}\mathbf{1}_{\left\{T^0_1<T^0_{-M}<\infty\right\}}  \right]}{E^W\left[e^{\eta T_{-M}^0}\mathbf{1}_{\left\{T^0_{-M}<\infty\right\}}\right]}= E^W\left[ e^{\eta T^0_1}\mathbf{1}_{\left\{T^0_1<T^0_{-M}\right\}} \right]E^W\left[e^{\eta T^1_0} \mathbf{1}_{\left\{T^1_0<\infty\right\}}\right]\ .
\]
 and hence
\begin{align}\label{eqn:decompositionc}
-\ln E^W\left[e^{\eta T^1_0} \mathbf{1}_{\left\{T^1_0<\infty\right\}}\right]\geq \ln E^W\left[ e^{\eta T^0_1}\mathbf{1}_{\left\{T^0_1<T^0_{-M}\right\}} \right]\ . 
\end{align}

\noindent Letting $M\rightarrow\infty$, we further have
\begin{align}\label{eqn:decompositiond}
-\ln E^W\left[e^{\eta T^1_0} \mathbf{1}_{\left\{T^1_0<\infty\right\}}\right]\geq \ln E^W\left[ e^{\eta T^0_1}\right]\ ,
\end{align}
because of $P^W(T^0_1<\infty)=1.$ Therefore,
$\mu^{\fwd}(\eta)<+\infty. $ Moreover, if $\eta\geq 0$, then $ \ln E^W\left[ e^{\eta T^0_1}\right]\geq 1$. If $\eta<0$, then
\begin{align*}
 \mu^{\fwd}(\eta)={\mathbf E}\left[\ln E^W\left[ e^{\eta T^0_1}\right]\right]
 & \geq 
   {\mathbf E}\left[\ln E^W\left[ e^{\eta T^0_1}\mathbf{1}_{\{T^0_1\leq n\}}\right]\right]\cr
  &\geq \eta n+ {\mathbf E}\left[\ln P^W\left(T^0_1\leq n\}\right)\right] >-\infty\ ,
\end{align*}
for $n$ large enough by noting that
\[
\lim_{n\rightarrow\infty}{\mathbf E}\left[\ln P^W\left( T^0_1\leq n\right)\right]= {\mathbf E}\left[\ln P^W\left(T^0_1<\infty\right)\right]=0\ .
\]

\noindent Therefore, we can conclude that for any $\eta\in\Lambda^{\bwd}$, $\mu^{\fwd}(\eta)$ is finite. Thus
\[ \Lambda^{\bwd}\subseteq \Lambda^{\fwd}\ ,\]
which, together with \eqref{fwdsubbwd}, gives \eqref{eqnLamfwdbwd}. Then \eqref{eqnmufwdbwd} also holds for all $\eta\in \Lambda$ because of \eqref{eqn:mucomparsion}.  Furthermore, noting that for $\eta<0$, both $\mu^{\fwd}(\eta)$ and $\mu^{\bwd}(\eta)$ are negative; and for $\eta\geq 0$, $\mu^{\fwd}\geq 0$. Then from \eqref{eqn:mucomparsion}, we see that $\mu^{\bwd}(\eta)\leq 0$ for all $\eta\in \Lambda.$
We are done.
\end{proof}

\begin{proposition}\label{Prop:MuBwdMinusMuFwd} Assume ${\mathbf E}[b]>0$ and ${\mathbf E}\left[\displaystyle{\int_0^1|b(s)|{\rm d}s}\right]<\infty.$ Then for every $\eta\in \Lambda$,
\begin{align}\label{eqn:relationmu}
\mu^{\bwd}(\eta)=\mu^{\fwd}(\eta)+ {\mathbf E}\left[\ln\rho_{0,2}(\eta)\right]\ ,
\end{align}
\end{proposition}

\begin{remark}  We shall show in Lemma \ref{Lm:ExpectationLnRhoIsConstant} below that $
{\mathbf E}[\ln \rho_{0,2}(\eta)]=-2{\mathbf E}[b]$. 
\end{remark}

\begin{proof}

\noindent Using strong Markov property again, one has for $M\geq 3$,
\begin{align}
&  E^W\left[ e^{\eta T_0^1}\mathbf{1}_{\left\{T_0^1<T^1_M\right\}} \right]\cr
& = E^W\left[ e^{\eta T_0^1}\mathbf{1}_{\left\{T_0^1<T^1_2\right\}} \right]+E^W\left[ e^{\eta T_0^1}\mathbf{1}_{\left\{T^1_2<T_0^1<T^1_M\right\}} \right]\cr
& =   E^W\left[ e^{\eta T_0^1}\mathbf{1}_{\left\{T_0^1<T^1_2\right\}} \right] \cr
&\qquad +E^W\left[ e^{\eta T_2^1}\mathbf{1}_{\left\{T_2^1<T^1_0\right\}} \right]E^W\left[ e^{\eta T_1^2}\mathbf{1}_{\left\{T_1^2<T^2_M\right\}} \right]E^W\left[ e^{\eta T_0^1}\mathbf{1}_{\left\{T_0^1<T^1_M\right\}} \right]\ .
\end{align}

\noindent Then ${\mathbf P}$-a.s,
\begin{align}\label{rho01}
E^W\left[ e^{\eta T_0^1}\mathbf{1}_{\left\{T_0^1<T^1_M\right\}} \right]E^W\left[ e^{\eta T_1^2}\mathbf{1}_{\left\{T_1^2<T^2_M\right\}} \right]
 = \frac{ E^W\left[ e^{\eta T_0^1}\mathbf{1}_{\left\{T_0^1<T^1_M\right\}} \right] }{E^W\left[ e^{\eta T_2^1}\mathbf{1}_{\left\{T_2^1<T^1_0\right\}} \right] }-\rho_{0,2}(\eta)\ .
\end{align}

\noindent By similar reasoning,
\begin{align}
&  E^W\left[ e^{\eta T^1_2} \right]\cr
& = E^W\left[ e^{\eta T^1_2}\mathbf{1}_{\left\{T_1^2<T_0^1\right\}} \right]+ E^W\left[ e^{\eta T^1_2}\mathbf{1}_{\left\{T_0^1<T^1_2\right\}} \right]\cr
& =   E^W\left[ e^{\eta T^1_2}\mathbf{1}_{\left\{T_1^2<T_0^1\right\}} \right]  +E^W\left[ e^{\eta T^1_0}\mathbf{1}_{\left\{T^1_0<T^2_1\right\}} \right]E^W\left[ e^{\eta T^0_1} \right]E^W\left[ e^{\eta T_2^1}\right]\ ,
\end{align}
and 
\begin{align}\label{rho02}
\rho_{0,2}(\eta) E^W\left[ e^{\eta T^0_1} \right]E^W\left[ e^{\eta T_2^1}\right]=
\frac{E^W\left[ e^{\eta T^1_2} \right]}{ E^W\left[ e^{\eta T^1_2}\mathbf{1}_{\left\{T^1_2<T_0^1\right\}} \right]  }-1\ .
\end{align}

\noindent Using \eqref{rho02} and then
\eqref{rho01}, one can get, ${\mathbf P}$-a.s,
\begin{align}
&\rho_{0,2}(\eta)
\left(1-E^W\left[ e^{\eta T_1^0}\right] E^W\left[ e^{\eta T^1_0}\mathbf{1}_{\left\{T_0^1<T^1_M\right\}} \right] \right) E^W\left[ e^{\eta T^1_2}\right]
\cr
&=\rho_{0,2}(\eta) E^W\left[ e^{\eta T^1_2}\right]-\left(\frac{E^W\left[ e^{\eta T^1_2} \right]}{ E^W\left[ e^{\eta T^1_2}\mathbf{1}_{\left\{T^1_2<T_0^1\right\}} \right]  }-1\right) E^W\left[ e^{\eta T_0^1}\mathbf{1}_{\left\{T_0^1<T^1_M\right\}} \right] 
\cr
&= \left(1-  E^W\left[ e^{\eta T_2^1} \right]E^W\left[ e^{\eta T_1^2}\mathbf{1}_{\left\{T_1^2<T^2_M\right\}} \right]\right)E^W\left[ e^{\eta T^1_0}\mathbf{1}_{\left\{T_0^1<T^1_M\right\}} \right]\ .
\end{align}

\noindent Therefore, ${\mathbf P}$-a.s,
\begin{align}
&\ln \rho_{0,2}(\eta)+
\ln  E^W\left[ e^{\eta T^1_2}\right]-\ln E^W\left[ e^{\eta T_0^1}\mathbf{1}_{\left\{T_0^1<T^1_M\right\}} \right]
\cr
&\quad= \ln\left(1-  E^W\left[ e^{\eta T^0_1} \right]E^W\left[ e^{\eta T_0^1}\mathbf{1}_{\left\{T_0^1<T^1_M\right\}} \right] \right)
\cr &\qquad\qquad-\ln \left(1-E^W\left[ e^{\eta T_2^1}\right]E^W\left[ e^{\eta T_1^2}\mathbf{1}_{\left\{T_1^2<T^2_M\right\}} \right] \right)\ .
\end{align}
Since $x\mapsto b(x)$ is stationary, then 
\[
E^W\left[ e^{\eta T_2^1}\right]E^W\left[ e^{\eta T_1^2}\mathbf{1}_{\left\{T_1^2<T^2_{M+1}\right\}} \right]\overset{d}{=} E^W\left[ e^{\eta T^0_1} \right]E^W\left[ e^{\eta T_0^1}\mathbf{1}_{\left\{T_0^1<T^1_M\right\}}\right]\ .
\]
Then
\begin{align}\label{rho03}
&{\mathbf E}[\ln \rho_{0,2}(\eta)]+
{\mathbf E}\left[\ln  E^W\left[ e^{\eta T^1_2}\right]\right]-\frac{1}{K-2}\sum_{M=3}^{K}{\mathbf E}\left[ \ln E^W\left[ e^{\eta T_0^1}\mathbf{1}_{\left\{T_0^1<T^1_M\right\}} \right]\right]
\cr
&= \frac{1}{K-2}{\mathbf E}\bigg{[} \ln\left(1-  E^W\left[ e^{\eta T^0_1} \right]E^W\left[ e^{\eta T_0^1}\mathbf{1}_{\left\{T_0^1<T^1_2\right\}} \right] \right)
\cr &\qquad\qquad\qquad-\ln \left(1-E^W\left[ e^{\eta T_2^1}\right]E^W\left[ e^{\eta T_1^2}\mathbf{1}_{\left\{T_1^2<T^2_K\right\}} \right] \right)\bigg{]}\ .
\end{align}
Since, by monotone convergence theorem, ${\mathbf P}$-a.s,
\[
 \lim_{M\rightarrow\infty}E^W\left[ e^{\eta T_0^1}\mathbf{1}_{\left\{T_0^1<T^1_M\right\}}\right]= E^W\left[ e^{\eta T_0^1}\mathbf{1}_{\left\{T_0^1<\infty\right\}}\right]\ ,
\]
and
\[\lim_{K\rightarrow\infty}E^W\left[ e^{\eta T_1^2}\mathbf{1}_{\left\{T_1^2<T^2_K\right\}}\right] =E^W\left[ e^{\eta T_1^2}\mathbf{1}_{\left\{T_1^2<\infty\right\}} \right]\ . 
\]
then the r.h.s. of \eqref{rho03} will converge to $0$ as $K\rightarrow\infty $ and 
\[
\frac{1}{K-2}\sum_{M=3}^{K}{\mathbf E}\left[ \ln E^W\left[ e^{\eta T_0^1}\mathbf{1}_{\left\{T_0^1<T^1_M\right\}} \right]\right]\rightarrow {\mathbf E}\left[ \ln E^W\left[ e^{\eta T_0^1}\mathbf{1}_{\left\{T_0^1<\infty\right\}} \right]\right]=\mu^{\bwd}(\eta)\ .
\]
And hence
\begin{align*}
&{\mathbf E}[\ln \rho_{0,2}(\eta)]+
\mu^{\fwd}(\eta)=\mu^{\bwd}(\eta),\quad \eta\in \Lambda^{\fwd}\ .
\end{align*}
We have completed the proof.
\end{proof}

Next, we establish a lemma showing that $\mathbf{E}[\ln\rho_{0,2}(\eta)]$ is actually a constant. Our approach is analytic.

\begin{lemma}\label{Lm:ExpectationLnRhoIsConstant}
 Assume ${\mathbf E}[b]>0$ and ${\mathbf E}\left[\displaystyle{\int_0^1|b(s)|{\rm d}s}\right]<\infty.$ Then we have
\begin{equation}\label{Lm:ExpectationLnRhoIsConstant:Eq:FinalRelation}
\mathbf{E}[\ln \rho_{0, 2}(\eta)] = -2\mathbf{E} [b] \ ,
\end{equation}
for all $\eta\in \Lambda$.
\end{lemma}

\begin{proof}
Set $$u_L(x)=u_L(x; \eta)=E^W\left[e^{\eta T^1_0}\mathbf{1}_{\left\{T^1_0<T^1_2\right\}}\right]$$ and $$u_R(x)=u_R(x; \eta)=E^W\left[e^{\eta T^1_2}\mathbf{1}_{\left\{T^1_2<T^1_0\right\}}\right] \ ,$$ where the hitting times $T^1_0, T^1_2$ are for the process $X^x(t)$ defined in (\ref{Eq:SDE-RandomDrift}). Then both $u_L$ and $u_R$ satisfy the ODE
\begin{equation}\label{Lm:ExpectationLnRhoIsConstant:Eq:ODEForRhoNumeratorDenominator}
\dfrac{1}{2}\dfrac{{\rm d}^2 u}{{\rm d}x^2}+b(x)\dfrac{{\rm d}u}{{\rm d}x}+\eta u = 0 \ , \ x\in [0,2]\ ,
\end{equation}
with Sturm-Liouville type boundary conditions $u_L(0)=1, u_L(2)=0$ and $u_R(0)=0, u_R(2)=1$. We have $\rho_{0,2}(\eta)=\dfrac{u_L(1)}{u_R(1)}$.

The general solution to (\ref{Lm:ExpectationLnRhoIsConstant:Eq:ODEForRhoNumeratorDenominator}) has the form
$$u(x)=Au_0(x)+Bu_1(x)$$
with the two fundamental solutions $u_0(x)$ such that $u_0(0)=0, u_0'(0)=1$ and $u_1(x)$ such that $u_1(0)=1, u_1'(0)=0$ (see \cite{ODEHartman}). Matching the boundary conditions for $u_L(x)$, $u_R(x)$ we get
$$u_L(x)=A_L u_0(x)+B_L u_1(x) \ , \ u_R(x)=A_R u_0(x)+B_R u_1(x) \ ,$$
where
$$A_L=-\dfrac{u_1(2)}{u_0(2)} \ , \ B_L=1 \ , \ A_R=-\dfrac{1}{u_0(2)} \ , \ B_R=0 \ .$$

\noindent So we get $$u_L(x)=-\dfrac{1}{u_0(2)}(u_1(2)u_0(x)-u_1(x)u_0(2)) \ , \ u_R(x)=\dfrac{u_0(x)}{u_0(2)} \ .$$
Thus
\begin{equation}\label{Lm:ExpectationLnRhoIsConstant:Eq:RepresentationOfRhoUsingFoundamentalSolutions}
\rho_{0,2}(\eta)=\dfrac{u_1(1)u_0(2)-u_1(2)u_0(1)}{u_0(1)} \ .
\end{equation}

\noindent Let the Wronskian matrix be $W(x)=\begin{pmatrix}u_0(x) & u_1(x) \\ u_0'(x) & u_1'(x)\end{pmatrix}$. Note that $W(0)=\begin{pmatrix}0 & 1 \\ 1 & 0\end{pmatrix}$ so $W(0)=W(0)^{-1}$, and $W(1)^{-1}=\dfrac{1}{\det W(1)}\begin{pmatrix}u_1'(1) & -u_1(1) \\ -u_0'(1) & u_0(1) \end{pmatrix}$. Define \begin{equation}\label{Lm:ExpectationLnRhoIsConstant:Eq:M0M1}
M_0=W(1)W(0)^{-1} \ , \ M_1=W(2)W(1)^{-1} \ .
\end{equation} Then we see that
$$M_0=\begin{pmatrix}u_1(1) & u_0(1) \\ u_1'(1) & u_0'(1)\end{pmatrix} \ ,$$
and
$$\begin{array}{ll}
M_1 & = \begin{pmatrix} u_0(2) & u_1(2) \\ u_0'(2) & u_1'(2)\end{pmatrix}
\cdot\dfrac{1}{\det W(1)}\begin{pmatrix}u_1'(1) & -u_1(1) \\ -u_0'(1) & u_0(1) \end{pmatrix}
\\
& = \dfrac{1}{\det W(1)}\begin{pmatrix}
u_0(2)u_1'(1)-u_1(2)u_0'(1) & -u_0(2)u_1(1)+u_1(2)u_0(1) \\ u_0'(2)u_1'(1)-u_1'(2)u_0'(1) & -u_0'(2)u_1(1)+u_1'(2)u_0(1)\end{pmatrix} \ .
\end{array} 
$$

\noindent The above two calculations combined with (\ref{Lm:ExpectationLnRhoIsConstant:Eq:RepresentationOfRhoUsingFoundamentalSolutions}) give
$$\rho_{0,2}(\eta)=-\det W(1)\dfrac{(M_1)_{12}}{(M_0)_{12}} \ .$$

\noindent The well-celebrated Abel's identity for Wronskian (see \cite{Abel1829}, which can be verified directly from the ODE (\ref{Lm:ExpectationLnRhoIsConstant:Eq:ODEForRhoNumeratorDenominator})) gives us $$\det W(1)=\det W(0)\cdot \exp\left(-2\displaystyle{\int_0^1 b(x){\rm d}x} \right)=- \exp\left(-2\displaystyle{\int_0^1 b(x){\rm d}x} \right) \ .$$ So we have
$$\rho_{0,2}(\eta)=\exp\left(-2\displaystyle{\int_0^1 b(x){\rm d}x} \right)\cdot \dfrac{(M_1)_{12}}{(M_0)_{12}} \ .
$$
In other words
\begin{equation}\label{Lm:ExpectationLnRhoIsConstant:Eq:Cocycle}
\ln \rho_{0,2}(\eta)=-2 \int_0^1 b(x){\rm d}x + \ln (M_1)_{12} - \ln (M_0)_{12} \ ,
\end{equation}
and thus
$$\mathbf{E}[\ln \rho_{0,2}(\eta)]=-2\mathbf{E} [b] + \mathbf{E} \ln (M_1)_{12} - \mathbf{E} \ln (M_0)_{12} \ .$$

We now aim to show that the stationarity of $b(x)$ and the way we define $M_0, M_1$ in (\ref{Lm:ExpectationLnRhoIsConstant:Eq:M0M1}) imply that under $\mathbf{E}$, we actually have $\mathbf{E} \ln (M_1)_{12} - \mathbf{E} \ln (M_0)_{12}=0$, which leads to  (\ref{Lm:ExpectationLnRhoIsConstant}). In fact, from the construction of $W(x)$ using fundamental solutions to (\ref{Lm:ExpectationLnRhoIsConstant:Eq:ODEForRhoNumeratorDenominator}) we see that
$$W'(x)=\begin{pmatrix}0 & 1 \\ -2\eta & -2b(x)\end{pmatrix}W(x) \ .$$
From here, by introducing the general fundamental matrix 
$$M(x,y)=W(x)W(y)^{-1} \ , \ x\geq y $$
we see that
$$
\dfrac{\partial M(x,y)}{\partial x} = \begin{pmatrix}0 & 1 \\ -2\eta & -2b(x)\end{pmatrix}M(x, y) \ , \ M(x,x)=I \ .$$

\noindent Since $b(x)$ is stationary, the law of the coefficient matrix $\begin{pmatrix}0 & 1 \\ -2\eta & -2b(x)\end{pmatrix}$ on $[1,2]$ 
coincides with that on $[0,1]$. This means, $M(x,y)=M(x,y;\omega)$ satisfies
\[
M(2,1;\omega)=M(1,0;\theta_1\omega),
\]
where $\theta_1$ denotes the spatial shift. Since $\theta_1$ is 
measure-preserving under $\mathbf{P}$, it follows that $M_1=M(2,1)$ and $M_0=M(1,0)$ 
have the same distribution under $\mathbf{P}$, which gives the desired result.
\end{proof}

Combining Proposition \ref{Prop:MuBwdMinusMuFwd} and Lemma \ref{Lm:ExpectationLnRhoIsConstant}, we obtain

\begin{corollary}\label{Corrollary:MuBwdMFwdDifferByConstant}
 Assume that  ${\mathbf E}[b]>0$ and ${\mathbf E}\left[\displaystyle{\int_0^1 |b(s)|{\rm d}s}\right]<\infty.$  Then for every $\eta\in \Lambda$, 
\begin{equation}\label{Corrollary:MuBwdMFwdDifferByConstant:Eq:Diff}
\mu^\bwd(\eta) - \mu^\fwd(\eta)= -2\mathbf{E} [b] \ .
\end{equation}
\end{corollary}

\begin{remark}\label{Rem:scalefunction}
  We  remark here that in the case of ${\mathbf E}[b]=0$, all the results in Section \ref{Sec:LDP:RelationMuAndMufwd} may hold with additional assumption that \eqref{eqn: generalintecond} holds (and hence \eqref{eqn:probilityT_01} also holds).  In particular, $\mu^{\fwd}(\eta)=\mu^{\bwd}(\eta)$ for $\eta\in \Lambda$ .
\end{remark}

\subsection{Large Deviations for the process $X^x(t)$}\label{Sec:LDP:ProcessX}
First, we state some general properties of the function $I^\bwd(a)\equiv I(a)$ defined in (\ref{Eq:EntropyHittingTime:Backward}):

\begin{lemma}\label{Lm:PropertiesOfIInOurCase} The function $I(a)$ has the following properties:

\begin{itemize}
\item[\emph{(1)}] $I(a)$ is convex in $a$ and $I(a)\geq 0$ for $a\in (0,\infty)$;

\item[\emph{(2)}] $I(a)$ is decreasing in $a$ for $a\in (0,\mu'(0)]$ and is increasing in $a$ for $a\in (\mu'(0), \infty)$, with $I(\mu'(0))=-\mu(0)$
to be the minimum point of $I(a)$ as $a\in (0,\infty)$;

\item[\emph{(3)}] $\lim\limits_{a\rightarrow 0+}I(a)=+\infty$;

\item[\emph{(4)}] $I(a)$ is piecewisely differentiable on both intervals $a\in (0, \mu'(\eta_c-))$ and $a\in [\mu'(\eta_c-), \infty)$ and
$I'(a)\leq \eta_c$ for all $a\in (0, \infty)$;

\item[\emph{(5)}] $I(a)\geq a\eta_c-\mu(\eta_c-)$ and $I(a)> a\eta_c-\mu(\eta_c-)$ when $a\in (0, \mu'(\eta_c-))$;

\item[\emph{(6)}] If $\mu'(\eta_c-)<\infty$, then $I(a)=a\eta_c-\mu(\eta_c-)$ for $a\in [\mu'(\eta_c-), \infty)$;

\item[\emph{(7)}] If $\mu'(\eta_c-)=\infty$, then $I(a)>a\eta_c-\mu(\eta_c-)$ for $a\in (0, \infty)$ and
$I(a)-[a\eta_c-\mu(\eta_c-)]$ decreases to $0$ as $a\rightarrow \infty$;

\item[\emph{(8)}] The function $a \mapsto aI\left(\dfrac{1}{a}\right)$ is convex on $a\in (0,\infty)$.
\end{itemize}
\end{lemma}

\begin{proof}
Properties (1)-(7) can be proved in the same way as we have already proved in \cite[Lemma 5.2]{FanHuTerlovCMP}. So we omit duplicating the proof. For (8), we define $g(x,a)=aI\left(\dfrac{x}{a}\right)$. Then $g(x,a)$ is the perspective transform (see \cite[Section 3.2.6]{BoydEtAlConvexOptimizationBook}) of $I(a)$, and so $aI\left(\dfrac{1}{a}\right)=g(1,a), a>0$ is convex.
\end{proof}

For the function $I^\fwd(a)$ defined in (\ref{Eq:EntropyHittingTime:Forward}), we have a similar result:

\begin{lemma}\label{Lm:PropertiesOfIInOurCase:Forward} The function $I^\fwd(a)$ has similar properties \emph{(1)}-\emph{(8)} as in \emph{Lemma \ref{Lm:PropertiesOfIInOurCase}}, by replacing $I, \mu$ with $I^\fwd, \mu^\fwd$.
\end{lemma}

\begin{proof}
By Corollary \ref{Corrollary:MuBwdMFwdDifferByConstant}, we have $\mu^\bwd(\eta)-\mu^\fwd(\eta)=-2\mathbf{E}[b]$ for $\eta\in (-\infty, \eta_c)= \Lambda^\circ$. By (\ref{Eq:EntropyHittingTime:Backward}), (\ref{Eq:EntropyHittingTime:Forward}) we know that 
$$I(a)=\sup\limits_{\eta\leq \eta_c}(a\eta - \mu(\eta))=\sup\limits_{\eta\leq \eta_c}(a\eta -\mu^\fwd(\eta))+2\mathbf{E} [b]=I^\fwd(a)+2\mathbf{E}[b] \ ,$$
where we have used $\eta_c=\eta_c^\fwd$ proved in Lemma \ref{Lm:BwdAndFwdCriticalEtaForMuAreTheSame}. Thus \begin{equation}\label{Lm:PropertiesOfIInOurCase:Forward:Eq:DifferenceIIfwd}
I^\fwd(a)=I(a)-2\mathbf{E}[b] \ , a>0
\end{equation} satisfies the same properties as $I(a)$ when replacing $\mu$ by $\mu^\fwd$, etc.
\end{proof}

The next Lemma is a result of convexity of $I^\bwd$ and $I^\fwd$:
\begin{lemma}\label{Lm:ConvexityEstimateIbwdIfwdSum}
For any $\alpha>0, \beta>0$ we have
\begin{equation}\label{Lm:ConvexityEstimateIbwdIfwdSum:Eq:Estimate}
I^\bwd(\alpha)+I^\fwd(\beta)\geq (\alpha+\beta) \lim\limits_{\delta\rightarrow 0}\delta I\left(\dfrac{1}{\delta}\right) \ .
\end{equation}
\end{lemma}

\begin{proof}
Define
\[
I_g(a)=\begin{cases}
    aI\left(\frac{1}{a}\right),\quad & a\geq 0,\cr
    |a|I^{\fwd}\left(\frac{1}{|a|}\right)\stackrel{(\ref{Lm:PropertiesOfIInOurCase:Forward:Eq:DifferenceIIfwd})}{=}|a|\left(I\left(\frac{1}{|a|}\right)-2{\mathbf E}[b]\right),\quad & a\leq 0,
\end{cases}
\]
with the convention that $I_g(0):=\lim\limits_{\delta\rightarrow 0}aI\left(\dfrac{1}{a}\right).$ 
By Property (8) of Lemma \ref{Lm:PropertiesOfIInOurCase} we see that $I_g(a)$ is a convex function for $a>0$. When $a<0$ we use the fact that the linear subtraction preserves convexity, so that $I_g(a)$ for $a<0$ is also convex. It remains to show that $I_g(a)$ remains convex at $a=0$. To this end we consider its left and right derivatives at $0$:
\[
(I_g)'_-(0) = \lim_{a\rightarrow 0-} \frac{I_g(a)-I_g(0)}{a}\ , 
\quad
(I_g)'_+(0) = \lim_{a\rightarrow 0+} \frac{I_g(a)-I_g(0)}{a}\ .
\]
By the definition of $I_g(a)$ we have
\[
(I_g)'_-(0) = (I_g)'_+(0) - 2\mathbf E[b] \le (I_g)'_+(0)\ , 
\]
which implies that $I_g(a)$ is still convex at $a=0$ by the standard one-dimensional convexity criterion.  Using the above fact that $I_g(a)$ is convex for all $a\in \mathbb{R}$, we get
\[\frac{I^\bwd(\alpha)+I^\fwd(\beta)}{\alpha+\beta}=\frac{\alpha I_g(1/\alpha)+\beta I_g(-1/\beta)}{\alpha+\beta}\geq I_g(0) \ ,\]
which is the desired estimate (\ref{Lm:ConvexityEstimateIbwdIfwdSum:Eq:Estimate}).
\end{proof}

\

Using Theorem \ref{Thm:LDPHittingTime} for the LDP of hitting time, we can further obtain LDP for the process $X^x(t)$ as the following 

\begin{theorem}[LDP for the Process $X^x(t)$] \label{Thm:LDPProcessX} Almost surely with respect to $\mathbf{P}$ the following estimates hold. Let $v>0$ and $\kappa\in (0,1]$. For any closed set $G\subset [0, +\infty)$ we have
\begin{equation}\label{Thm:LDPProcessX:Eq:UpperBoundPositive}
\limsup\limits_{t\rightarrow\infty}\dfrac{1}{\kappa t} \ln P^W\left(\dfrac{vt - X^{vt}(\kappa t)}{\kappa t}\in G\right)\leq -\inf\limits_{c\in G}cI\left(\dfrac{1}{c}\right) \ ;
\end{equation}
and for any open set $F\subset [0, +\infty)$ we have
\begin{equation}\label{Thm:LDPProcessX:Eq:LowerBoundPositive}
\liminf\limits_{t\rightarrow\infty}\dfrac{1}{\kappa t} \ln P^W\left(\dfrac{vt - X^{vt}(\kappa t)}{\kappa t}\in F\right)\geq -\inf\limits_{c\in F}cI\left(\dfrac{1}{c}\right) \ .
\end{equation}
For any closed set $G\subset (-\infty, 0]$ we have 
\begin{equation}\label{Thm:LDPProcessX:Eq:UpperBoundNegative}
\limsup\limits_{t\rightarrow\infty}\dfrac{1}{\kappa t} \ln P^W\left(\dfrac{-vt - X^{-vt}(\kappa t)}{\kappa t}\in G\right)\leq -\inf\limits_{-c\in G}cI^\fwd\left(\dfrac{1}{c}\right) \ ;
\end{equation}
and for any open set $F\subset (-\infty, 0]$ we have
\begin{equation}\label{Thm:LDPProcessX:Eq:LowerBoundNegative}
\liminf\limits_{t\rightarrow\infty}\dfrac{1}{\kappa t} \ln P^W\left(\dfrac{-vt - X^{-vt}(\kappa t)}{\kappa t}\in F\right)\geq -\inf\limits_{-c\in F}cI^\fwd\left(\dfrac{1}{c}\right) \ .
\end{equation}
\end{theorem}

\begin{proof} We first prove separately for (\ref{Thm:LDPProcessX:Eq:UpperBoundPositive}) and (\ref{Thm:LDPProcessX:Eq:LowerBoundPositive}), respectively. 

\noindent{\textit{Proof of \emph{(\ref{Thm:LDPProcessX:Eq:UpperBoundPositive})}}}. 
Consider $c\geq \dfrac{1}{\mu'(0)}$. Note that
\begin{equation}\label{Thm:LDPProcessX:Eq:UpperBoundPositive:InvertedToHittingTime1}
P^W\left(\dfrac{vt-X^{vt}(\kappa t)}{\kappa t}>c\right)
= P^W\left(X^{vt}(\kappa t)<(v-c\kappa)t\right)
\leq P^W\left(\dfrac{T^{vt}_{(v-c\kappa)t}}{t}<\kappa\right) \ .
\end{equation}

Since $\kappa\leq c\kappa \mu'(0)$, using (\ref{Thm:LDPProcessX:Eq:UpperBoundPositive:InvertedToHittingTime1}) and  (\ref{Thm:LDPHittingTime:Eq:UpperBoundReducedToIntervalPart1}) we have

\begin{equation}\label{Thm:LDPProcessX:Eq:UpperBoundReducedToIntervalPart1}
\limsup\limits_{t\rightarrow\infty} \dfrac{1}{\kappa t}\ln P^W\left(\dfrac{vt - X^{vt}(\kappa t)}{\kappa t} > c\right)
\leq \limsup\limits_{t\rightarrow\infty} \dfrac{1}{\kappa t}\ln P^W\left(\dfrac{T^{vt}_{(v-c\kappa)t}}{t}<\kappa\right)
\leq -c I\left(\dfrac{1}{c}\right) \ .
\end{equation}

The case when $0\leq c<\dfrac{1}{\mu'(0)}$, is a bit more involved, and we will use a stopping time segmentation argument (see the proof of \cite[Theorem 1]{CometsGantertZeitouni2000}). We first work out the case when $c=0$. Let $\varepsilon, \nu, \delta>0$ with $\nu<\mu'(0)$ and $\delta<\dfrac{1}{\mu'(0)}$ \footnote{We will later send $t\rightarrow \infty$ first, then $\varepsilon\rightarrow 0$, then $\nu \rightarrow 0$, and in the last $\delta\rightarrow 0$.}. Then we have
\begin{equation}\label{Thm:LDPProcessX:Eq:UpperBoundPositive:SmallcCase:cIs0:Decomposition}
\begin{array}{ll}
& P^W\left(\dfrac{vt-X^{vt}(\kappa t)}{\kappa t}\geq 0\right)
\\
\leq & P^W\left(T^{vt}_{vt-\kappa t\delta}\geq \kappa t\right) + P^W\left(\nu\kappa t \delta \leq T^{vt}_{vt-\kappa t \delta}<\kappa t, \dfrac{vt-X^{vt}(\kappa t)}{\kappa t}\geq 0\right)
\\
& \qquad \qquad \qquad \qquad +P^W(T^{vt}_{vt-\kappa t \delta}< \nu \kappa t \delta)
\\
\leq & P^W\left(T^{vt}_{vt-\kappa t \delta}\geq \kappa t\right) + \sum\limits_{k\geq \frac{\nu}{\varepsilon}; l\geq 0; (k+l+2)\varepsilon\leq \frac{1}{\delta}}P^W\left(\dfrac{T^{vt}_{vt-\kappa t \delta}}{\kappa t \delta}\in [k\varepsilon, (k+1)\varepsilon)\right)
\\
& \qquad \cdot P^W\left(\dfrac{T^{vt-\kappa t \delta}_{vt}}{\kappa t \delta}\in [l\varepsilon, (l+1)\varepsilon)\right)\cdot \sup\limits_{-2\kappa t \delta \varepsilon\leq s-\kappa t(1-(k+l)\delta\varepsilon)\leq 0} P^W(vt-X^{vt}(s)\geq 0)
\\
& \qquad \qquad \qquad \qquad + P^W\left(T^{vt}_{vt-\kappa t \delta}<\nu \kappa t \delta\right) \ .
\end{array}
\end{equation}
Here the event $\left\{\nu \kappa t \delta \leq T^{vt}_{vt-\kappa t \delta}<\kappa t , \dfrac{vt-X^{vt}(\kappa t)}{\kappa t}\geq 0\right\}$ is decomposed into $3$ pieces: the trajectory $X^{vt}(\kappa t)$ running from $vt$ to $vt-\kappa t \delta$ during time $t_1\in [k\varepsilon\kappa t\delta, (k+1)\varepsilon \kappa t \delta)$; running back from $vt-\kappa t\delta$ to $vt$ during time $t_2\in [t_1+l\varepsilon\kappa t \delta, t_1+(l+1)\varepsilon\kappa t \delta)$; running from $vt$ to $X^{vt}(\kappa t)\leq vt$ during time $t_3\in [\kappa t-t_2, \kappa t]$, and we used the strong Markov propety to connect these $3$ pieces of trajectory. We now define the random variable
\begin{equation}\label{Thm:LDPProcessX:Eq:UpperBoundPositive:SmallcCase:cIs0:Decomposition:a}
a=\limsup\limits_{t\rightarrow\infty} \dfrac{1}{\kappa t}\sup\limits_{s: \ -2\kappa t\delta \varepsilon\leq s-\kappa t\leq 0}\ln P^W\left(vt-X^{vt}(s)\geq 0\right) \ .
\end{equation}
Apparently, since $0\leq P^W(vt-X^{vt}(s)\geq 0)\leq 1$, we have
$$\limsup\limits_{t\rightarrow\infty}\dfrac{1}{\kappa t}\ln P^W\left(vt-X^{vt}(\kappa t)\geq 0\right)\leq a \leq 0 \ .$$

Moreover, since the process $X^{vt}(s)$ starts from $vt$ and the environment is stationary and ergodic, and $\kappa t\rightarrow \infty$ as $t\rightarrow\infty$, we observe that $a$ defined in (\ref{Thm:LDPProcessX:Eq:UpperBoundPositive:SmallcCase:cIs0:Decomposition:a}) is actually independent of the choice of $\kappa$.

Using the inequality, for $s\leq \kappa t$, that $$P^W\left(vt-X^{vt}(\kappa t)\geq 0\right)\geq P^W\left(vt-X^{vt}(s)\geq 0\right)\cdot \inf\limits_{w\leq vt}P^W\left(X^{w}(\kappa t-s)\leq w-(\kappa t-s)\right)$$
we further estimate that
$$\begin{array}{ll}
& \limsup\limits_{t\rightarrow\infty} \dfrac{1}{\kappa t}\ln P^W\left(vt-X^{vt}(\kappa t)\geq 0\right)
\\
\geq & a+\limsup\limits_{t\rightarrow\infty}\dfrac{1}{\kappa t} \ln \sup\limits_{s: \  -2t\delta\varepsilon\leq s-\kappa t\leq 0}\inf\limits_{w\leq vt} P^W\left(X^w(\kappa t-s)\leq w-(\kappa t-s)\right)
\\
\geq & a-C\delta \varepsilon \ ,
\end{array}$$
with some $C>0$ using Lemma \ref{Lm:WorstCaseEnvironmentEstimateForaInThmLDPProcessX}, which may depend on $\kappa$. So the above reasoning gives

\begin{equation}\label{Thm:LDPProcessX:Eq:UpperBoundPositive:SmallcCase:cIs0:BoundFora}
a-C\delta\varepsilon\leq \limsup\limits_{t\rightarrow\infty}\dfrac{1}{\kappa t}\ln P^W\left(vt-X^{vt}(\kappa t)\geq 0\right)\leq a \ .
\end{equation}

Using Theorem \ref{Thm:LDPHittingTime}, as well as (\ref{Thm:LDPProcessX:Eq:UpperBoundPositive:SmallcCase:cIs0:BoundFora}), we can estimate term-by-term of the right-hand-side of (\ref{Thm:LDPProcessX:Eq:UpperBoundPositive:SmallcCase:cIs0:Decomposition}) as follows:

\begin{itemize}
\item[(i)] The term $P^W\left(T^{vt}_{vt-\kappa t \delta}\geq \kappa t\right)$. We have, using (\ref{Thm:LDPHittingTime:Eq:UpperBound}), that
$$\begin{array}{ll}
\limsup\limits_{t\rightarrow\infty}\dfrac{1}{\kappa t}\ln P^W\left(T^{vt}_{vt-\kappa t \delta}\geq \kappa t\right)
& = \limsup\limits_{t\rightarrow\infty}\dfrac{1}{\kappa t}\ln P^W\left(\dfrac{T^{vt}_{(v-\kappa \delta)t}}{t}\geq \kappa\right) 
\\
& \leq  -\dfrac{1}{\kappa}(v-(v-\kappa\delta))\inf\limits_{a\in [\kappa, +\infty)} I\left(\dfrac{a}{v-(v-\kappa \delta)}\right)
\\
& = - \delta I\left(\dfrac{1}{\delta}\right) \ .
\end{array}$$
Here in the last inequality we have used part (2) of Lemma \ref{Lm:PropertiesOfIInOurCase}, that $I(a)$ is increasing when $a\geq \mu'(0)$, as well as $\dfrac{a}{\kappa}\geq 1$, $\dfrac{1}{\delta}> \mu'(0)$.

\item[(ii)] The term $P^W\left(\dfrac{T^{vt}_{vt-\kappa t \delta}}{\kappa t \delta}\in [k\varepsilon, (k+1)\varepsilon)\right)$. Similarly as in (i), we use (\ref{Thm:LDPHittingTime:Eq:UpperBound}) to get
$$\begin{array}{ll}
& \limsup\limits_{t\rightarrow\infty}\dfrac{1}{\kappa t}\ln P^W\left(\dfrac{T^{vt}_{vt-\kappa t \delta}}{\kappa t \delta}\in [k\varepsilon, (k+1)\varepsilon)\right)
\\
=& \limsup\limits_{t\rightarrow\infty}\dfrac{1}{\kappa t}\ln P^W\left(\dfrac{T^{vt}_{vt-\kappa t \delta}}{t}\in [k\kappa \delta\varepsilon, (k+1)\kappa \delta\varepsilon)\right)
\\
\leq &  -\dfrac{1}{\kappa}(v-(v-\kappa \delta))\inf\limits_{a\in [k\kappa\delta\varepsilon, (k+1)\kappa\delta\varepsilon)}I\left(\dfrac{a}{(v-(v-\kappa\delta))}\right)
\\
\leq & -\delta I(k\varepsilon)+\delta \omega(\delta;\varepsilon) \ ,
\end{array}$$
where $\omega(\delta;\varepsilon)$ is the oscillation of $I(\bullet)$ defined by 
$$\omega(\delta;\varepsilon)=\max\left\{|I(a)-I(a')|; a, a'\in [0, 1/\delta], |a-a'|\leq \varepsilon\right\} \ ,$$
which tends to $0$ with $\varepsilon$ for any fixed $\delta>0$.

\item[(iii)] The term $P^W\left(\dfrac{T_{vt}^{vt-\kappa t \delta}}{\kappa t \delta}\in [l\varepsilon, (l+1)\varepsilon)\right)$. Similarly as in (i) and (ii), using (\ref{Thm:LDPHittingTime:Forward:Eq:UpperBound}) we get

$$\begin{array}{ll}
& \limsup\limits_{t\rightarrow\infty}\dfrac{1}{\kappa t}\ln P^W\left(\dfrac{T^{vt-\kappa t \delta}_{vt}}{\kappa t \delta}\in [l\varepsilon, (l+1)\varepsilon)\right)
\\
= & \limsup\limits_{t\rightarrow\infty}\dfrac{1}{\kappa t}\ln P^W\left(\dfrac{T^{vt-\kappa t \delta}_{vt}}{t}\in [l\kappa\delta\varepsilon, (l+1)\kappa\delta\varepsilon)\right)
\\
\leq & -\dfrac{1}{\kappa}(v-(v-\kappa \delta))\inf\limits_{a\in [l\kappa\delta\varepsilon, (l+1)\kappa\delta\varepsilon)}I^{\fwd}\left(\dfrac{a}{v-(v-\kappa \delta)}\right)
\\
\leq & -\delta I^\fwd(l\varepsilon)+\delta \omega^\fwd(\delta; \varepsilon) \ ,
\end{array}$$
where $\omega^\fwd(\delta;\varepsilon)$ is the oscillation of $I^\fwd(\bullet)$ defined by
$$\omega^\fwd(\delta;\varepsilon)=\max\{|I^\fwd(a)-I^\fwd(a')|; a, a'\in [0, 1/\delta], |a-a'|\leq \varepsilon\} \ ,$$
which tends to $0$ with $\varepsilon$ for any fixed $\delta>0$. 

\item[(iv)] The term $\sup\limits_{-2\kappa t \delta \varepsilon\leq s-\kappa t(1-(k+l)\delta\varepsilon)\leq 0} P^W(vt-X^{vt}(s)\geq 0)$. Let $\widetilde{\kappa}=\kappa(1-(k+l+2)\delta \varepsilon)$. Assume $\widetilde{\kappa}>0$, the by the invariance of $a$ with respect to the choice of $\kappa$ as we have demonstrated below (\ref{Thm:LDPProcessX:Eq:UpperBoundPositive:SmallcCase:cIs0:Decomposition:a}), we have
\begin{eqnarray}\label{eqn:def}a=\limsup\limits_{t\rightarrow\infty}\dfrac{1}{\widetilde{\kappa}t}\sup\limits_{s: -2\widetilde{\kappa}t\delta\varepsilon\leq s - \widetilde{\kappa}t\leq 0}\ln P^W\left(vt-X^{vt}(s)\geq 0\right) \ .\end{eqnarray}

From here, using (\ref{Lm:SupProbabilityCompareOnVerySmallIntervalDifference:EqDifferenceEqual}) in Lemma \ref{Lm:SupProbabilityCompareOnVerySmallIntervalDifference} we get
$$ \limsup\limits_{t\rightarrow\infty}\dfrac{1}{\kappa t}\ln \sup\limits_{-2\kappa t \delta \varepsilon\leq s-\kappa t(1-(k+l)\delta\varepsilon)\leq 0} P^W(vt-X^{vt}(s)\geq 0)
\\
\leq (1-(k+l+2)\delta\varepsilon)a \ .
$$
If $\widetilde{\kappa}=0$, then the above estimate also following from (\ref{Lm:SupProbabilityCompareOnVerySmallIntervalDifference:EqDifferenceEqual}) in Lemma \ref{Lm:SupProbabilityCompareOnVerySmallIntervalDifference} by setting the right-hand side to be $0$.

\item[(v)] The term $P^W\left(T^{vt}_{vt-\kappa t \delta}<\nu \kappa t \delta\right)$. Since $\nu<\mu'(0)$, we apply (\ref{Thm:LDPHittingTime:Eq:UpperBoundReducedToIntervalPart1}) to get
$$\limsup\limits_{t\rightarrow\infty}\dfrac{1}{\kappa t}\ln P^W\left(T^{vt}_{vt-\kappa t \delta}<\nu \kappa t \delta\right)\leq - \delta I(\nu) \ .$$
\end{itemize}

Combining (i)-(v) together with (\ref{Thm:LDPProcessX:Eq:UpperBoundPositive:SmallcCase:cIs0:Decomposition}), (\ref{Thm:LDPProcessX:Eq:UpperBoundPositive:SmallcCase:cIs0:BoundFora}) we see that for some $C>0$ we have 
$$\begin{array}{ll}
a & \leq C\delta \varepsilon + \max\left\{-\delta I\left(\dfrac{1}{\delta}\right), \max\limits_{k\geq \frac{\nu}{\varepsilon}; l\geq 0; (k+l+2)\varepsilon\leq \frac{1}{\delta}}\left[-\delta I(k\varepsilon)+\delta\omega(\delta;\varepsilon)\right.\right.
\\
& \qquad \left.\left.-\delta I^\fwd(l\varepsilon)+\delta \omega^\fwd(\delta; \varepsilon)+(1-(k+l+2)\delta\varepsilon)a\right], -\delta I (\nu)\right\} \ .
\end{array}$$
By Lemma \eqref{Lm:ConvexityEstimateIbwdIfwdSum} , since $\delta<\dfrac{1}{\mu'(0)}$, we obtain $$I(k\varepsilon)+I^\fwd(l\varepsilon)\geq (k+l)\varepsilon \lim\limits_{\delta\rightarrow 0}\delta I\left(\dfrac{1}{\delta}\right)\geq (k+l)\varepsilon \delta I\left(\dfrac{1}{\delta}\right) \ .$$

\noindent Set $a'=a+\delta I\left(\dfrac{1}{\delta}\right)$. From the above we have the estimate
$$\begin{array}{ll}
a'\leq & C\delta \varepsilon +
\\
& \max\left\{0, \max\limits_{k\geq \frac{\nu}{\varepsilon}; l\geq 0; (k+l+2)\varepsilon\leq \frac{1}{\delta}}\left[\delta\omega(\delta;\varepsilon)+\delta\omega^\fwd(\delta;\varepsilon)+2\delta\varepsilon \cdot \delta I\left(\dfrac{1}{\delta}\right)+(1-(k+l+2)\delta\varepsilon)a'\right],\right.\\
& \qquad \qquad \qquad \left. \delta\left[I\left(\dfrac{1}{\delta}\right)-I(\nu)\right]\right\} \ .
\end{array}$$
Since when $k\geq \dfrac{\nu}{\varepsilon}, l\geq 0, (k+l+2)\varepsilon\leq \dfrac{1}{\delta}$ we must have $1-(k+l+2)\delta\varepsilon\leq 1-\nu\delta$, this gives
\begin{align*}
\nu \delta a'\leq\,  & C\delta \varepsilon +
\\
& \max\left\{0, \left[\delta\omega(\delta;\varepsilon)+\delta\omega^\fwd(\delta;\varepsilon)+2\delta\varepsilon \cdot \delta I\left(\dfrac{1}{\delta}\right)\right], \delta\left[I\left(\dfrac{1}{\delta}\right)-I(\nu)\right]\right\} \ .
\end{align*}
Simplifying the above by canceling $\delta$ and divide by $\nu$ we get
\begin{align*}
a'\leq\, & C\dfrac{\varepsilon}{\nu} +
\\
& \max\left\{0, \dfrac{1}{\nu}\left[\omega(\delta;\varepsilon)+\omega^\fwd(\delta;\varepsilon)+2\varepsilon \cdot \delta I\left(\dfrac{1}{\delta}\right)\right], \dfrac{1}{\nu}\left[I\left(\dfrac{1}{\delta}\right)-I(\nu)\right]\right\} \ .
\end{align*}
We first send $\varepsilon\rightarrow 0$ and we get
$$a'\leq \max\left\{0, \dfrac{1}{\nu}\left[I\left(\dfrac{1}{\delta}\right)- I(\nu)\right]\right\}\ .$$

\noindent By property (3) in Lemma \ref{Lm:PropertiesOfIInOurCase}, we then send $\nu\rightarrow 0$ to ensure that $I\left(\dfrac{1}{\delta}\right)-I(\nu)<0$. This settles $a'\leq 0$, i.e., $a\leq -\delta I\left(\dfrac{1}{\delta}\right)$ for any $\delta>0$. Lastly we send $\delta\rightarrow 0$ and we use (\ref{Thm:LDPProcessX:Eq:UpperBoundPositive:SmallcCase:cIs0:BoundFora}) to conclude that almost surely with respect to $\mathbf{P}$ we have
\begin{equation}\label{Thm:LDPProcessX:Eq:UpperBoundPositive:SmallcCase:cIs0:FinalEstimate}
\limsup\limits_{t\rightarrow\infty} \dfrac{1}{\kappa t}\ln P^W\left(\dfrac{vt-X^{vt}(\kappa t)}{\kappa t}\geq 0\right) \leq -\lim\limits_{\delta\rightarrow 0}\delta I\left(\dfrac{1}{\delta}\right) \ .
\end{equation}

\noindent For any arbitrary $0< c < \dfrac{1}{\mu'(0)}$, using the fact that $$\left\{\dfrac{vt-X^{vt}(\kappa t)}{\kappa t}\geq c\right\}=\left\{X^{vt}(\kappa t)\leq (v-c\kappa)t\right\}\subseteq \{T^{vt}_{vt-c\kappa t}\leq \kappa t\} \ ,$$ we can write
\begin{equation}\label{Thm:LDPProcessX:Eq:UpperBoundPositive:SmallcCase:cIsPositive:Decomposition}
\begin{array}{ll}
P^W\left(\dfrac{vt-X^{vt}(\kappa t)}{\kappa t}\geq c\right) & = P^W\left(T^{vt}_{vt-c\kappa t}\leq \kappa t, X^{vt}(\kappa t)\leq (v-c\kappa)t\right)
\\
& \leq \sum\limits_{0\leq k\leq \frac{1}{\varepsilon}} P^W\left(\dfrac{T^{vt}_{vt-c\kappa t}}{\kappa t}\in [k\varepsilon, (k+1)\varepsilon) \right)
\\
&  \ \ \cdot \sup\limits_{s: \kappa t (1-(k+1)\varepsilon)\leq s\leq \kappa t (1-k \varepsilon)} P^W((v-c\kappa) t-X^{(v-c\kappa) t}(s)\geq 0) \ .
\end{array}
\end{equation}
Similar to part (ii) in the right-hand-side of (\ref{Thm:LDPProcessX:Eq:UpperBoundPositive:SmallcCase:cIs0:Decomposition}) we can estimate
$$\begin{array}{ll}
& \limsup\limits_{t\rightarrow\infty}\dfrac{1}{\kappa t}\ln P^W\left(\dfrac{T^{vt}_{vt-c\kappa t}}{\kappa t}\in [k\varepsilon, (k+1)\varepsilon)\right)
\\
=& \limsup\limits_{t\rightarrow\infty}\dfrac{1}{\kappa t}\ln P^W\left(\dfrac{T^{vt}_{vt-c\kappa t}}{t}\in [k\kappa \varepsilon, (k+1)\kappa\varepsilon)\right)
\\
\leq &  -\dfrac{1}{\kappa}(v-(v-c\kappa))\inf\limits_{a\in [k\kappa\varepsilon, (k+1)\kappa\varepsilon)}I\left(\dfrac{a}{(v-(v-c\kappa))}\right)
\\
\leq & -c I\left(\dfrac{k\varepsilon}{c}\right)+\delta \omega(c;\varepsilon) \ ,
\end{array}$$
where $\omega(c;\varepsilon)$ is the oscillation of $I(\bullet)$ defined by 
$$\omega(c;\varepsilon)=\max\left\{|I(a)-I(a')|; a, a'\in [0, c^{-1}], |a-a'|\leq \varepsilon\right\} \ ,$$
which tends to $0$ with $\varepsilon$ for any fixed $c>0$.

Using (\ref{Thm:LDPProcessX:Eq:UpperBoundPositive:SmallcCase:cIs0:FinalEstimate}) we also obtain
$$\begin{array}{ll}
& \limsup\limits_{t\rightarrow\infty}\dfrac{1}{\kappa t}\sup\limits_{s: \kappa t (1-(k+1)\varepsilon)\leq s\leq \kappa t (1-k \varepsilon)} P^W((v-c\kappa) t-X^{(v-c\kappa) t}(s)\geq 0)
\\
\leq &  -(1-k\varepsilon)\lim\limits_{\delta\rightarrow 0}\delta I\left(\dfrac{1}{\delta}\right) \ .
\end{array}$$
Thus by (\ref{Thm:LDPProcessX:Eq:UpperBoundPositive:SmallcCase:cIsPositive:Decomposition}) we get, for $0<c<\dfrac{1}{\mu'(0)}$, that

\begin{equation}\label{Thm:LDPProcessX:Eq:UpperBoundPositive:SmallcCase:cIsPositive:FinalEstimate}
\begin{array}{ll}
P^W\left(\dfrac{vt-X^{vt}(\kappa t)}{\kappa t}\geq c\right) & \leq \limsup\limits_{\varepsilon\rightarrow 0}\max\limits_{0\leq k \leq \frac{1}{\varepsilon}}\left\{-cI\left(\dfrac{k\varepsilon}{c}\right)-(1-k\varepsilon)\lim\limits_{\delta\rightarrow 0}\delta I\left(\dfrac{1}{\delta}\right)\right\}
\\
& \leq \limsup\limits_{\varepsilon\rightarrow 0}\left[-\min\limits_{0\leq k \leq \frac{1}{\varepsilon}}\left\{k\varepsilon\dfrac{c}{k\varepsilon}I\left(\dfrac{k\varepsilon}{c}\right)+(1-k\varepsilon)\lim\limits_{\delta\rightarrow 0}\delta I\left(\dfrac{1}{\delta}\right)\right\}\right]
\\
& = -cI\left(\dfrac{1}{c}\right) \ ,
\end{array}
\end{equation}
where we have used convexity of the function $a\mapsto aI\left(\dfrac{1}{a}\right)$ (Property (8) of Lemma \ref{Lm:PropertiesOfIInOurCase}). Combining (\ref{Thm:LDPProcessX:Eq:UpperBoundReducedToIntervalPart1}), (\ref{Thm:LDPProcessX:Eq:UpperBoundPositive:SmallcCase:cIs0:FinalEstimate}) and (\ref{Thm:LDPProcessX:Eq:UpperBoundPositive:SmallcCase:cIsPositive:FinalEstimate}), we derive (\ref{Thm:LDPProcessX:Eq:UpperBoundPositive}).

\noindent{\textit{Proof of \emph{(\ref{Thm:LDPProcessX:Eq:LowerBoundPositive})}}}.
To prove \eqref{Thm:LDPProcessX:Eq:LowerBoundPositive}, we follow the method introduced in \cite[Theorem 5]{FanHuTerlovCMP}, \cite[Section 5]{Taleb2001}. Let $u>0$, $\varepsilon>0$ and $\delta>0$ be given. Denote by $B_\delta(u)=(u-\delta, u+\delta)$. We then have the identity 
\begin{equation}\label{Thm:LDPProcessX:LowerBoundPositive:Proof:TurnDeviationToCenterEstimate}
P^W\left(\dfrac{vt - X^{vt}(\kappa t)}{\kappa t}\in B_\delta(u)\right)
= P^W\left(|X^{vt}(\kappa t)- (v-\kappa u)t|<\kappa t \delta\right) \ .
\end{equation}

We split the event $\{(1-\varepsilon)\kappa t < T^{vt}_{(v-\kappa u)t} < \kappa\}$ into two parts depending on whether or not $|X^{vt}(\kappa t) - (v-\kappa u) t| < \kappa t \delta$,

\begin{equation}\label{Thm:LDPProcessX:LowerBoundPositive:Proof:SplitEstimateHittingTimeTwoParts}
\begin{array}{ll}
& P^W\left((1-\varepsilon) \kappa t < T^{vt}_{(v-\kappa u) t} < \kappa t\right)
\\
\leq & P^W\left(|X^{vt}(\kappa t)-(v-\kappa u)t|<\kappa t \delta\right)
\\
& \qquad + P^W\left(|X^{vt}(\kappa t) -(v-\kappa u) t|\geq \kappa t \delta; (1-\varepsilon)\kappa t < T^{vt}_{(v-\kappa u)t}<\kappa t\right) \ . 
\end{array}
\end{equation}

\noindent Combining (\ref{Thm:LDPProcessX:LowerBoundPositive:Proof:TurnDeviationToCenterEstimate}) and (\ref{Thm:LDPProcessX:LowerBoundPositive:Proof:SplitEstimateHittingTimeTwoParts}) we see that

\begin{equation}\label{Thm:LDPProcessX:LowerBoundPositive:Proof:CenterEstimateReductionToHittingTime}
\begin{array}{ll}
& P^W\left(\dfrac{vt-X^{vt}(\kappa t)}{\kappa t}\in B_\delta(u)\right)
\\
\geq & P^W\left((1-\varepsilon)\kappa t < T^{vt}_{(v-\kappa u)t}<\kappa t\right)
\\
& \qquad - P^W\left(|X^{vt}(\kappa t) -(v-\kappa u)t|\geq \kappa t\delta; (1-\varepsilon)\kappa t < T^{vt}_{(v-\kappa u)t}< \kappa t\right) \ .
\end{array}
\end{equation}

\noindent The last term 

\begin{equation}\label{Thm:LDPProcessX:LowerBoundPositive:Proof:HittingTimeRemainderEstimate1}
\begin{array}{ll}
& P^W\left(|X^{vt}(\kappa t) -(v-\kappa u)t|\geq \kappa t\delta; \ (1-\varepsilon)\kappa t < T^{vt}_{(v-\kappa u)t}< \kappa t\right)
\\
\leq & P^W\left(\sup\limits_{0<s-T^{vt}_{(v-\kappa u)t}<\varepsilon \kappa t} |X^{vt}(s)-(v-\kappa u) t|\geq \kappa t \delta\right)
\\
= & P^W\left(\sup\limits_{0<s<\varepsilon \kappa t} |X^{(v-\kappa u)t}(s)-(v-\kappa u)t|\geq \kappa t \delta\right) \ ,
\end{array}
\end{equation}
where the first inequality is due to the fact that $0<\kappa t - T^{vt}_{(v-\kappa u) t}< \varepsilon \kappa t$ and the second inequality is due to the strong Markov property of $X_t$. The last term in (\ref{Thm:LDPProcessX:LowerBoundPositive:Proof:HittingTimeRemainderEstimate1}) is turned into the hitting time by observing that 

\begin{equation}\label{Thm:LDPProcessX:LowerBoundPositive:Proof:HittingTimeRemainderEstimate2}
\begin{array}{ll}
& P^W\left(\sup\limits_{0<s<\varepsilon \kappa t} |X^{(v-\kappa u)t}(s)-(v-\kappa u)t|\geq \kappa t \delta\right)
\\
= & P^W\left(T^{(v-\kappa u)t}_{(v-\kappa u)t-\kappa t \delta}\wedge T^{(v-\kappa u)t}_{(v-\kappa u)t+\kappa t \delta}<\varepsilon \kappa t\right)
\\
\leq & P^W\left(T^{(v-\kappa u)t}_{(v-\kappa u)t-\kappa t \delta}<\varepsilon \kappa t\right) + P^W\left(T^{(v-\kappa u)t}_{(v-\kappa u)t+\kappa t \delta}<\varepsilon \kappa t\right) \ .
\end{array}
\end{equation}

\noindent By Lemma \ref{Lm:SuperExponentialSmallProbabilityOfHittingTimeAtLinearDistance}, we have

\begin{equation}\label{Thm:LDPProcessX:LowerBoundPositive:Proof:HittingTimeRemainderEstimate3}
\lim\limits_{\varepsilon \rightarrow 0} \limsup\limits_{t\rightarrow \infty} \dfrac{1}{t}\ln\left[P^W\left(T^{(v-\kappa u)t}_{(v-\kappa u)t-\kappa t \delta}<\varepsilon \kappa t\right) + P^W\left(T^{(v-\kappa u)t}_{(v-\kappa u)t+\kappa t \delta}<\varepsilon \kappa t\right)\right] = -\infty \ .
\end{equation}

\noindent Then by (\ref{Thm:LDPProcessX:LowerBoundPositive:Proof:SplitEstimateHittingTimeTwoParts}), (\ref{Thm:LDPProcessX:LowerBoundPositive:Proof:HittingTimeRemainderEstimate3}) combined with (\ref{Thm:LDPHittingTime:Eq:LowerBound}) in Theorem \ref{Thm:LDPHittingTime}, we obtain 

$$\begin{array}{ll}
 & \liminf\limits_{t \rightarrow \infty} \dfrac{1}{t} \ln P^W\left(\dfrac{vt - X^{vt}(\kappa t)}{\kappa t} \in B_\delta(u)\right)
\\
\geq & \liminf\limits_{\varepsilon \rightarrow 0}\liminf\limits_{t\rightarrow \infty} \dfrac{1}{t} \ln P^W\left(T^{vt}_{(v-\kappa u)t}\in ((1-\varepsilon)\kappa t, \kappa t)\right)
\\
= & -\kappa u I\left(\dfrac{1}{u}\right) \ ,
\end{array}$$
which proves the lower bound (\ref{Thm:LDPProcessX:Eq:LowerBoundPositive}).

\noindent{\textit{Proof of \emph{(\ref{Thm:LDPProcessX:Eq:UpperBoundNegative})} and \emph{(\ref{Thm:LDPProcessX:Eq:LowerBoundNegative})}}}. This follows the same rationale as we did in the proof of (\ref{Thm:LDPProcessX:Eq:UpperBoundPositive}) and (\ref{Thm:LDPProcessX:Eq:LowerBoundPositive}), so we omit the details.
\end{proof}

\begin{lemma}\label{Lm:WorstCaseEnvironmentEstimateForaInThmLDPProcessX}
For any $\kappa\in (0,1]$, there exists some $C>0$ such that for any $v>0$ and any $\delta>0$, $\varepsilon>0$ small enough we have
$$\limsup\limits_{t\rightarrow\infty}\dfrac{1}{\kappa t} \ln \sup\limits_{s: \  -2t\delta\varepsilon\leq s-\kappa t\leq 0}\inf\limits_{w\leq vt} P^W\left(X^w(\kappa t-s)\leq w - (\kappa t-s)\right)\geq -C\delta\varepsilon \ .$$
\end{lemma}

\begin{proof}
By (\ref{Eq:SDE-RandomDrift}) and the spacial stationarity of $b(x)$ we know that $X^w(\kappa t - s)-w$ has the same distribution as $X^0(\kappa t - s)$, where $\displaystyle{X^0(t)=\int_0^t b(X^0(s)){\rm d}s+W_t}$. This gives
$$\begin{array}{ll}
& \sup\limits_{s: \  -2t\delta\varepsilon\leq s-\kappa t\leq 0}\inf\limits_{w\leq vt} P^W\left(X^w(\kappa t-s)\leq w - (\kappa t-s)\right)
\\
= & \sup\limits_{0\leq r\leq 2t\delta\varepsilon} P^W\left(X^0(r)\leq -r\right) 
\\
\geq & P^W\left(X^0(2t\delta\varepsilon)\leq -2t\delta \varepsilon\right)
\\
= & P^W\left(\displaystyle{\int_0^{2t\delta\varepsilon}b(X^0(s)){\rm d}s}+W_{2t\delta\varepsilon}\leq -2 t \delta \varepsilon\right) \ .
\end{array}$$

\noindent Since we have assumed that $b(x)\leq B<+\infty$ for all $x$ (Assumption (3) in Section \ref{Sec:Introduction}), this gives
$$\begin{array}{ll}
P^W\left(\displaystyle{\int_0^{2t\delta\varepsilon}b(X^0(s)){\rm d}s}+W_{2t\delta\varepsilon}\leq -2 t \delta \varepsilon\right) & \geq P^W\left(W_{2t\delta\varepsilon}\leq -2t\delta\varepsilon(1+B)\right)
\\
& = \displaystyle{\int_{-\infty}^{-2t\delta\varepsilon(1+B)} \dfrac{1}{\sqrt{2\pi \cdot 2t\delta\varepsilon}}e^{-\frac{x^2}{2\cdot 2t\delta\varepsilon}}{\rm d}x}
\\
& = 1-\Phi\left(\sqrt{2t\delta\varepsilon}(1+B)\right)
\\
& \geq e^{-2t\delta\varepsilon(1+B)^2} \ ,
\end{array}$$
where we have used the Gaussian tail bound $1-\Phi(x)\geq e^{-x^2}$.
Thus taking the limit $t\rightarrow\infty$ we get
$$\limsup\limits_{t\rightarrow\infty} \dfrac{1}{\kappa t}\ln P^W(X^0(2t\delta\varepsilon)\leq - 2t\delta\varepsilon)\geq \limsup\limits_{t\rightarrow\infty} \dfrac{1}{\kappa t}\cdot \left(-2t\delta\varepsilon(1+B)^2\right)=-\dfrac{2(1+B)^2}{\kappa}\delta\varepsilon \ .$$
Thus we pick $C=\dfrac{2(1+B)^2}{\kappa}>0$ and obtain the conclusion of this Lemma.
\end{proof}

\begin{lemma}\label{Lm:SupProbabilityCompareOnVerySmallIntervalDifference}
Let $\kappa\in (0,1]$, $\varepsilon>0$, $\delta>0$, integers $k,l\geq 0$, $(k+l+2)\varepsilon\delta\leq 1$ and $\widetilde{\kappa}=\kappa(1-(k+l+2)\delta\varepsilon)$. Then for small enough $\varepsilon>0$, $\delta>0$ we have
\begin{equation}\label{Lm:SupProbabilityCompareOnVerySmallIntervalDifference:EqDifferenceEqual}
\begin{array}{ll}
& \limsup\limits_{t\rightarrow\infty }\dfrac{1}{\kappa t}\ln\sup\limits_{-2\kappa t \delta \varepsilon\leq s-\kappa t(1-(k+l)\delta\varepsilon)\leq 0} P^W(vt-X^{vt}(s)\geq 0)
\\
\leq & \limsup\limits_{t\rightarrow\infty}\dfrac{1}{\kappa t}\ln \sup\limits_{s: -2\widetilde{\kappa}t\delta\varepsilon\leq s-\widetilde{\kappa}t\leq 0} P^W\left(vt-X^{vt}(s)\geq 0\right) \ .
\end{array}
\end{equation}
\end{lemma}

\begin{proof} 
Consider the consecutive intervals
$$I_1 = \left[\widetilde{\kappa} t(1-2\delta\varepsilon), \widetilde{\kappa} t\right] \ , \ 
I_2 = \left(\widetilde{\kappa} t, \kappa t(1-(k+l)\delta\varepsilon)\right]\ . $$

\noindent We first assume $(k+l+2)\varepsilon\delta<1$, so that $\widetilde{\kappa}>0$. Recall from \eqref{eqn:def} that
\begin{eqnarray*}a=\limsup\limits_{t\rightarrow\infty}\dfrac{1}{\widetilde{\kappa}t}\sup\limits_{s\in I_1}\ln P^W\left(vt-X^{vt}(s)\geq 0\right) \ .\end{eqnarray*}
Thus
\begin{eqnarray}\label{lem3.10a}
a(1-(k+l+2)\delta\varepsilon)=\limsup\limits_{t\rightarrow\infty}\dfrac{1}{{\kappa}t}\sup\limits_{s\in I_1}\ln P^W\left(vt-X^{vt}(s)\geq 0\right) \ .\end{eqnarray}

\noindent One also note that
\[
I_2=[\kappa_1(1-2\delta\varepsilon)t, \kappa_1 t]\cup [\kappa_2(1-\delta\varepsilon)t, \kappa_2 t]=:I_{21}\cup I_{22},
\]
where $\kappa_1=\frac{\widetilde{\kappa}}{1-2\delta\varepsilon}$ 
and $\kappa_2=\kappa (1-(k+l)\delta\varepsilon).$ Then from the invariance of $a$ with respect to $\kappa$ as we demonstrated below (\ref{Thm:LDPProcessX:Eq:UpperBoundPositive:SmallcCase:cIs0:Decomposition:a}), one has
\begin{align}\label{lem3.10b}
\limsup\limits_{t\rightarrow\infty}\dfrac{1}{{\kappa}t}\sup\limits_{s\in I_{21}}\ln P^W\left(vt-X^{vt}(s)\geq 0\right)=a\frac{\kappa_1}{\kappa}=a\frac{1-(k+l+2)\delta\varepsilon}{1-2\delta\varepsilon}    
\end{align}
and
\begin{align}\label{lem3.10c}
\limsup\limits_{t\rightarrow\infty}\dfrac{1}{{\kappa}t}\sup\limits_{s\in I_{22}}\ln P^W\left(vt-X^{vt}(s)\geq 0\right)=a\frac{\kappa_2}{\kappa}=a(1-(k+l+2)\delta\varepsilon)\ .
\end{align}

\noindent Therefore
\begin{align}
 & \limsup\limits_{t\rightarrow\infty}\dfrac{1}{{\kappa}t}\sup\limits_{s\in I_{2}}\ln P^W\left(vt-X^{vt}(s)\geq 0\right) \cr
 & \quad 
 \leq     \limsup\limits_{t\rightarrow\infty}\dfrac{1}{{\kappa}t}\sup\limits_{s\in I_{1}}\ln P^W\left(vt-X^{vt}(s)\geq 0\right)
+\frac{2a(1-(k+l+2)\delta\varepsilon)\delta\varepsilon}{1-2\delta\varepsilon}\ .
\end{align}
Since $(k+l+2)\delta\varepsilon <1$, we have
\[
0<\frac{1-(k+l+2)\delta\varepsilon}{1-2\delta\varepsilon}=1-\frac{(k+l)\delta\varepsilon}{1-2\delta\varepsilon}<1 ,
\]
this together with the fact that $a\leq 0$ yield (\ref{Lm:SupProbabilityCompareOnVerySmallIntervalDifference:EqDifferenceEqual}).

Now we consider the case when $(k+l+2)\varepsilon\delta=1$, in which case $\widetilde{\kappa}=0$, and trivially
\begin{equation}\label{Lm:SupProbabilityCompareOnVerySmallIntervalDifference:EqDifferenceEqual:ZeroCase}
\limsup\limits_{t\rightarrow\infty }\dfrac{1}{\kappa t}\ln\sup\limits_{0\leq s \leq 2\varepsilon\delta\kappa t} P^W(vt-X^{vt}(s)\geq 0)
\leq 0 \ ,
\end{equation}
which automatically yields (\ref{Lm:SupProbabilityCompareOnVerySmallIntervalDifference:EqDifferenceEqual}).
\end{proof}

\subsection{Other Auxiliary Lemmas}\label{Sec:LDP:OtherAuxiliaryLemmas}
Here we provide a few more auxiliary lemmas that will be used throughout the proof of wave-propagation.

Recall that when $s>r$. the backward hitting time $T^s_r$ is defined in (\ref{Eq:HittingTimeSDE-RandomDrift:Backward}), and when $s<r$, the forward hitting time $T^s_r$ is defined in (\ref{Eq:HittingTimeSDE-RandomDrift:Forward}). Combining these two, we get the general definition of the hitting time $T^s_r$ which has the following property we show below:

\begin{lemma}\label{Lm:SuperExponentialSmallProbabilityOfHittingTimeAtLinearDistance}
For any $x,y\in \mathbb{R}$ such that $x\neq y$, there exists some $\varepsilon_0=\varepsilon_0(x, y, b(x))>0$ depending on $x, y$ and the drift term $b(x)$, such that for any $0<\varepsilon<\varepsilon_0$ and any $M>0$, we have
\begin{equation}\label{Lm:SuperExponentialSmallProbabilityOfHittingTimeAtLinearDistance:Eq:Estimate}
\limsup\limits_{t\rightarrow\infty} \dfrac{1}{t} \ln P^W\left(T^{xt}_{yt} < \varepsilon t\right)\leq -M \ ,
\end{equation}
almost surely with respect to $\mathbf{P}$.
\end{lemma}

\begin{proof}
By Chebyshev's inequality, for any $\lambda>0$,
$$P^W(T^{xt}_{yt}<\varepsilon t)=P^W\left(e^{-\lambda T^{xt}_{yt}}>e^{-\lambda \varepsilon t}\right)\leq e^{\lambda \varepsilon t}E^W e^{-\lambda T^{at}_{bt}} \ ,$$
and therefore
\begin{equation}\label{Lm:SuperExponentialSmallProbabilityOfHittingTimeAtLinearDistance:Eq:EstimateViaExpMoment}
\dfrac{1}{t}\ln P^W\left(T^{xt}_{yt}<\varepsilon t\right)\leq \lambda \varepsilon + \dfrac{1}{t}\ln E^W e^{-\lambda T^{xt}_{yt}} \ .
\end{equation}

\noindent  We will show that, for some $C=C(b(x))>0$ depending on the drift term $b(x)$,
\begin{equation}\label{Lm:SuperExponentialSmallProbabilityOfHittingTimeAtLinearDistance:Eq:UpperBoundExpMoment}
\limsup\limits_{t\rightarrow\infty} \dfrac{1}{t}\ln E^W e^{-\lambda T^{xt}_{yt}}\leq -C\lambda |y-x| \ .
\end{equation}

\noindent  Assume (\ref{Lm:SuperExponentialSmallProbabilityOfHittingTimeAtLinearDistance:Eq:UpperBoundExpMoment}) holds, then by (\ref{Lm:SuperExponentialSmallProbabilityOfHittingTimeAtLinearDistance:Eq:EstimateViaExpMoment}) we can bound

\begin{equation}\label{Lm:SuperExponentialSmallProbabilityOfHittingTimeAtLinearDistance:Eq:EstimateViaLinearDistance}
\limsup\limits_{t\rightarrow\infty} \dfrac{1}{t} \ln P^W(T^{xt}_{yt}<\varepsilon t)\leq \lambda \varepsilon - \lambda C |y-x| = -\lambda [C|y-x| - \varepsilon] \ .
\end{equation}
We then pick $0<\varepsilon_0<\dfrac{1}{2}C|b-a|$ and choose $\lambda>\dfrac{2M}{C|b-a|}$ to conclude (\ref{Lm:SuperExponentialSmallProbabilityOfHittingTimeAtLinearDistance:Eq:Estimate}).

It remains to prove (\ref{Lm:SuperExponentialSmallProbabilityOfHittingTimeAtLinearDistance:Eq:UpperBoundExpMoment}). Without loss of generality we assume $x<y$. Suppose $t$ is large, we find $i\in \mathbb{Z}$ such that 
$$i-1 < xt \leq i < i+1 < ... < i+n-1 \leq yt < i+n$$
where $n=n(x,y,t)$ is the number of integer points between $xt$ and $yt$. This implies $n-1 \leq (y-x)t < n+1$, leading to 
\begin{equation}\label{Lm:SuperExponentialSmallProbabilityOfHittingTimeAtLinearDistance:Eq:LimitnOvert}
\lim\limits_{t\rightarrow\infty}\dfrac{n}{t}=|y-x| \ .
\end{equation}

\noindent The strong Markov property of $X^x(t)$  in (\ref{Eq:SDE-RandomDrift}), and the stationary assumption on $b(x)$, implies the i.i.d. decomposition (under the environment probability $\mathbf{P}$)
$$\ln E^W e^{-\lambda T^{xt}_{yt}}=\ln E^W e^{-\lambda T^{xt}_{i}} + \ln E^W e^{-\lambda T^{i}_{i+1}} + ... + \ln E^W e^{-\lambda T^{i+n-2}_{i+n-1}} + \ln E^W e^{-\lambda T^{i+n-1}_{yt}} \ ,$$
which, by the Law of Large Numbers and (\ref{Lm:SuperExponentialSmallProbabilityOfHittingTimeAtLinearDistance:Eq:LimitnOvert}), further imply that
\begin{equation}\label{Lm:SuperExponentialSmallProbabilityOfHittingTimeAtLinearDistance:Eq:UpperBoundExpMomentViaT01}
\lim\limits_{t\rightarrow\infty} \dfrac{1}{t} \ln E^W e^{-\lambda T^{xt}_{yt}}=|y-x|\cdot \mathbf{E}\left(\ln E^W e^{-\lambda T^{0}_{1}}\right) \ .
\end{equation} 
Since $X^x(t)$ defined in (\ref{Eq:SDE-RandomDrift}) is a diffusion process with unit diffusivity and the drift term $b(x)$, there exists some $\kappa=\kappa(b(x))>0$ such that $T^0_1\geq \kappa>0$ with $P^W$--probability 1. This means, from (\ref{Lm:SuperExponentialSmallProbabilityOfHittingTimeAtLinearDistance:Eq:UpperBoundExpMomentViaT01}), we get (\ref{Lm:SuperExponentialSmallProbabilityOfHittingTimeAtLinearDistance:Eq:UpperBoundExpMoment}).
\end{proof}

\begin{lemma}\label{Lm:ExpoentialSmall ProbabilityOfHittingTimeFromZeroToLinearDistance}
For any bounded set $C \subset \left\{c>0: cI\left(\dfrac{1}{c}\right)-\beta > 0\right\}$ and any small $\delta>0$, there is a finite constant $K>0$ such that
\begin{equation}\label{Lm:ExpoentialSmall ProbabilityOfHittingTimeFromZeroToLinearDistance:Eq:Estimate}
\liminf\limits_{t\rightarrow\infty} \ln\left(\inf\limits_{\widetilde{c}\in B_{\delta}(c)}P^W\left(X^{\widetilde{c}t}(t)\in B_\delta(0)\right)\right)\geq -K
\end{equation}
uniformly over all $c\in C$ such that $B_\delta(c)\subset \left\{c>0: cI\left(\dfrac{1}{c}\right)-\beta > 0\right\}$.

Similarly, for any bounded set $C^\fwd \subset \left\{c>0: cI^\fwd\left(\dfrac{1}{c}\right)-\beta > 0\right\}$ and any small $\delta>0$, there is a finite constant $K>0$ such that
\begin{equation}\label{Lm:ExpoentialSmall ProbabilityOfHittingTimeFromZeroToLinearDistance:Forward:Eq:Estimate}
\liminf\limits_{t\rightarrow\infty} \ln\left(\inf\limits_{\widetilde{c}\in B_\delta(c)}P^W\left(X^{-\widetilde{c}t}(t)\in B_\delta(0)\right)\right)\geq -K
\end{equation}
uniformly over all $c\in C^\fwd$ such that $B_\delta(c)\subset \left\{c>0: cI^\fwd\left(\dfrac{1}{c}\right)-\beta > 0\right\}$.
\end{lemma}

\begin{proof}
We first observe the event inclusion $\{T_{-\delta}^{(c+\delta)t}\leq t\}\subseteq \{X^{\widetilde{c}t}(t)\in B_\delta(0), \widetilde{c}\in B_\delta(c)\}$, which is because a trajectory starting from $(c+\delta)t$ and hitting $-\delta$ before time $t$, belongs to $\{X^{\widetilde{c}t}(t)\in B_\delta(0), \widetilde{c}\in B_\delta(c)\}$, due to the continuity of $X^x(t)$. Thus by this event inclusion we have 
$$\inf\limits_{\widetilde{c}\in B_\delta(c)}P^W\left(X^{\widetilde{c}t}(t)\in B_\delta(0)\right)\geq P^W\left(\dfrac{T^{(c+\delta)t}_{-\delta}}{t}\leq 1\right) \ .$$

\noindent From here, we estimate
$$\begin{array}{ll}
\liminf\limits_{t\rightarrow\infty} \dfrac{1}{t}\ln\left(\inf\limits_{\widetilde{c}\in B_\delta(c)}P^W\left(X^{\widetilde{c}t}(t)\in B_\delta(0)\right)\right)& \geq \liminf\limits_{t\rightarrow\infty}\dfrac{1}{t}\ln P^W\left(\dfrac{T_{-\delta}^{(c+\delta)t}}{t}\leq 1\right)
\\
& \geq \liminf\limits_{t\rightarrow\infty}\dfrac{1}{t}\ln P^W\left(\dfrac{T^{(c+\delta)t}_{-\delta t}}{t}\leq 1\right)
\\
& \geq -(c+2\delta)\inf\limits_{a\in (0,1)}I\left(\dfrac{a}{c+2\delta}\right) \ ,
\end{array}$$
where the last equality is due to (\ref{Thm:LDPHittingTime:Eq:LowerBound}) in Theorem \ref{Thm:LDPHittingTime}. By Lemma \ref{Lm:PropertiesOfIInOurCase}, we have $$\inf\limits_{a\in (0,1)}I\left(\dfrac{a}{c+2\delta}\right)\leq I\left(\dfrac{1}{c+2\delta}\right) \ .$$ Thus when $\delta<1$ we can set $K=\sup\limits_{c\in C}\left[(c+2)I\left(\dfrac{1}{c+2}\right)\right]>0$ to obtain (\ref{Lm:ExpoentialSmall ProbabilityOfHittingTimeFromZeroToLinearDistance:Eq:Estimate}), using the fact that $I(a)>0$ via Lemma \ref{Lm:FurtherPropertiesOfMuAndIUsedInWaveFrontShape}.

The proof of (\ref{Lm:ExpoentialSmall ProbabilityOfHittingTimeFromZeroToLinearDistance:Forward:Eq:Estimate}) follows a similar argument in parallel with the above, using (\ref{Thm:LDPHittingTime:Forward:Eq:LowerBound}) in Theorem \ref{Thm:LDPHittingTime:Forward}. We omit the details.
\end{proof}

\section{Wave Propagation}\label{Sec:WavePropagation}
Consider the following two equations:
\begin{equation}\label{Eq:IntersectionOfIandbetaWaveSpeed}
c I\left(\dfrac{1}{c}\right) = \beta \ ,
\end{equation}
and
\begin{equation}\label{Eq:IntersectionOfIandbetaWaveSpeed:Forward}
c I^\fwd\left(\dfrac{1}{c}\right) = \beta \ .
\end{equation}

\noindent Set 
\begin{equation}\label{Eq:PositiveAndNegativeDifferenceOfIandbetaWaveSpeed}
\Lambda_0 = \left\{c>0: cI\left(\dfrac{1}{c}\right)-\beta >0 \right\} \ , \ 
\Lambda_1 = \left\{c>0: cI\left(\dfrac{1}{c}\right)-\beta <0 \right\} \ ,
\end{equation}
and
\begin{equation}\label{Eq:PositiveAndNegativeDifferenceOfIandbetaWaveSpeed:Forward}
\Lambda_0^\fwd = \left\{c>0: cI^\fwd\left(\dfrac{1}{c}\right)-\beta >0 \right\} \ , \ 
\Lambda_1^\fwd = \left\{c>0: cI^\fwd\left(\dfrac{1}{c}\right)-\beta <0 \right\} \ .
\end{equation}

The following Theorem characterizes the sets $R_0$ and $R_1$ asked in the Wave Propagation Problem (see Section \ref{Sec:Introduction}), and the conclusion can be summarized as $R_0=\Lambda_0\cup (-\Lambda^\fwd_0)$ and $R_1=\Lambda_1\cup (-\Lambda^\fwd_1)$.

\begin{theorem}[Wave Propagation]\label{Thm:WavePropagation}
For any closed set $F\subset \Lambda_0$, any compact set $K\subset \Lambda_1$ we have
\begin{equation}\label{Thm:WavePropagation:Eq:Wave}
\lim\limits_{t\rightarrow\infty}\sup\limits_{c\in F} u(t,ct)=0 \ , \
\lim\limits_{t\rightarrow\infty}\inf\limits_{c\in K} u(t,ct)=1 \ ,
\end{equation}
almost surely with respect to $\mathbf{P}$. Similarly, for any closed set $F\subset \Lambda_0^\fwd$, any compact set $K\subset \Lambda_1^\fwd$ we have
\begin{equation}\label{Thm:WavePropagation:Forward:Eq:Wave}
\lim\limits_{t\rightarrow\infty}\sup\limits_{c\in F} u(t,-ct)=0 \ , \
\lim\limits_{t\rightarrow\infty}\inf\limits_{c\in K} u(t,-ct)=1 \ ,
\end{equation}
almost surely with respect to $\mathbf{P}$.
\end{theorem}

\begin{proof}
The conclusion of (\ref{Thm:WavePropagation:Eq:Wave}) is a result of the combination of Propositions \ref{Prop:AsymptoticSolutionZero} and \ref{Prop:AsymptoticSolutionOne}. For (\ref{Thm:WavePropagation:Forward:Eq:Wave}), the proof is similar so we do not repeat it.
\end{proof}

\begin{proposition} \label{Prop:AsymptoticSolutionZero}
For any closed set $F\subset \Lambda_0$ we have
\begin{equation}\label{Prop:AsymptoticSolutionZero:Eq:Limit}
\lim\limits_{t\rightarrow \infty}\sup\limits_{c\in F} u(t, ct) = 0
\end{equation}
almost surely with respect to $\mathbf{P}$. 
\end{proposition}

\begin{proof}
Since $F$ is closed and $I(a)$ is a continuous function in $a$, there is some $\varepsilon=\varepsilon(F)>0$ such that for all $c\in F$ we have $cI\left(\dfrac{1}{c}\right)-\beta >\varepsilon>0$. We apply (\ref{Eq:FeynmannKacRDERandomDrift}) to get the estimate

$$u(t,x)=E^W\left[u_0(X^x(t))\exp\left(\int_0^t c(u(t-s, X^x(s))){\rm d}s \right)\right]\leq \exp(\beta t)E^W u_0(X^x(t)) \ .$$
Since we assume $\text{supp}u_0\subset (-\delta, \delta)$ for some $\delta>0$, from the above estimate we further conclude that 
$$\begin{array}{ll}
u(t,ct) & \leq \|u_0\|\exp(\beta t)P^W\left(-\delta \leq X^{ct}(t)\leq \delta\right)
\\
& = \|u_0\|\exp(\beta t)P^W\left(c+\dfrac{\delta}{t}\geq \dfrac{ct-X^{ct}(t)}{t}\geq c-\dfrac{\delta}{t}\right) \ ,
\end{array}$$
where $\|u_0\|=\sup\limits_{x\in \mathbb{R}} u_0(x)>0$.

Making use of Theorem \ref{Thm:LDPProcessX} with $\kappa=1$ and $v=c$, as $t$ is large, we get 
$$\limsup\limits_{t\rightarrow\infty}\dfrac{1}{t}\ln u(t,ct)\leq \left[\beta - cI\left(\dfrac{1}{c}\right)\right]+\dfrac{\varepsilon}{2} \leq  -\dfrac{\varepsilon}{2}$$ 
almost surely with respect to $\mathbf{P}$. This implies (\ref{Prop:AsymptoticSolutionZero:Eq:Limit}).
\end{proof}

\begin{proposition} \label{Prop:AsymptoticSolutionOne}
For any compact set $K\subset \Lambda_1$ we have
\begin{equation}\label{Prop:AsymptoticSolutionOne:Eq:Limit}
\lim\limits_{t\rightarrow \infty}\inf\limits_{c\in K} u(t, ct) = 1
\end{equation}
almost surely with respect to $\mathbf{P}$.
\end{proposition}

\begin{proof}
The proof follows the argument used in the proof of Lemma 6.3 in \cite{FanHuTerlovCMP} or Lemma 7 in \cite{Nolen-XinCMP2007}. We will use the following stopping-time version of the Feynman-Kac formula (\ref{Eq:FeynmannKacRDERandomDrift}): if $\tau$ is any stopping time, then

\begin{equation}\label{Prop:AsymptoticSolutionOne:Eq:Feynmann-KacStoppingTime}
u(t,x)=E^W \left[u(t-t\wedge \tau, X^x(t\wedge \tau))\exp\left\{\int_0^{t\wedge\tau} c(u(t-s, X^x(s))){\rm d}s\right\}\right] \ .
\end{equation}
The validity of (\ref{Prop:AsymptoticSolutionOne:Eq:Feynmann-KacStoppingTime}) is justified in the same manner as \cite[Lemma 6.3]{FanHuTerlovCMP}, so we omit the details.

For any $s>0$, we construct the sets
$$\Psi(s)=\left\{c>0: cI\left(\dfrac{1}{c}\right)-\beta = s\right\}
\text{ and }
\underline{\Psi}(s)=\left\{c>0: cI\left(\dfrac{1}{c}\right)-\beta \leq  s\right\} \ .$$
For any $\delta>0$ and $T>1$ we define 
\begin{equation}\label{Prop:AsymptoticSolutionOne:Eq:GammaT}
\Gamma_T=\left([\{1\}\times \underline{\Psi}(\delta)]\cup 
\left[\bigcup\limits_{1\leq t \leq T}(\{t\}\times t \Psi(\delta))\right]\right) \ .
\end{equation}

\noindent Lemma \ref{Lm:LowerBoundSolutionOnGammat} will ensure that for sufficiently large $t$ we have 
\begin{equation}\label{Prop:AsymptoticSolutionOne:Eq:LowerBoundSolutionOnGammat}
u(s,\xi)\geq e^{-2\delta t} \text{ for all } (s, \xi)\in \Gamma_t \ .
\end{equation}
Let $c\in K\subset \Lambda_1$. By (\ref{Eq:PositiveAndNegativeDifferenceOfIandbetaWaveSpeed}) we know that $c>0$. Set $h\in (0,1)$ and $t>0$. Define the stopping times

$$\begin{array}{rcl}
\sigma_h(t) & = & \min\{s\in [0,t];  \ u(t-s, X^{ct}(s))\geq h\} \ ,
\\
\sigma_\Gamma(t) & = & \min\{s\in [0,t]; \ (t-s, X^{ct}(s))\in \Gamma_t\} \ ,
\\
\hat{\sigma}(t) & = & \sigma_h(t) \wedge \sigma_\Gamma(t) \ .
\end{array}$$

\noindent We apply (\ref{Prop:AsymptoticSolutionOne:Eq:Feynmann-KacStoppingTime}) with $\tau=\hat{\sigma}(t)$, so that we express $u(t,x)$ as
\begin{equation}\label{Prop:AsymptoticSolutionOne:Eq:Feynmann-KacStoppingTimeAppliedToHatSigma}
u(t,x)=E^W \left[u(t-t\wedge \hat{\sigma}, X^x(t\wedge \hat{\sigma}))\exp\left\{\int_0^{t\wedge \hat{\sigma}} c(u(t-s, X^x(s))){\rm d}s\right\}(\mathbf{1}_{A_1}+\mathbf{1}_{A_2}+\mathbf{1}_{A_3})\right] \ ,
\end{equation}
where 
$$\begin{array}{rcl}
A_1 & = & \{\sigma_h(t)\leq t\} \ ,
\\
A_2 & = & \{\sigma_h(t)>t, \sigma_\Gamma(t)\geq rt\} \ ,
\\
A_3 & = & \{\sigma_h(t)>t, \sigma_\Gamma(t)<rt\} \ ,
\end{array}$$
for some $r\in (0,1)$ to be chosen. Because $A_1, A_2, A_3$ are disjoint, the expectation (\ref{Prop:AsymptoticSolutionOne:Eq:Feynmann-KacStoppingTimeAppliedToHatSigma}) splits into three integrals, while the first two integrals are bounded from below respectively by

$$E^W \left[u(t-t\wedge \hat{\sigma}, X^x(t\wedge \hat{\sigma}))\exp\left\{\int_0^{t\wedge \hat{\sigma}} c(u(t-s, X^x(s))){\rm d}s\right\}\mathbf{1}_{A_1}\right] \geq h P^W(A_1) \ ,
$$
because on $A_1$ we have $u(t-t\wedge \hat{\sigma}, X^x(t\wedge \hat{\sigma}))\geq h$ and $\displaystyle{\exp\left\{\int_0^{t\wedge \hat{\sigma}} c(u(t-s, X^x(s))){\rm d}s\right\}}\geq 1$;

$$
E^W \left[u(t-t\wedge \hat{\sigma}, X^x(t\wedge \hat{\sigma}))\exp\left\{\int_0^{t\wedge \hat{\sigma}} c(u(t-s, X^x(s))){\rm d}s\right\}\mathbf{1}_{A_2}\right] \geq e^{-2\delta t}e^{c(h)rt} P^W(A_2) \ ,$$
because on $A_2$ we have $u(t-t\wedge \hat{\sigma}, X^x(t\wedge \hat{\sigma}))\geq e^{-2\delta t}$ due to (\ref{Prop:AsymptoticSolutionOne:Eq:LowerBoundSolutionOnGammat}), and $$\exp\left\{\int_0^{t\wedge \hat{\sigma}} c(u(t-s, X^x(s))){\rm d}s\right\}\geq e^{c(h) rt} \ .$$

The above two estimates combined, together with the fact that the third integral is non-negative, give
\begin{equation}\label{Prop:AsymptoticSolutionOne:Eq:LowerBoundSolutionViaParameterh}
u(t,x)\geq hP^W(A_1)+e^{-2\delta t}e^{c(h)rt}P^W(A_2) \ .
\end{equation}

\noindent Given $h,r>0$, we choose $\delta=\delta(h,r)>0$ such that $-2\delta t + c(h)rt>0$, i.e. $0<\delta<\dfrac{1}{2}c(h)r$. Since $u(t,x)\in (0,1)$, we must have $P^W(A_2)\rightarrow 0$ as $t\rightarrow \infty$ with the small $\delta>0$ chosen in our way. We will also show that $P^W(A_3)\rightarrow 0$ as $t\rightarrow \infty$. Since $P^W(A_1)+P^W(A_2)+P^W(A_3)=1$, we conclude that $P^W(A_1)\rightarrow 1$ as $t \rightarrow \infty$. This combined with (\ref{Prop:AsymptoticSolutionOne:Eq:LowerBoundSolutionViaParameterh}) gives $u(t,x)\geq h$ as $t\rightarrow\infty$ for any $h\in (0,1)$, which is (\ref{Prop:AsymptoticSolutionOne:Eq:Limit}). 

It remains to show that $P^W(A_3)\rightarrow 0$ as $t\rightarrow \infty$. Due to Lemma \ref{Lm:SolutionStructureBalanceEquation}, the set $\Lambda_1$ is either empty or an interval of the form $(L, R)$, $L< R$ (this also applies to $\Lambda^\fwd_1$), and we only have to consider the latter case. So $\Psi(0)=\{L, R\}$ and $\Psi(\delta)=\{L_{\Psi}(\delta), R_{\Psi}(\delta)\}$ where $L_{\Psi}(\delta)<L<R<R_{\Psi}(\delta)$. The initial point $x=ct$ for $c\in K\subset \Lambda_1=(L, R)$. Thus $$\sigma_\Gamma(t)\geq \min\{s\in [0,t]: X^{ct}(s)=L(t-s) \text{ or } X^{ct}(s)=R(t-s)\} \ ,$$
which leads to the inclusion of events
$$\{\sigma_\Gamma(t)\leq rt\}\subseteq \left\{\min\{s\in [0,t]: X^{ct}(s)=L(t-s) \text{ or } X^{ct}(s)=R(t-s)\}\leq rt\right\} \ .$$
Since $L<c<R$, when $r>0$ is small one can pick $\varepsilon_c>0$ independent of $r$ such that $(c-\varepsilon_c, c+\varepsilon_c)\subset \bigcap\limits_{0\leq \lambda \leq r} (L(1-\lambda), R(1-\lambda))$ . Then on the event $$\left\{\min\{s\in [0,t]: X^{ct}(s)=L(t-s) \text{ or } X^{ct}(s)=R(t-s)\}\leq rt\right\}$$ we have 
$$\min\{s\in [0,t]: X^{ct}(s)=L(t-s) \text{ or } X^{ct}(s)=R(t-s)\}\geq T^{ct}_{(c-\varepsilon_c) t}\wedge T^{ct}_{(c+\varepsilon_c) t} \ ,$$
so that
$$P^W(A_3)\leq
P^W(\sigma_\Gamma(t)\leq rt)\leq P^W\left(T^{ct}_{(c-\varepsilon_c) t}\wedge T^{ct}_{(c+\varepsilon_c) t}<rt\right)\rightarrow 0 $$ 
as $t\rightarrow\infty$, when $r>0$ is picked to be sufficiently small due to Lemma \ref{Lm:SuperExponentialSmallProbabilityOfHittingTimeAtLinearDistance}.
\end{proof}

\begin{lemma}\label{Lm:LowerBoundSolutionOnGammat}
Given any sufficiently small $\delta>0$, for sufficiently large $T>0$ we have
$$u(s,\xi)\geq e^{-2\delta T} \text{ for all } (s, \xi)\in \Gamma_T \ ,$$
where the set $\Gamma_T$ is defined in \emph{(\ref{Prop:AsymptoticSolutionOne:Eq:GammaT})}.
\end{lemma}

\begin{proof}
Given $\delta>0$, let $c\in \Psi(\delta)$, so that $c>0$ and
$cI\left(\dfrac{1}{c}\right)-\beta =\delta$. By Lemma \ref{Lm:LowerBoundBehaviorSolutionInThePositiveBranch} we know that for such $c\in \Psi(\delta)$ we have $\liminf\limits_{t\rightarrow\infty}\dfrac{1}{t}\ln u(t, ct)\geq -\left[cI\left(\dfrac{1}{c}\right)-\beta\right]=-\delta$. Thus we can find $t_1=t_1(\delta)>0$ such that for any $t>t_1$ we have
\begin{equation}\label{Lm:LowerBoundSolutionOnGammat:Eq:LowerBoundWhenTimeLarge}
u(t, ct)\geq e^{-2\delta t} \ .
\end{equation}
Let $m=\min\limits_{1\leq \widetilde{t}\leq t_1, c\in \Psi(\delta)}u(\widetilde{t}, c\widetilde{t})>0$. Pick $k\geq 2$ sufficiently large
such that $e^{-2\delta kt_1}\leq m$. Let $T=kt_1$. Then by (\ref{Lm:LowerBoundSolutionOnGammat:Eq:LowerBoundWhenTimeLarge}), we have, $$u(s,\xi)\geq e^{-2\delta T} \text{ for all } (s,\xi) \in \bigcup\limits_{1\leq \widetilde{t}\leq T}(\{\widetilde{t}\}\times \widetilde{t}\Psi(\delta)) \ ,$$
which concludes this Lemma.
\end{proof}

\begin{lemma}\label{Lm:LowerBoundBehaviorSolutionInThePositiveBranch}
For any compact set $K\subset \Lambda_0$ we have
\begin{equation}\label{Lm:LowerBoundBehaviorSolutionInThePositiveBranch:Eq:LowerBound}
\liminf\limits_{t\rightarrow\infty}\dfrac{1}{t}\ln \inf\limits_{c\in K}u(t,ct)\geq -\max\limits_{c\in K}\left[cI\left(\dfrac{1}{c}\right)-\beta\right] \ .
\end{equation} 
\end{lemma}

\begin{proof} 
The compactness of $K$ implies that it suffices to show that given $\varepsilon>0$ and any $c>0$ for which $cI\left(\dfrac{1}{c}\right)-\beta>0$, we have

\begin{equation}\label{Lm:LowerBoundBehaviorSolutionInThePositiveBranch:Eq:ReductionTodeltaNeighborhood}
\liminf\limits_{t\rightarrow\infty} \left(\dfrac{1}{t}\ln \inf\limits_{\widetilde{c}\in B_\delta(c)}u(t,\widetilde{c}t)\right)\geq \beta - cI\left(\dfrac{1}{c}\right)-\varepsilon \ ,
\end{equation}
for $\delta>0$ sufficiently small. We define the limit on the left-hand of (\ref{Lm:LowerBoundBehaviorSolutionInThePositiveBranch:Eq:ReductionTodeltaNeighborhood}) as 
\begin{equation}\label{Lm:LowerBoundBehaviorSolutionInThePositiveBranch:Eq:ReductionTodeltaNeighborhood:LHS:q}
q=\liminf\limits_{t\rightarrow\infty} \left(\dfrac{1}{t}\ln \inf\limits_{\widetilde{c}\in B_\delta(c)}u(t,\widetilde{c}t)\right) \ .
\end{equation}

\noindent By Lemma \ref{Lm:ExpoentialSmall ProbabilityOfHittingTimeFromZeroToLinearDistance} we know that $q>-\infty$. If $q=+\infty$, (\ref{Lm:LowerBoundBehaviorSolutionInThePositiveBranch:Eq:ReductionTodeltaNeighborhood}) is automatically satisfied. So without loss of generality we assume for the moment that $q$ is finite. By the representation (\ref{Prop:AsymptoticSolutionOne:Eq:Feynmann-KacStoppingTime}) we have for any $\kappa\in (0,1]$ that

\begin{equation}\label{Lm:LowerBoundBehaviorSolutionInThePositiveBranch:Eq:LowerBoundvUsingkappa}
\inf\limits_{\widetilde{c}\in B_\delta(c)} u(t, \widetilde{c}t)\geq \inf\limits_{\widetilde{c}\in B_\delta(c)}E^W \left[u(t-\kappa t, X^{\widetilde{c}t}(\kappa t))\exp\left\{ \int_0^{\kappa t}c(u(t-s, X^{\widetilde{c}t}(s))){\rm d}s\right\}\cdot \mathbf{1}_A\right]
\end{equation}
for some $P^W$-adapted set $A$. We pick small $h>0$ and choose $A$ to be the set of paths satisfying that for all $\widetilde{c}\in B_\delta(c)$ we have both
\begin{equation}\label{Lm:LowerBoundBehaviorSolutionInThePositiveBranch:Eq:SetACondition1:ProcessInsideNeighborboodScaledByOneMinusKappa}
X^{\widetilde{c}t}(\kappa t)\in B_{(1-\kappa)\delta t}((1-\kappa)tc)
\end{equation}
and
\begin{equation}\label{Lm:LowerBoundBehaviorSolutionInThePositiveBranch:Eq:SetACondition2:SolutionPDEControlledByh}
c(u(t-s, X^{\widetilde{c} t}(s)))\geq \beta (1-h) \text{ for all } s\in [0, \kappa t] \ .
\end{equation}
Then 
$$\begin{array}{l}
\displaystyle{\inf\limits_{\widetilde{c}\in B_\delta(c)}E^W \left[u(t-\kappa t, X^{\widetilde{c}t}(\kappa t))\exp\left\{ \int_0^{\kappa t}c(u(t-s, X^{\widetilde{c}t}(s))){\rm d}s\right\}\cdot \mathbf{1}_A\right]
}\\
\qquad \geq \displaystyle{\inf\limits_{\widetilde{c}\in B_\delta(c)} u((1-\kappa)t, \widetilde{c}(1-\kappa)t)\cdot e^{\beta(1-h)\kappa t}\cdot \inf\limits_{\widetilde{c}\in B_\delta(c)}P^W(A) \ ,}
\end{array}$$
which gives
$$\begin{array}{l}
\dfrac{1}{t}\ln \inf\limits_{\widetilde{c}\in B_\delta(c)}u(t, \widetilde{c}t)
\\
\ \geq (1-\kappa)\dfrac{1}{(1-\kappa)t}\ln \inf\limits_{\widetilde{c}\in B_\delta(c)}u((1-\kappa)t, \widetilde{c}(1-\kappa)t)+\kappa\beta(1-h)+\dfrac{1}{t}\ln\inf\limits_{\widetilde{c}\in B_\delta(c)}P^W(A) \ .
\end{array}$$
Taking $t\rightarrow\infty$ in the above inequality gives 
$$q\geq (1-\kappa)q + \kappa \beta (1-h) + \liminf\limits_{t\rightarrow\infty}\dfrac{1}{t}\ln \inf\limits_{\widetilde{c}\in B_\delta(c)}P^W(A) \ ,$$
which is
\begin{equation}\label{Lm:LowerBoundBehaviorSolutionInThePositiveBranch:Eq:LowerBoundViaProbabilityOfA}
q\geq \beta(1-h)+\liminf\limits_{t\rightarrow\infty}\dfrac{1}{\kappa t}\ln\inf\limits_{\widetilde{c}\in B_\delta(c)} P^W(A) \ .
\end{equation}

\noindent By Proposition \ref{Prop:AsymptoticSolutionZero} and the FKPP property $\beta=c(0)=\sup\limits_{u>0}c(u)$ we know that there is a $\delta>0$ sufficiently small so that for any $h\in (0,1)$ there is a constant $t_0>0$ depending on $h$ such that
$$c(u(t,c't))\geq \beta(1-h) \text{ for all } c'\in B_{6\delta}(c) \text{ and all } t\geq t_0 \ .$$

\noindent Now if $0<\kappa<\dfrac{1}{2}$ and for any $\widetilde{c}\in B_\delta(c)$ we have
\begin{equation}\label{Lm:LowerBoundBehaviorSolutionInThePositiveBranch:Eq:DeviationEstimateThatLeadToSetACondition2}
\sup\limits_{s\in [0, \kappa t]}|X^{\widetilde{c}t}(s)-(t-s)c|\leq 3\delta t \ ,
\end{equation}
then when $t\geq 2t_0$ we have $t-s\geq t_0$, $\dfrac{t}{t-s}\leq 2$ and
$$(t-s)(c-6\delta)\leq (t-s)\left[c-3\delta\dfrac{t}{t-s}\right]\leq X^{\widetilde{c}t}(s)\leq(t-s)\left[c+3\delta\dfrac{t}{t-s}\right]\leq (t-s)(c+6\delta) \ ,$$
which means (\ref{Lm:LowerBoundBehaviorSolutionInThePositiveBranch:Eq:SetACondition2:SolutionPDEControlledByh}) is achieved along such paths when $t>2t_0$.

Lemma \ref{Lm:TechnicalBoundOfcAndHatcInBall} ensures that for sufficiently small $\kappa_0>0$ and any $0<\kappa<\kappa_0$, for each $\widetilde{c}\in B_\delta(c)$ there is a $\hat{c}\in B_{2\delta}(c)$ such that (\ref{Lm:LowerBoundBehaviorSolutionInThePositiveBranch:Eq:SetACondition1:ProcessInsideNeighborboodScaledByOneMinusKappa}) is achieved whenever

\begin{equation}\label{Lm:LowerBoundBehaviorSolutionInThePositiveBranch:Eq:DeviationBoundThatEnsuresSetACondition1}
\dfrac{\widetilde{c} t - X^{\widetilde{c} t}(\kappa t)}{\kappa t} \in B_\delta(\hat{c}) \ .
\end{equation}

\noindent By (\ref{Lm:LowerBoundBehaviorSolutionInThePositiveBranch:Eq:DeviationEstimateThatLeadToSetACondition2}) and (\ref{Lm:LowerBoundBehaviorSolutionInThePositiveBranch:Eq:DeviationBoundThatEnsuresSetACondition1}),
we can estimate
\begin{equation}\label{Lm:LowerBoundBehaviorSolutionInThePositiveBranch:Eq:LowerBoundProbabilityOfSetAViaDeviationBoundAndEstimate}
\begin{array}{l}
\inf\limits_{\widetilde{c}\in B_\delta(c)}P^W(A)
\\
\geq \inf\limits_{\hat{c}\in B_{2\delta}(c), \widetilde{c}\in B_\delta(c)}P^W\left(
\sup\limits_{s\in [0, \kappa t]}|X^{\hat{c}t}(s)-(t-s)c|\leq 3\delta t \text{ and } \dfrac{\widetilde{c}t-X^{\widetilde{c}t}(\kappa t)}{\kappa t}\in B_\delta(\hat{c})\right) \ .
\end{array}
\end{equation}

\noindent For $\kappa\in \left(0, \dfrac{2\delta}{3\max(1,c)}\right)$, we see that
\begin{align*}
\sup\limits_{\hat{c}\in B_{2\delta}(c)}P^W\left(\sup\limits_{s\in [0, \kappa t]}|X^{\hat{c}t}(s)-(t-s)c|>3\delta t\right)
& \leq \sup\limits_{\hat{c}\in B_{2\delta}(c)}P^W\left(\sup\limits_{s\in[0,\kappa t]}|X^{\hat{c}t}(s)-\hat{c} t|>\dfrac{\delta t}{3}\right)
\cr
& \leq \sup\limits_{\hat{c}\in B_{2\delta}(c)}P^W\left(T^{\hat{c}t}_{(\hat{c}-\delta/3)t}\wedge T^{\hat{c}t}_{(\hat{c}+\delta/3)t}<\kappa t\right) \ .
\end{align*}

\noindent We apply Lemma \ref{Lm:SuperExponentialSmallProbabilityOfHittingTimeAtLinearDistance}, so that for any $M>0$ there exist some $\widetilde{\kappa}_0$ such that for any $0<\kappa<\widetilde{\kappa}_0$ we have
$$\limsup\limits_{t\rightarrow\infty} \dfrac{1}{\kappa t}\ln \sup\limits_{\hat{c}\in B_{2\delta}(c)}P^W\left(T^{\hat{c}t}_{(\hat{c}-\delta/3)t}\wedge T^{\hat{c}t}_{(\hat{c}+\delta/3)t}<\kappa t\right)\leq -M  \ .$$ 

\noindent Combining the above two estimates we get
\begin{equation}\label{Lm:LowerBoundBehaviorSolutionInThePositiveBranch:Eq:DeviationEstimateOneStepToFinal}
\limsup\limits_{t\rightarrow\infty}\dfrac{1}{\kappa t}\ln\sup\limits_{\hat{c}\in B_{2\delta}(c)}P^W\left(\sup\limits_{s\in [0, \kappa t]}|X^{\hat{c}t}(s)-(t-s)c|>3\delta t\right)\leq -M \ . 
\end{equation}

\noindent Combining (\ref{Lm:LowerBoundBehaviorSolutionInThePositiveBranch:Eq:LowerBoundViaProbabilityOfA}), (\ref{Lm:LowerBoundBehaviorSolutionInThePositiveBranch:Eq:LowerBoundProbabilityOfSetAViaDeviationBoundAndEstimate}), (\ref{Lm:LowerBoundBehaviorSolutionInThePositiveBranch:Eq:DeviationEstimateOneStepToFinal}) we get, for $0<\kappa<\min\left(\dfrac{1}{2}, \kappa_0, \dfrac{2\delta}{3\max(1,c)}, \widetilde{\kappa}_0\right)$, that
\begin{equation}\label{Lm:LowerBoundBehaviorSolutionInThePositiveBranch:Eq:FinalEstimateLeadingToResult}
q\geq \beta(1-h)+\liminf\limits_{t\rightarrow\infty} \dfrac{1}{\kappa t} \inf\limits_{\hat{c}\in B_{2\delta}(c), \widetilde{c}\in B_{\delta}(c)} P^W\left(\dfrac{\widetilde{c} t - X^{\widetilde{c} t}(\kappa t)}{\kappa t}\in B_\delta(\hat{c})\right) \ .
\end{equation}
We apply the LDP \emph{lower bound} (\ref{Thm:LDPProcessX:Eq:LowerBoundPositive}), and setting $h>0$ and $\delta>0$ sufficiently small, to get that (\ref{Lm:LowerBoundBehaviorSolutionInThePositiveBranch:Eq:FinalEstimateLeadingToResult}) implies (\ref{Lm:LowerBoundBehaviorSolutionInThePositiveBranch:Eq:LowerBound}).
\end{proof}

\begin{lemma}\label{Lm:TechnicalBoundOfcAndHatcInBall}
Let $\widetilde{c}\in B_\delta(c)$. Define $\hat{c}=c+2(\widetilde{c}-c)$. Then for any $\Delta\in (-\delta, \delta)$, there exists small $\kappa_0>0$, for any $0<\kappa<\kappa_0$, we have
\begin{equation}\label{Lm:TechnicalBoundOfcAndHatcInBall:Eq:BallInclusion}
\widetilde{c}t-\kappa t \hat{c}+\kappa t\Delta\in B_{(1-\kappa)\delta t}((1-\kappa)ct) \ .
\end{equation}
\end{lemma}

\begin{proof}
Write $\widetilde{c}=c+\Delta_1$ with $-\delta<\Delta_1<\delta$. Then $\hat{c}=c+2\Delta_1$. Pick $\kappa_0=\dfrac{1}{4}-\dfrac{|\widetilde{c}-c|}{4\delta}=\dfrac{1}{4}-\dfrac{|\Delta_1|}{4\delta}>0$ and any $0<\kappa<\kappa_0$. Since $-3\delta<2\Delta_1-\Delta<3\delta$, we have $-2\delta<\delta-(2\Delta_1-\Delta)<4\delta$ and $-2\delta<\delta+(2\Delta_1-\Delta)<4\delta$. So that  $$\kappa(\delta-(2\Delta_1-\Delta))<\delta-|\Delta_1|\leq \delta - \Delta_1 \ ,$$
and
$$\kappa(\delta+(2\Delta_1-\Delta))<\delta-|\Delta_1|\leq \delta + \Delta_1 \ .$$
This gives
\begin{equation}\label{Lm:TechnicalBoundOfcAndHatcInBall:Eq:BallInclusion:Eq:Step1}
-(1-\kappa)\delta<\Delta_1-\kappa(2\Delta_1-\Delta)<(1-\kappa)\delta \ .
\end{equation}
We then calculate
\begin{equation}\label{Lm:TechnicalBoundOfcAndHatcInBall:Eq:BallInclusion:Eq:Step2}
\begin{array}{ll}
& (\widetilde{c}t-\kappa t\hat{c}+\kappa t \Delta)-(1-\kappa)ct
\\
= & [(c+\Delta_1)t-\kappa t (c+2\Delta_1)+\kappa t \Delta]-(1-\kappa)ct 
\\
= & t[\Delta_1-\kappa(2\Delta_1-\Delta)] \ .
\end{array}
\end{equation}
Combining (\ref{Lm:TechnicalBoundOfcAndHatcInBall:Eq:BallInclusion:Eq:Step1}) and (\ref{Lm:TechnicalBoundOfcAndHatcInBall:Eq:BallInclusion:Eq:Step2}) we get (\ref{Lm:TechnicalBoundOfcAndHatcInBall:Eq:BallInclusion}).
\end{proof}

\section{Exact shape of the asymptotic wave front}\label{Sec:Shape}
In this section, by combining Theorem \ref{Thm:WavePropagation} with the properties of $\mu, \mu^\fwd, I, I^\fwd$ in Lemmas \ref{Lm:PropertiesMuAndTruncatedMu}, \ref{Lm:PropertiesMuFwd}, \ref{Lm:PropertiesOfIInOurCase}, \ref{Lm:PropertiesOfIInOurCase:Forward}, we present the exact possible shapes of the asymptotic wave front.

We recall the critical value $\eta_c$ defined in (\ref{Eq:CriticalEta:Backward}). The following Lemma investigates possible configurations of $I(a)$ and $I^\fwd(a)$.

\begin{lemma}\label{Lm:FurtherPropertiesOfMuAndIUsedInWaveFrontShape}
Assume $\mathbf{E}[b]>0$. Then if $\eta_c>0$, then we have
\begin{itemize}
\item[\emph{(a1)}] $I(a)$ is convex and monotonically decreasing on $(0, \mu'(0))$ and monotonically increasing on $(\mu'(0), +\infty)$, with its minimum $I(\mu'(0))=-\mu(0)>0$ and asymptotically $\lim\limits_{a\rightarrow\infty}[I(a)-(\eta_c a - \mu(\eta_c-))]=0$ from the above \emph{(}i.e. $I(a)\geq \eta_c a -\mu(\eta_c-)$ for all $a$, see \emph{Figure \ref{Fig:Ifwdbwd_case_ab} (a) (b))}. The intercept of the asymptotic line $\eta_c a - \mu(\eta_c-)$ is positive, i.e., $\mu(\eta_c-)<0$.

\item[\emph{(a2)}] $I^\fwd(a)= I(a)-2\mathbf{E}[b]$ is convex and monotonically decreasing on $(0, (\mu^\fwd)'(0))$ and monotonically increasing on $((\mu^\fwd)'(0), +\infty)$, with its minimum $I((\mu^\fwd)'(0))=0$ and asymptotically $\lim\limits_{a\rightarrow\infty}[I(a)-(\eta_c a - \mu^\fwd(\eta_c-))]=0$ from the above \emph{(}i.e. $I^\fwd(a)\geq \eta_c a -\mu^\fwd(\eta_c-)$ for all $a$, see \emph{Figure \ref{Fig:Ifwdbwd_case_ab} (a) (b)}. The intercept of the asymptotic line $\eta_c a - \mu^\fwd(\eta_c-)$ is negative, i.e., $\mu^\fwd(\eta_c-)>0$.
\end{itemize}

If $\eta_c=0$, then we have
\begin{itemize}
\item[\emph{(b1)}] $I(a)$ is convex and monotonically decreasing on $(0, +\infty)$ and asymptotically $\lim\limits_{a\rightarrow\infty}I(a)=-\mu(0)>0$ from the above \emph{(}i.e. $I(a)\geq -\mu(0)$ for all $a$, see \emph{Figure \ref{Fig:Ifwdbwd_case_cd} (c) (d))}.

\item[\emph{(b2)}] $I^\fwd(a)=I(a)-2\mathbf{E}[b]$ is convex and monotonically decreasing on $(0, +\infty)$ and asymptotically $\lim\limits_{a\rightarrow\infty}I(a)=0$ from the above \emph{(}i.e. $I(a)\geq 0$ for all $a$, see \emph{Figure \ref{Fig:Ifwdbwd_case_cd} (c) (d))}.
\end{itemize}
\end{lemma}

\begin{figure}[H]
\centering
\includegraphics[height=5.8cm, width=13cm]{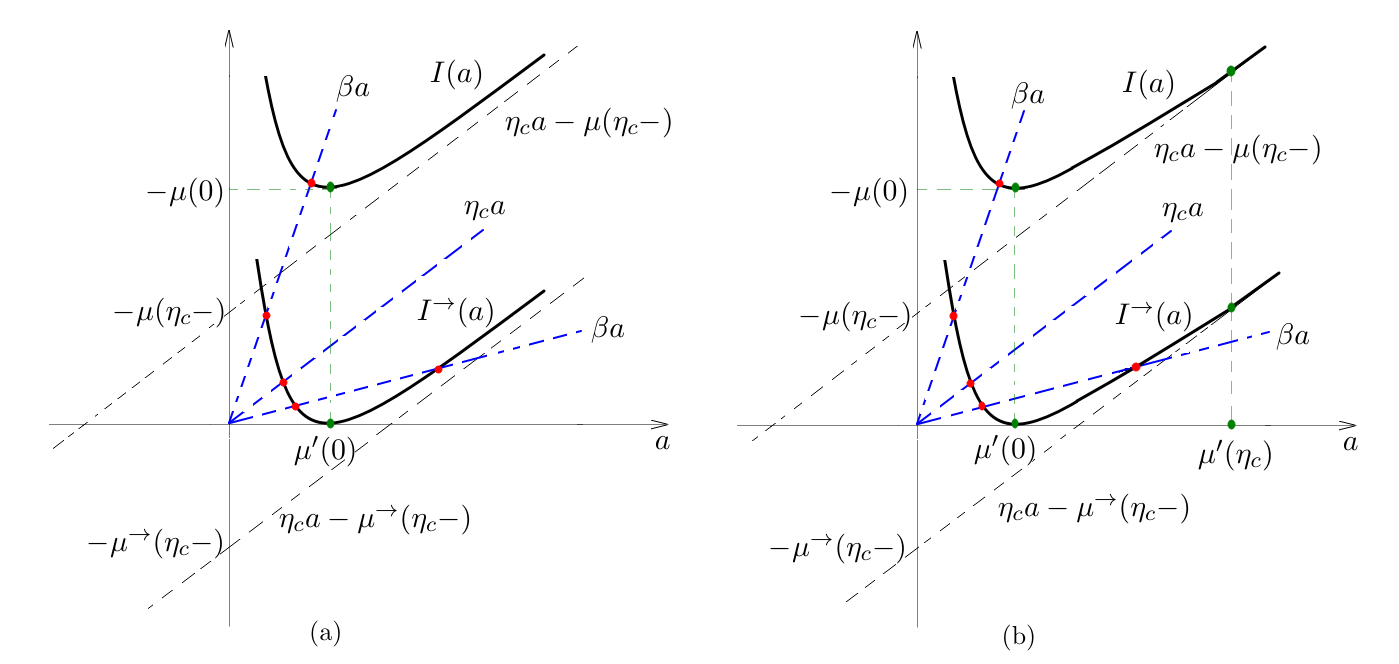}
\caption{Shape of $I(a)\equiv I^\bwd(a)$ and $I^\fwd(a)$. (a) when $\eta_c>0$ and $\mu'(\eta_c-)=+\infty$; (b) when $\eta_c>0$ and $\mu'(\eta_c)<+\infty$.}
\label{Fig:Ifwdbwd_case_ab}
\end{figure}

\begin{figure}[H]
\centering
\includegraphics[height=5.8cm, width=13cm]{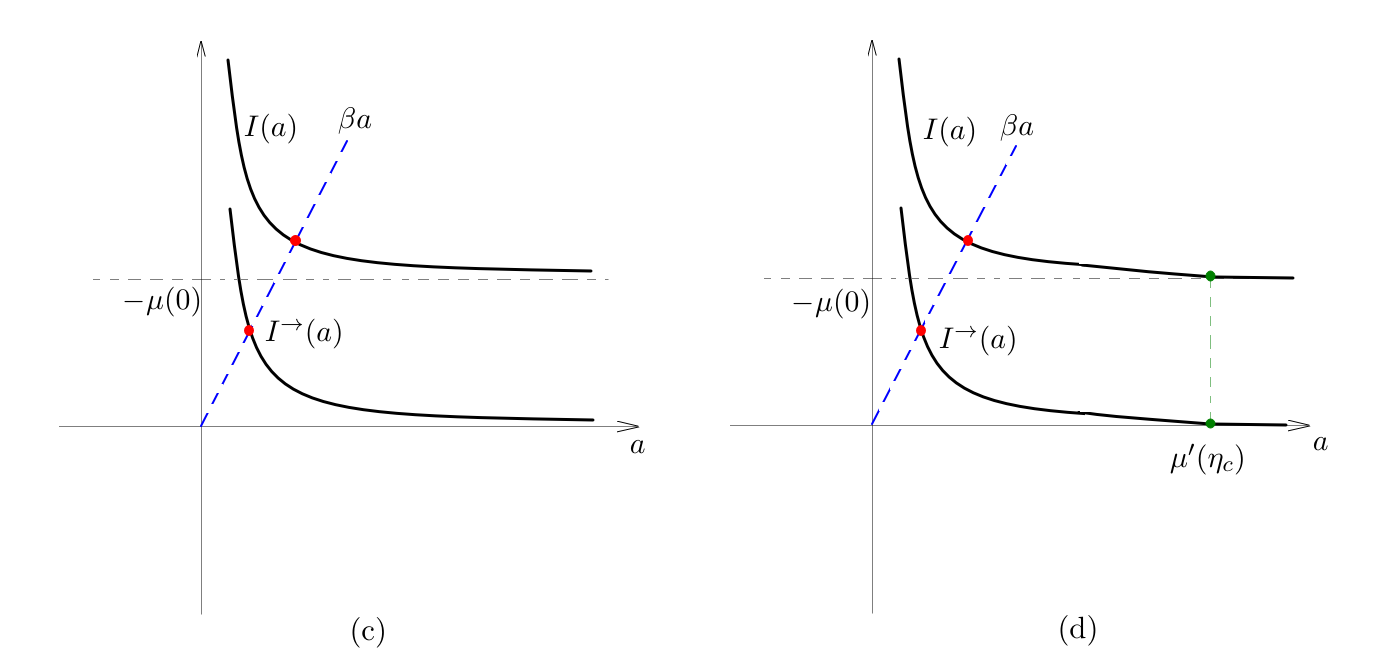}
\caption{Shape of $I(a)\equiv I^\bwd(a)$ and $I^\fwd(a)$. (c) when $\eta_c=0$ and $\mu'(\eta_c-)=+\infty$; (b) when $\eta_c=0$ and $\mu'(\eta_c)<+\infty$.}
\label{Fig:Ifwdbwd_case_cd}
\end{figure}

\begin{proof} Recall that by (\ref{Eq:SDE-RandomDrift}) we have
$$X^x(t)=x+\int_0^t b(X^x(s)){\rm d}s+W_t \ ,
$$
where $W_t$ is a standard Brownian motion, with given filtration $(\Omega^W, \mathcal{F}_t^W, P^W)$. Following (\ref{Eq:HittingTimeSDE-RandomDrift:Backward}), (\ref{Eq:HittingTimeSDE-RandomDrift:Forward}), we set
\[
T_0^1=\inf\{t > 0, X^1(t) \leq 0\} \ , \ 
T^1_0=\inf\{t>0, X^0(t) \geq 1\} \ .
\]

\noindent Recall the Lyapunov functions defined in (\ref{Eq:LyapunovFunctionHittingTime:Backward}), (\ref{Eq:LyapunovFunctionHittingTime:Forward}) as 
$$
\mu(\eta)=\mathbf{E}\left[\ln E^W\left(
e^{\eta T^1_0}\mathbf{1}_{\{T^1_0<+\infty\}}\right)\right] \ , \ 
\mu^\fwd(\eta)=\mathbf{E}\left[\ln E^W\left(
e^{\eta T^0_1}\mathbf{1}_{\{T^0_1<+\infty\}}\right)\right] \ ,
$$
and the critical values defined in (\ref{Eq:CriticalEta:Backward}), (\ref{Eq:CriticalEta:Forward}) as 
\[
\eta_c=\sup\{\eta\geq0: \mu(\eta)<+\infty\} \ , \ \eta_c^\fwd=\sup\{\eta\geq 0: \mu^\fwd(\eta)<+\infty\} \ .
\]

In Lemma \ref{Lm:BwdAndFwdCriticalEtaForMuAreTheSame} we have shown that $\eta_c=\eta_c^\fwd$. By the above definition it is also clear that $\eta_c\geq 0$. 

Let us first assume $\eta_c>0$. As we already discussed in Section 3.2, in this case of ${\mathbf E}[b]>0$, one always has 
  ${\mathbf P}(P^W(T_0^1<\infty)<1)=1$
 and then by the same reasoning as \eqref{eqn: disrho0a}, one has
${\mathbf E}[\ln P^W(T_0^1<\infty)]=-2{\mathbf E}[b].$ Therefore $\mu(0)=\mathbf{E} [\ln P^W(T^1_0<\infty)]=-2\mathbf{E}[b]<0$. Since $\mu^\fwd(0)=\mu(0)+2\mathbf{E}[b]$ from Lemma \ref{Lm:ExpectationLnRhoIsConstant}, we get $\mu^\fwd(0)=\mathbf{E} [\ln P^W(T^0_1<\infty)]=0$. By Lemma \ref{Lm:BwdAndFwdCriticalEtaForMuAreTheSame} we also see that $\mu(\eta_c-)<0$ and $\mu^\fwd(\eta_c-)>0$. Based on these facts, according to Lemma \ref{Lm:PropertiesOfIInOurCase}, we justified (a1) and (a2).

Now let us assume that $\eta_c=0$. In this case we have \[\inf_{a\in {\mathbb R}}I(a)=\mu(0)=-{\mathbf P}[\ln P^W(T_0^1<\infty) ]=-2\mathbf{E}[b]>0.\]
and so we justified properties (b1) and (b2).
\end{proof}

The next Lemma characterize the exact shapes of the sets $\Lambda_0,  \Lambda_1$ and $\Lambda^\fwd_0, \Lambda^\fwd_1$ that we introduced in (\ref{Eq:PositiveAndNegativeDifferenceOfIandbetaWaveSpeed}) and (\ref{Eq:PositiveAndNegativeDifferenceOfIandbetaWaveSpeed:Forward}).

\begin{lemma}[Configuration of $\Lambda_0$, $\Lambda_1$, $\Lambda_0^\fwd$, $\Lambda_1^\fwd$ under given $\beta$]\label{Lm:SolutionStructureBalanceEquation}
Let the function $\mu(\eta)$ be defined in \emph{(\ref{Eq:LyapunovFunctionHittingTime:Backward})}. Then we have
\begin{itemize}
\item[\emph{(1)}] When $\beta\in (0, \eta_c)$, $\Lambda_0 = (0, +\infty)$, $\Lambda_1 = \emptyset$, $\Lambda_0^\fwd = (0, -c_2^*)\cup (c_1^*, +\infty)$,  $\Lambda_1^\fwd = (-c_2^*, c_1^*)$ with $c_1^*>0, c_2^*<0, c_1^*+c_2^*>0$;
\item[\emph{(2)}] When $\beta=\eta_c$, $\Lambda_0=(0, +\infty)$, $\Lambda_1=\emptyset$, $\Lambda_0^\fwd = (c_1^*, +\infty)$, $\Lambda_1^\fwd=(0, c_1^*)$ with $c_1^*>0$;
\item[\emph{(3)}] When $\beta\in (\eta_c, +\infty)$, $\Lambda_0 = (c_2^*, +\infty)$, $\Lambda_1 = (0, c_2^*)$, $\Lambda_0^\fwd = (c_1^*, +\infty)$,  $\Lambda_1^\fwd = (0, c_1^*)$ with $c_1^*>0, c_2^*>0, c_1^*-c_2^*>0$;
\end{itemize}
\end{lemma}

\begin{proof}
Equations (\ref{Eq:IntersectionOfIandbetaWaveSpeed}), (\ref{Eq:IntersectionOfIandbetaWaveSpeed:Forward}) are equivalent to $I\left(\dfrac{1}{c}\right)=\dfrac{\beta}{c}$, $I^\fwd\left(\dfrac{1}{c}\right)=\dfrac{\beta}{c}$. Let $\dfrac{1}{c}=a$, so that we are solving $I(a)=\beta a, I^\fwd(a)=\beta a$ to determine $\Lambda_0, \Lambda_1, \Lambda_0^\fwd, \Lambda_1^\fwd$. The results of this Lemma then follows from our previous description of the configurations of $I(a)$ and $I^\fwd(a)$ in Lemma \ref{Lm:FurtherPropertiesOfMuAndIUsedInWaveFrontShape}. In fact, from Figures \ref{Fig:Ifwdbwd_case_ab}, \ref{Fig:Ifwdbwd_case_cd} we see that the red dots are intersection points of $I(a)$ with the line $\beta a$ for different $\beta$. Each intersection point $a^*$ correspond to a wave speed $c^*=\dfrac{1}{a^*}$. Assume $\eta_c>0$, then we have the following cases

\begin{itemize} 
\item[(1)] When $0<\beta<\eta_c$, the two intersection points are both lying on $I^\fwd(a)$. They are $((a^\fwd)^*_1, I^\fwd((a^\fwd)^*_1))$, $((a^\fwd)^*_2, I^\fwd((a^\fwd)^*_2))$ with $(a^\fwd)^*_1<(a^\fwd)^*_2$. So correspondingly $\Lambda_0^\fwd=(0, -c_2^*)\cup (c_1^*, +\infty) \ , \Lambda_1^\fwd=(-c_2^*, c_1^*)$, $c_1^*>0, c_2^*<0$ and $-c_2^*=\dfrac{1}{(a^\fwd)^*_2}$, $c_1^*=\dfrac{1}{(a^\fwd)^*_1}$, $c_1^*>-c_2^*$. Since there is no intersection point on $I(a)$, it is also easy to see that the sets $\Lambda_0=(0, +\infty), \Lambda_1=\emptyset$; 

\item[(2)] When $\beta=\eta_c$, there will be a unique intersecting point $((a^\fwd)^*, I^\fwd((a^\fwd)^*)$ lying on $I^\fwd(a)$, and we set $c_1^*=\dfrac{1}{(a^\fwd)^*}>0$. This gives $\Lambda_0^\fwd=(c_1^*, +\infty), \Lambda_1^\fwd=(0, c_1^*)$. Since $-\mu(\eta_c-)>0$, there is no intersection on $I(a)$ and thus $\Lambda_0=(0, +\infty), \Lambda_1=\emptyset$; 

\item[(3)] When $\beta>\eta_c$, there will be one intersecting point $(a^*, I(a^*))$ on $I(a)$ and another intersecting point $((a^\fwd)^*, I((a^\fwd)^*))$ on $I^\fwd(a)$, and due to the fact that $I(a)>I^\fwd(a)$ we also have $a^*>(a^\fwd)^*$, indicating that $c_2^*=\dfrac{1}{a^*}<\dfrac{1}{(a^\fwd)^*}<c_1^*$ and $\Lambda_0=(c_2^*, +\infty), \Lambda_1=(0, c_2^*), \Lambda_0^\fwd=(c_1^*, +\infty), \Lambda_1^\fwd=(0, c_1^*)$.
\end{itemize}
Now we assume $\eta_c=0$, then we have only the above case (3) happens for any $\beta\in (0, +\infty)$.
\end{proof}

Combining Lemma \ref{Lm:SolutionStructureBalanceEquation} with Theorem \ref{Thm:WavePropagation}, noting that $R_0=\Lambda_0\cup (-\Lambda_0^\fwd), R_1=\Lambda_1\cup (-\Lambda_1^\fwd)$, we obtain the following result describing the exact shape of the asymptotic wave corresponding to Figure \ref{Fig:TwoWaves}.

\begin{theorem}[Exact shape of the asymptotic wave]\label{Thm:ShapeAsymptoticWave}
The exact shapes of the asymptotic wave are given by \emph{Table \ref{Thm:ShapeAsymptoticWave}} that we demonstrated in \emph{Section \ref{Sec:BasicIdeaMainResults}}.
\end{theorem}

\bibliographystyle{plain}
\bibliography{bibliography_RDE-drift}

\end{document}